\documentclass[12pt]{amsart}
\usepackage{latexsym,amsmath,amsfonts,amscd,amssymb}

\usepackage{mathrsfs,lmodern}
\usepackage[T1]{fontenc}

\title[Universal Hitchin moduli spaces] 
{Universal Hitchin moduli spaces}
\author[\'Alvarez-C\'onsul]{Luis \'Alvarez-C\'onsul}
\address{Instituto de Ciencias Matem\'aticas (CSIC-UAM-UC3M-UCM)\\ Nicol\'as Cabrera 13--15, Cantoblanco\\ 28049 Madrid, Spain}
\email{l.alvarez-consul@icmat.es}

\author[Garcia-Fernandez]{Mario Garcia-Fernandez}
\address{Instituto de Ciencias Matem\'aticas (CSIC-UAM-UC3M-UCM)\\ Nicol\'as Cabrera 13--15, Cantoblanco\\ 28049 Madrid, Spain}
\email{mario.garcia@icmat.es}

\author[Garc\'{\i}a-Prada]{Oscar Garc\'{\i}a-Prada}
\address{Instituto de Ciencias Matem\'aticas (CSIC-UAM-UC3M-UCM)\\ Nicol\'as Cabrera 13--15, Cantoblanco\\ 28049 Madrid, Spain}
\email{oscar.garcia-prada@icmat.es}

\author[Trautwein]{Samuel Trautwein}
\address{Formerly at ETH Z\"urich, Switzerland}
\email{sam87trautwein@gmail.com}

\usepackage{stmaryrd}

\usepackage{calligra}

\usepackage{lscape}
\usepackage{array}
\usepackage[mathscr]{eucal} 
\usepackage{graphicx}
\usepackage{tikz-cd} 

\DeclareFontFamily{OT1}{rsfs}{}
\DeclareFontShape{OT1}{rsfs}{n}{it}{<->rsfs10}{}
\DeclareMathAlphabet{\curly}{OT1}{rsfs}{n}{it}

\theoremstyle{plain}  % default
\newtheorem{theorem}{Theorem}[section]
\newtheorem*{theorem*}{Theorem}

\newtheorem{corollary}[theorem]{Corollary}
\newtheorem{lemma}[theorem]{Lemma}
\newtheorem{proposition}[theorem]{Proposition}

\theoremstyle{definition}
\newtheorem{definition}[theorem]{Definition}
\theoremstyle{remark}
\newtheorem*{acknowledgements}{Acknowledgements}

\newtheorem{remark}[theorem]{Remark}

\newtheorem*{claim*}{Claim}
\numberwithin{equation}{section}
\newcommand{\suchthat}{\;\;:\;\;}

\renewcommand{\leq}{\leqslant}

\renewcommand{\geq}{\geqslant}

\newcommand{\C}{\mathbb{C}}

\newcommand{\GGG}{\curly{G}}
\newcommand{\HHH}{\curly{H}}

\newcommand{\cC}{\mathcal{C}}
\newcommand{\cM}{\mathcal{M}}
\newcommand{\calR}{\mathcal{R}}

\newcommand{\cS}{\mathcal{S}}
\newcommand{\cT}{\mathcal{T}}

\newcommand{\dbar}{\bar{\partial}}

\newcommand{\lra}{\longrightarrow}

\newcommand{\GL}{\mathrm{GL}}
\newcommand{\SL}{\mathrm{SL}}

\DeclareMathOperator{\ad}{ad}
\DeclareMathOperator{\Ad}{Ad}

\DeclareMathOperator{\tr}{tr}

\DeclareMathOperator{\Hom}{Hom}
\DeclareMathOperator{\End}{End}

\DeclareMathOperator{\Id}{Id}

\renewcommand{\phi}{\varphi}

\newcommand{\lieg}{\mathfrak{g}}

\newcommand{\liek}{\mathfrak{k}}

\newcommand{\glie}{\mathfrak{g}}

\newcommand{\RR}{\mathbb{R}}

\newcommand{\AAA}{{\curly A}}
\newcommand{\CCC}{{\curly C}}
\newcommand{\DDD}{{\curly D}}

\newcommand{\JJJ}{\curly{J}}

\newcommand{\KKK}{{\curly K}}

\newcommand{\XXX}{\curly{X}}

\newcommand{\Lie}{\operatorname{Lie}}

\newcommand{\Mg}{\mathcal{M}^{\operatorname{Hit}}}

\newcommand{\Mf}{\mathcal{M}^{\operatorname{Flat}}}

\newcommand{\Mh}{\mathcal{M}^{\operatorname{Harm}}}

\newcommand{\Uf}{\mathcal{U}^{\operatorname{Flat}}}

\newcommand{\Uh}{\mathcal{U}^{\operatorname{Harm}}}

\newcommand{\bmu}{{\boldsymbol{\mu}}}

\newcommand{\Real}{\mathop{{\fam0 Re}}\nolimits}
\newcommand{\Imaginary}{\mathop{{\fam0 Im}}\nolimits}

\begin{document}

\thanks{
The first three authors are partially supported by the Spanish Ministry of Science and Innovation, through the `Severo Ochoa Programme for Centres of Excellence in R\&D' (CEX2023-001347-S).
The first and the third authors are partially supported by MICINN under grant PID2022-141387NB-C21. The second author is partially supported by MICINN under grants PID2022-141387NB-C22 and CNS2022-135784.
}

\begin{abstract}
We study metric aspects of the universal moduli space of solutions to Hitchin's equations as the complex structure $J$ varies over the Teichmüller space $\mathcal{T}$ of a closed surface $\Sigma$. Our approach is gauge theoretical and builds on the theory of K\"ahler fibrations and the moment map interpretation of constant scalar curvature K\"ahler metrics. Our first main result establishes
that, over the moduli space of cscK metrics, the universal moduli space of solutions to Hitchin's equations carries a natural complex structure together with a family of pseudo-K\"ahler metrics forming a K\"ahler fibration with a K\"ahler Ehresmann connection.

We then investigate a second universal moduli space, constructed from the space of flat $G$-connections over $\mathcal{T}$, which admits a nontrivial $J$-dependent K\"ahler fibration structure discovered by Hitchin. Using symplectic reduction, we build universal moduli spaces of solutions to the harmonicity equations depending on a coupling constant $\alpha$, obtaining natural complex and pseudo-K\"ahler structures and an explicit K\"ahler potential. The main novelty here is that this moduli space is defined by a system coupling the scalar curvature with a cubic term in the Higgs field. Finally, we propose a conjectural relationship between the two resulting families of moduli spaces in the weak-coupling limit $\alpha\to 0$, inspired by the twistor geometry of Hitchin's hyperk\"ahler moduli space.
\end{abstract}
\maketitle

%\newpage

%\tableofcontents

%%%%%%%%%%%%%%%%%%%%%%%%
\section{Introduction}
\label{intro}
%%%%%%%%%%%%%%%%%%%%%%%%

The moduli space $\mathcal{R}$ of representations of the fundamental group of a closed oriented surface $\Sigma$ has an interesting interplay with classical Teichm\"uller theory. For the case of a compact Lie group $K$, a choice of complex structure $J$ on $\Sigma$ determines a homeomorphism of the moduli space $\mathcal{R}(K)$ of $K$-representations with a moduli space of holomorphic principal $G$-bundles over the compact Riemann surface $X_J = (\Sigma,J)$, via the Narasimhan–Seshadri theorem \cite{narasimhan-seshadri:1965} and its extension to principal bundles by Ramanathan \cite{ramanathan:1975,ramanathan:1996}. Here, $G$ denotes the complex reductive Lie group given by the complexification of $K$. This way, the \emph{character variety} $\mathcal{R}(K)$ determines a holomorphic fibration over the Teichm\"uller space $\mathcal{T}$ of $\Sigma$, given by complex structures on $\Sigma$ modulo the action of diffeomorphisms isotopic to the identity.

A different aspect of this interplay arose forty years ago in Nigel Hitchin's study of the self-duality equations on a Riemann surface $X_J$ \cite{hitchin:1987}
\begin{equation}\label{eq:hitchinintro}
\begin{split}
F_A -[\varphi,\tau(\varphi)]&= 0,\\
\dbar_{J,A}\varphi&=0,
\end{split}
\end{equation}
and subsequent work by Hitchin~\cite{hitchin:duke,hitchin:1992} and Simpson~\cite{simpson:1988,simpson:1992,simpson:1994,simpson:1995}. Here $A$ is a connection on a principal $K$-bundle over $\Sigma$ and $\varphi\in \Omega^{1,0}(X_J,E_K(\lieg))$ is the \emph{Higgs field}. The corresponding gauge-theoretical moduli space $\Mg(G)$ has three different incarnations, as a character variety $\mathcal{R}_c(G)$ for the  complex reductive Lie group $G$, as a moduli space $\Mf(G)$ of reductive flat $G$-connections, and as a moduli space of polystable $G$-Higgs bundles $\cM^{Higgs}(G)$. These three viewpoints on the moduli space of Hitchin's equations determine a hyperk\"ahler structure on $\Mg(G)$. Interestingly, only the moduli space $\cM^{Higgs}(G)$ depends on the choice of a complex structure on $\Sigma$. Furthermore, the twistor space $Z \to \mathbb{CP}^1$ for the hyperk\"ahler structure has generic complex structure $\Mf(G)$, while the special points over $\mathbb{CP}^1$ correspond to $\cM^{Higgs}(G)$ and its conjugate.

In this paper we investigate metric aspects of the universal moduli space of solutions of Hitchin's equations varying over the Teichm\"uller space $\mathcal{T}$. Our approach is gauge-theoretical in nature, and builds on the classical theory of K\"ahler fibrations \cite{GLS,mundet:2000} and on the symplectic interpretation of constant scalar curvature K\"ahler metrics \cite{Do2,Fj}. We fix a symplectic structure $\omega$ on $\Sigma$ and consider the coupled system of equations
\begin{equation}\label{eq:cHitchinintintro}
\begin{split}
F_A -[\varphi,\tau(\varphi)] & = 0,\\
\dbar_{J,A}\varphi & = 0,\\
S_g & = \frac{2\pi \chi(\Sigma)}{V},
\end{split}
\end{equation}
where $J$ is a complex structure on $\Sigma$ and $S_g$ is the scalar curvature of the metric $g = \omega(,J)$. Let $\mathcal{U}^{Hit}(G)$ be the moduli space of solutions of \eqref{eq:cHitchinintintro} modulo \emph{unitary gauge} (see Section \ref{section-Uhiggsmetric}), and   $\mathcal{U}^{Higgs}(G)$ be the universal moduli space of $G$-Higgs bundles (see Section \ref{ssec:UHiggs}).
Our first main result can be stated as follows (cf. Theorem \ref{thm:metricI}, and Theorem \ref{thm:existenceI}).

\begin{theorem}\label{thm:existenceIintro}
Let $X = (\Sigma,J)$ be a compact Riemann surface with genus $g(\Sigma) \geqslant 2$. Then, for any fixed total volume $V > 0$ and parameter $\varepsilon \in \{-1,1\}$, there exists $\alpha_0 > 0$ such that for any $0 < \alpha < \alpha_0$ there exists a non-empty open subset
$$
\mathcal{U}^*_{\alpha,\varepsilon} \subset \mathcal{U}^{Hit}(G),
$$
endowed with a complex structure $\mathbb{I}$ and a pre-symplectic structure $\boldsymbol{\omega}^{\mathbb{I}}_{\alpha,\varepsilon}$. Furthermore, the induced maps $\mathcal{U}^*_{\alpha,\varepsilon} \to \mathcal{U}^{Higgs}(G)$ and
\begin{equation}\label{eq:KFibModHiggsintro}
\mathcal{U}^*_{\alpha,\varepsilon} \lra \mu_{\HHH}^{-1}(0)/\HHH 
\end{equation}
are holomorphic, where $\mu_{\HHH}^{-1}(0)/\HHH$ is the moduli space of constant scalar curvature K\"ahler metrics on $\Sigma$ with total volume $V$.

Furthermore, one has

\begin{enumerate}

\item if $\varepsilon = 1$ the tensor $\mathbf{g}^\mathbb{I}_{\alpha,\varepsilon}$ is possibly degenerate,

\item if $\varepsilon = -1$ the tensor $\mathbf{g}^\mathbb{I}_{\alpha,\varepsilon}$ is non-degenerate, and defines a pseudo-K\"ahler structure on the moduli space.

\end{enumerate}
In either case, the restriction of $\mathbf{g}^{\mathbb{I}}_{\alpha,\varepsilon} = \boldsymbol{\omega}^{\mathbb{I}}_{\alpha,\varepsilon}(,\mathbb{I})$ to the fibres of \eqref{eq:KFibModHiggsintro} is $\alpha \mathbf{g}$, where $\mathbf{g}$ denotes the hyperk\"ahler metric on the moduli space of solutions of Hitchin's equations. Consequently, \eqref{eq:KFibModHiggsintro} has a natural structure of K\"ahler fibration with coupling form $\boldsymbol{\omega}^{\mathbb{I}}_{1,0}$ (see \eqref{eq:omegagmoduliexI}) and K\"ahler Ehresmann connection.
\end{theorem}

We comment briefly on the proof. The moduli space $\mathcal{U}^{Hit}(G)$  of solutions of \eqref{eq:cHitchinintintro} modulo \emph{gauge} is defined via symplectic reduction, and hence is naturally endowed with a family of (pre)-symplectic structures 
\begin{equation}\label{eq:uomegaIintro}
\boldsymbol{\omega}^{\mathbb{I}}_{\alpha,\varepsilon} = \varepsilon\boldsymbol{\omega}_\JJJ + \alpha \boldsymbol{\sigma}^{\mathbb{I}}.
\end{equation}
Here $\boldsymbol{\sigma}^{\mathbb{I}}$ is a closed $2$-form restricting on the fibres to Hitchin's symplectic structure $\omega_{\mathbf{I}}$ (see Proposition \ref{p:existencesigmaI}) and $\boldsymbol{\omega}_\JJJ$ is the pull-back of the symplectic structure on the moduli space $\mu_{\HHH}^{-1}(0)/\HHH$, induced by the choice of symplectic form $\omega$ on $\Sigma$ (see \eqref{eq:SympJ}). Even though the phase space of parameters $(J,A,\psi)$, with $\psi = -i(\varphi-\tau_h \varphi)$, has a natural complex structure $\mathbb{I}$, the symmetric tensor $\mathbf{g}^{\mathbb{I}}_{\alpha,\varepsilon} = \boldsymbol{\omega}^{\mathbb{I}}_{\alpha,\varepsilon}(,\mathbb{I})$ is not positive definite (see Lemma \ref{lemma:indef}). This follows from an explicit formula (see Corollary \ref{cor:omegaexpI})
\begin{equation}\label{eq:omegagmoduliexI}
\begin{split}
\mathbf{g}^{\mathbb{I}}_{\alpha,\varepsilon} & = \frac{\varepsilon}{2}\int_{\Sigma}\tr(\dot J\dot J) \omega \\
& + \alpha  \int_\Sigma B(a \wedge J a) + \alpha \int_\Sigma B\left(\left(\dot \psi + \tfrac{\alpha}{2}\psi(J \dot J)\right) \wedge J\left(\dot \psi + \tfrac{\alpha}{2}\psi(J \dot J)\right)\right)\\
& - \frac{\alpha}{4}\int_\Sigma B\left(\psi(\dot J) \wedge J \psi(\dot J)\right),
\end{split}
\end{equation}
where $B$ is a positive-definite invariant metric on $\mathfrak{k} = \Lie K$. Hence, it is not obvious  a priori that $\mathcal{U}^{Hit}(G)$ inherits a complex structure compatible with $\boldsymbol{\omega}^{\mathbb{I}}_{\alpha,\varepsilon}$. The key step is to undertake a \emph{gauge fixing} for solutions of the coupled Hitchin equations \eqref{eq:cHitchinintintro}, whereby the complex structure $\mathbb{I}$ and the symmetric tensor $\mathbf{g}^{\mathbb{I}}_{\alpha,\varepsilon}$ descend to the moduli space. Difficulties will arise, due to the fact that $\mathbf{g}^{\mathbb{I}}_{\alpha,\varepsilon}$ is neither a definite pairing nor non-degenerate.

Our second main result is concerned with a universal moduli space of flat $G$-connections over $\cT$. Even though the space of flat $G$-connections is independent of any choice of complex structure $J$ on $\Sigma$, its product with the space of complex structures on $\Sigma$ carries a non-trivial, $J$-dependent structure of K\"ahler fibration $\widehat{\boldsymbol{\omega}}_{\mathbf{J}}$ discovered by Hitchin \cite{hitchin:1987}. Building on this observation, we fix a symplectic form $\omega$ on $\Sigma$, and for any choice of coupling constant $\alpha >0$ and parameter $\varepsilon \in \{-1,1\}$, we consider the \emph{coupled harmonic equations}
\begin{equation}\label{eq:charmonicityintro}
\begin{split}
F_A -\frac{1}{2}[\psi,\psi] & = 0,\\
d_A\psi & = 0,\\
d_A^\ast\psi & = 0,\\
S_g - \alpha * d\left(B\left(\Lambda_\omega F_A,*\psi\right)\right) & = \frac{2\pi \chi(\Sigma)}{V},
\end{split}
\end{equation}
where $g = \omega(,J)$ and $*$ is the corresponding Hodge star operator. We denote by $\Uh(G)_\alpha^\varepsilon$ the moduli space of solutions of \eqref{eq:charmonicityintro} modulo \emph{unitary gauge} (see Section \ref{section-Uhitchin}). We denote by $\Uf(G)$ the universal moduli space of reductive flat $G$-connections over $\cT$ (see Section \ref{ssec:UFlat}). Our second main result can be stated as follows (cf. Theorem \ref{thm:metric} and Theorem \ref{thm:existence}):

\begin{theorem}\label{thm:existence-intro}
Let $X = (\Sigma,J)$ be a compact Riemann surface with genus $g(\Sigma) \geqslant 2$. Then, for any fixed total volume $V > 0$ and parameter $\varepsilon \in \{-1,1\}$, there exists $\alpha_0 > 0$ such that for any $0 < \alpha < \alpha_0$ there exists a non-empty open subset
$$
\mathcal{U}^*_{\alpha,\varepsilon} \subset \Uh(G)_\alpha^\varepsilon,
$$
endowed with a complex structure $\mathbb{J}$ and a pre-symplectic structure $\boldsymbol{\omega}^{\mathbb{J}}_{\alpha,\varepsilon}$. Furthermore, the induced map
\begin{equation}\label{eq:modulimapHFintro}
\mathcal{U}^*_{\alpha,\varepsilon} \lra \Uf(G)
\end{equation}
is holomorphic, and

\begin{enumerate}

\item if $\varepsilon = 1$ the tensor $\mathbf{g}^{\mathbb{J}}_{\alpha,\varepsilon} = \boldsymbol{\omega}^{\mathbb{J}}_{\alpha,\varepsilon}(,\mathbb{J})$ is possibly degenerate,

\item if $\varepsilon = -1$ the tensor $\mathbf{g}^{\mathbb{J}}_{\alpha,\varepsilon}$ is non-degenerate, and defines a pseudo-K\"ahler structure on the moduli space.

\end{enumerate}
In either case, $\boldsymbol{\omega}^{\mathbb{J}}_{\alpha,\varepsilon}$ admits a global K\"ahler potential, that is, $\boldsymbol{\omega}^{\mathbb{J}}_{\alpha,\varepsilon} = dd^c_\mathbb{J} \Phi$, where
$$
\Phi = \varepsilon\nu_\JJJ + \frac{\alpha}{2} \|\psi\|^2_{L^2} % \int_\Sigma B(\psi \wedge J \psi),
$$
and $\nu_\JJJ$ is induced by the global K\"ahler potential in the space of complex structures compatible with the orientation $\JJJ$ (see \cite[Section 4]{Fj}).
\end{theorem}

Similarly to the universal moduli space of Hitchin's equations \eqref{eq:cHitchinintintro}, the (possibly degenerate) moduli space K\"ahler form and metric admit explicit formulae upon suitable \emph{gauge fixing}
\begin{equation}\label{eq:omegagmoduliexintro}
\begin{split}
\boldsymbol{\omega}^{\mathbb{J}}_{\alpha,\varepsilon}(v_1,v_2) & = \frac{\varepsilon}{2}\int_{\Sigma}\tr(J\dot J_1\dot J_2) \omega \\
& + \alpha  \int_\Sigma B((a_1 - \psi(\dot J_1)) \wedge J(\dot \psi_2 - (J\psi)(\dot J_2))) \\
& - \alpha \int_\Sigma B((\dot \psi_1 - (J\psi)(\dot J_1)) \wedge J(a_2 - \psi(\dot J_2)))\\
& - \alpha \int_\Sigma B(\psi(\dot J_1)\wedge \psi(\dot J_2)),\\
\mathbf{g}^{\mathbb{J}}_{\alpha,\varepsilon}(v_1,v_2) & = \frac{\varepsilon}{2}\int_{\Sigma}\tr(\dot J_1\dot J_2) \omega \\
& + \alpha  \int_\Sigma B((a_1 - \psi(\dot J_1)) \wedge J(a_2 - \psi(\dot J_2))) \\
& + \alpha \int_\Sigma B((\dot \psi_1 - (J\psi)(\dot J_1)) \wedge J(\dot \psi_2 - (J\psi)(\dot J_2)))\\
& - \alpha \int_\Sigma B(\psi(\dot J_1)\wedge J(\psi(\dot J_2))).
\end{split}
\end{equation}
From the previous expression, we observe that, at least formally, the restriction of $\boldsymbol{\omega}^{\mathbb{J}}_{\alpha,\varepsilon}$ to the fibres over $\cT$ coincides up to scaling with Hitchin's symplectic structure $\omega_{\mathbf{J}}$. A complete proof of this requires a deeper understanding of the particular gauge fixing mechanism which we use to construct the moduli space complex structure $\mathbb{J}$ (see Proposition \ref{prop:lineargaugefixed}), and we leave it as an open question.

Based on the twistor space structure for Hitchin's original hyperk\"ahler moduli space $\Mg(G)$, it is natural to speculate on a relation between the moduli spaces $\Uh(G)_\alpha^\varepsilon$ and $\mathcal{U}^{Higgs}(G)$ in the weak coupling limit 
$$
\alpha \to 0.
$$
In this limit, for instance, the equations \eqref{eq:charmonicityintro} reduce to the coupled Hitchin equations \eqref{eq:cHitchinintintro}, and we expect a suitable adiabatic limit convergence
$$
\mathbf{g}^{\mathbb{J}}_{\alpha,\varepsilon}  \to \mathbf{g}^{\mathbb{I}}_{\tilde \alpha,\varepsilon}.
$$ 
We are far from even grasping the solution to this problem, and we leave it as an open question for future studies.

%In most problems dealing with moduli spaces of holomorphic bundles $E$, or bundles with Higgs fields of different types, one normally fixes the complex structure of the base manifold $X$. However it is very natural to consider the moduli problem for pairs $(X,E)$ in which both $X$ and $E$ `move' in a possibly  coupled way, that is, the complex structures on the base and on the bundle are both variables of the problem. Such study may  allow the construction of some sort of `universal moduli space' on a given base  parametrizing a distinguished class of complex structures on the underlying smooth manifold.

Motivation for the present work comes from the programme initiated by the first three authors in \cite{AGG1} more than 15 years ago based on the 2009 PhD Thesis of the second author \cite{GF}.
In that paper, we consider a pair $(X,E)$ consisting of a holomorphic principal $G$-bundle $E$ over a compact complex manifold $X$, where $G$ is a complex reductive group,  and study certain  coupled equations for a K\"ahler metric on $X$ and a reduction of structure group of $E$ to a maximal compact subgroup $K\subset G$. These equations, that we named  K\"ahler--Yang--Mills equations, appear naturally as moment map equations for the action of the extended gauge group of the $K$-bundle $E_K$ over the underlaying smooth manifold to $X$ equipped with a symplectic form $\omega$. The extension of the  gauge group of $E_K$ is given in this case by the group of Hamiltonian symplectomorphisms defined by  $\omega$. This theory is in some sense a combination of the usual Yang--Mills theory
and the Donaldson--Fujiki theory on K\"ahler manifolds, interpreting  the constant scalar curvature as the moment map for the action of the group of Hamiltonian symplectomorphisms on the space of complex structures compatible with the symplectic structure $\omega$.

While much work remains to be done in this programme, like for example finding the algebraic  general stability condition solving the K\"ahler--Yang--Mills equations, certain particular cases with symmetry have been understood using dimensional reduction methods \cite{AGG2,AGGP} and symplectic reduction in stages \cite{AGGPY}.  
A generalization of the K\"ahler--Yang--Mills equations
involving Higgs fields of different kinds has been introduced in \cite{AGG3} and remains also to be fully  explored. 

In early joint work with the fourth author, we extended our symplectic point of view to the problem of constructing universal  moduli spaces for Hitchin's equations and Higgs bundles. The first steps of this attempt appeared in Chapter 7 of the fourth author's 2018 ETH PhD Thesis \cite{samuel}, and is the seed of the present work. In view of the current interest in the subject, after a dormant period, we decided to resume this investigation.
We should warn that the use of the term `universal' that we make here does not necessarily
imply any  expected functoriality properties  that such  term some times entails. We have sticked, however, to this term since that is the one that we  originally used in our initial work \cite{samuel}.

After the completion of this work, we received a copy of a manuscript~\cite{Collier-Toulisse-Wentworth}, where related topics are treated.
It should be very interesting to compare our approach to other analytic constructions under study \cite{Collier-Toulisse-Wentworth,nigel}, as well as the more algebraic point of view  taken by Simpson \cite{simpson:1992,simpson:1994} and others \cite{BBN,BDP,D,DF}.

\begin{acknowledgements}
The authors wish to thank Emilio Franco, Nigel Hitchin and Ignasi Mundet for helpful conversations. Special thanks are due to Richard Wentworth for stimulating discussions --- in particular for spotting a mistake in the signature of the metric $\mathbf{g}^{\mathbb{I}}_{\alpha,1}$ in an earlier version of this manuscript --- as well as for his lectures at ICMAT in June 2025, which encouraged us to resume this project.
\end{acknowledgements}

%%%%%%%%%%%%%%%%%%%%%%%%%%%%%%%%%%%%%%%%%%%%%%%%%%%%%%%%%%%%%%%%%
\section{Flat connections and Higgs bundles}\label{flat}
%%%%%%%%%%%%%%%%%%%%%%%%%%%%%%%%%%%%%%%%%%%%%%%%%%%%%%%%%%%%%%%%%

%%%%%%%%%%%%%%%%%%%%%%%%%%%%%%%%%%%%%%%%%%%%%%%%%%%%%
\subsection{Flat $\boldsymbol{G}$-connections}\label{ssec:FlatG}
%%%%%%%%%%%%%%%%%%%%%%%%%%%%%%%%%%%%%%%%%%%%%%%%%%%%%

Through this section we consider a connected semisimple complex Lie group $G$, with Lie algebra $\glie$, and
fix an antiholomorphic involution $\tau$ of $G$ defining  a maximal compact subgroup $K:=G^\tau\subset G$, with Lie algebra
$\liek$. The Killing form will be denoted by $B$. We also fix a smooth oriented compact surface $\Sigma$.

Let $E_G$ be a smooth principal $G$-bundle over $\Sigma$. The space $\DDD$  of connections on $E_G$ is an  affine space modelled 
on the complex vector space $\Omega^1(\Sigma,E_G(\lieg))$, where $E(\glie)$ is the adjoint bundle associated to $E$ via
the adjoint representation of $G$ in $\glie$. Hence, $\DDD$ has a natural integrable complex structure $\mathbf{J}$, defined by
\begin{equation}\label{eq:bfJ}
\mathbf{J}(\dot D) = i \dot D, \quad \mbox{for}\;\;
D \in \DDD\;\; \mbox{and}\;\; \dot D \in T_D\DDD=\Omega^1(\Sigma,E_G(\lieg)),
\end{equation}
induced by the complex structure of $\lieg$ and compatible with the affine structure. 

The space $\DDD$ is furthermore holomorphic symplectic, with holomorphic symplectic structure, defined by 
\begin{equation}\label{eq:OmegaJ}
\Omega_{\mathbf{J}}(\dot D_1,\dot D_2)=\int_\Sigma B(\dot D_1\wedge \dot D_2),
\end{equation}
which is preserved by the action of the gauge group $\GGG$ of $E_G$. 
We have  that $\Lie \GGG=\Omega^0(\Sigma,E_G(\lieg))$ and hence its dual $(\Lie \GGG)^\ast$ can be identified with $\Omega^2(\Sigma,E_G(\lieg))$. 
As in the case of unitary connections, studied by Atiyah--Bott \cite{atiyah-bott:1982}, there exists an equivariant complex moment map for the action of $\GGG$ on $\DDD$ given by
\begin{equation}\label{curvature-cxmoment-map}
\begin{array}{rcl}
\mu_\GGG \colon \DDD   & \lra & \Omega^2(\Sigma,E_G(\lieg)) \\
D &\longmapsto &  F_D,
\end{array}\end{equation}
where $F_D$ is the curvature of $D$. The zero moment map condition, that is, $F_D = 0$, corresponds to the flatness of $D$. In this case one refers to the pair $(E_G,D)$ as a flat $G$-bundle.

In order to construct a Hausdorff moduli space of flat $G$-connections modulo gauge, one needs to impose a natural stability condition arising from Geometric Invariant Theory. If the connection $D$ is flat, the holonomy map defines a
representation $\rho$ of the fundamental group $\pi_1(\Sigma)$ of $\Sigma$ in $G$. The representation $\rho:\pi_1(\Sigma)\to G$ is said to be {\emph reductive} if the Zariski closure of the image of $\rho$ is a reductive group, or equivalently if $\ad_G \circ \rho$ is completely reducible, where  $\ad_G: G\to \GL(\lieg)$ is the adjoint representation. We say that  $(E_G,D)$ is \emph{reductive} if $\rho$ is reductive.

The existence of the moment map $\mu_\GGG$ implies that the flatness condition is invariant under the action of $\GGG$ on $\DDD$, and so is the reductiveness condition, and we can thus consider the moduli space
$$
\Mf(G):=\{\mbox{reductive}\;\; D\in \DDD\;\;  \mbox{with}\;\;
F_D=0\}/\GGG.
$$
The following important result, which we state informally, can be found in the work of Goldman \cite{goldman:1984}.

\begin{theorem}\label{FlatGmoduli}
The moduli space of flat $G$-connections $\Mf(G)$ on $E_G$ admits a natural structure of a (possibly singular) complex manifold, and a holomorphic symplectic structure induced by $\Omega_{\mathbf{J}}$ on its smooth locus.
\end{theorem}

%%%%%%%%%%%%%%%%%%%%%%%%%%%%%%%%%%%%%%%%%%%%%%%%%%%%%%%%%%%%%%%%%%%
\subsection{Surface group representations} 
%%%%%%%%%%%%%%%%%%%%%%%%%%%%%%%%%%%%%%%%%%%%%%%%%%%%%%%%%%%%%%%%%%%

The complex structure on the moduli space $\Mf(G)$ can be understood explicitly using $G$-representations of the fundamental group of the surface $\Sigma$ \cite{atiyah-bott:1982,hitchin:1987} (see Remark \ref{rem:algebraicst}). Let $\Hom(\pi_1(\Sigma),G)$ be the set of representations of $\pi_1(\Sigma)$ in $G$. Since
\begin{displaymath}
  \pi_{1}(\Sigma) = \big\langle a_{1},b_{1}, \dotsc, a_{g},b_{g} \suchthat
  \prod_{i=1}^{g}[a_{i},b_{i}] = 1 \big\rangle,
\end{displaymath}
$\Hom(\pi_1(\Sigma),G)$ can be naturally identified with the subset
of $G^{2g}$ consisting of $2g$-tuples
$(A_{1},B_{1}\dotsc,A_{g},B_{g})$ satisfying the algebraic equation
$\prod_{i=1}^{g}[A_{i},B_{i}] = 1$, and hence has a natural structure as a complex algebraic variety.

The \emph{moduli space of representations} of  $\pi_1(\Sigma)$ 
in $G$, or $G$-\emph{character variety} of $\pi_1(\Sigma)$  is defined as the quotient
$$
\calR(G):=\Hom^{+}(\pi_1(\Sigma),G)/ G,
$$
where
$\Hom^+(\pi_1(\Sigma),G)\subset \Hom(\pi_1(\Sigma),G)$ is the subvariety of  reductive representations, and $G$ acts by conjugation.
%Recall that, here, a \textbf{reductive representation} is  one that composed with the adjoint representation in the Lie algebra of $G$ decomposes as a sum of irreducible representations, or equivalently, one for which the Zariski closure of the image of $\pi_1(\Sigma)$ in $G$ is  a reductive group. 
This quotient coincides with the GIT quotient and hence
$$
\calR(G)=\Hom(\pi_1(\Sigma),G)\sslash G
$$
is a complex algebraic variety. 

%One can associate a topological invariant to  a representation $\rho\colon\pi_{1}(\Sigma) \to G$. To do this, 
Given a representation
$\rho\colon\pi_{1}(\Sigma) \to G$, consider the  associated flat $G$-bundle on
$\Sigma$, defined as
\begin{math}
  E_{\rho} = \tilde{\Sigma}\times_{\rho}G
\end{math},
where $\tilde{\Sigma} \to \Sigma$ is the universal cover and $\pi_{1}(\Sigma)$ acts
on $G$ via $\rho$.
This gives in fact  an identification between the set of equivalence classes
of representations  $\Hom(\pi_1(\Sigma),G) / G$ and the set of equivalence classes
of flat $G$-bundles, which in turn is parametrized by the  cohomology set
$H^1(\Sigma,G)$. We can then assign  a topological invariant  to a representation
$\rho$ given by the  characteristic class $c(\rho):=c(E_{\rho})\in \pi_1(G)$
corresponding
to $E_{\rho}$. To define this, let $\widetilde G$ be the universal covering group
of $G$. We have an exact  sequence
$$
1 \lra\pi_1(G)\lra \widetilde G \lra G \lra 1
$$
which gives rise to the (pointed sets) cohomology sequence
\begin{equation}\label{characteristic}
H^1(\Sigma, {\widetilde G}) \lra  H^1(\Sigma, {G})\stackrel{c}{\lra}   H^2(\Sigma, \pi_1(G)).
\end{equation}

Since $\pi_1(G)$ is abelian, we have
$$
 H^2(\Sigma, \pi_1(G))\cong  \pi_1(G),
$$
and $c(E_\rho)$ is defined as the image of $E$ under the last map in
(\ref{characteristic}). Thus the class $c(E_\rho)$ measures  the
obstruction to lifting $E_\rho$ to a flat $\widetilde G$-bundle, and hence to
lifting $\rho$ to a representation of $\pi_1(\Sigma)$ in $\widetilde G$.  

For a fixed $c\in \pi_1(G)$, the
\emph{moduli space of reductive representations} $\mathcal{R}_c(G)$
with topological invariant $c$
is defined as the subvariety
\begin{equation}\label{eq:RdG}
  \calR_c(G):=\{\rho \in \mathcal{R}(G)  \; \mid \;
  c(\rho)=c\}.
\end{equation}

\begin{proposition}\label{flat-rep}
  Let $E_G$ be a smooth principal $G$-bundle over $\Sigma$ with topological class $c$. Then there is a complex analytic isomorphism
  $$
\Mf(G)\cong {\calR}_c(G).
  $$
\end{proposition}

\begin{remark}\label{rem:algebraicst}
We warn the reader that, even though there is a real analytic isomorphism between the moduli space of flat $G$-connections $\Mf(G)$ and the $G$-character variety with fixed topological invariant $c$, these spaces carry very different complex structures (see e.g. \cite{simpson:1994}).
\end{remark}

%%%%%%%%%%%%%%%%%%%%%%%%%%%%%%%%%%%%%%%%%%%%%%%%%%%%%
\subsection{Harmonicity equations and hyperk\"ahler structure} 
%%%%%%%%%%%%%%%%%%%%%%%%%%%%%%%%%%%%%%%%%%%%%%%%%%%%%

Let $E_G$ be a smooth principal $G$-bundle over $\Sigma$. Let $h\in \Omega^0(E_G(G/K))$ be a reduction
of structure group of $E_G$ to $K$, and let $E_K$ be the corresponding principal $K$-bundle. From the decomposition
$\lieg=\liek\oplus i\liek$ one has that any connection $D \in \DDD$ on $E_G$ admits a decomposition
$$
D=A+i\psi
$$
where $A$ is a connection on $E_K$ and $\psi\in \Omega^1(X,E_K(\liek))$. We thus have a one-to-one correspondence
\begin{equation}\label{eq:TADDD}
\begin{array}{rcl}
 \AAA \times \Omega^1(\Sigma,E_K(\liek))  & \lra &  \DDD \\
(A,\psi) &\longmapsto & D:=A+i\psi,
\end{array}\end{equation}
where $\AAA$ denotes the space of connections on $E_K$. As observed by Hitchin \cite{hitchin:1987}, the previous isomorphism naturally identifies the space of $G$-connections $\DDD$ with the cotangent space to $\AAA$, suggesting a relation to hyperk\"ahler geometry.

To recall this relation, we fix a complex structure $J$ on $\Sigma$ compatible with the orientation, and denote by $X = (\Sigma,J)$ the corresponding Riemann surface. Via $J$, the space $\AAA$ inherits an integrable complex structure \cite{atiyah-bott:1982}, and hence so does $\DDD$ via its identification with $T^*\AAA$ above. Explicitly, this is given by
$$
\mathbf{I}(a,\dot \psi) = (Ja,-J\dot \psi)
$$
where $Ja = - a(J)$. Let us denote
$\mathbf{K}:= \mathbf{I}\mathbf{J}$. Then, one can check that $\mathbf{K}$ defines a new complex structure and furthermore that the three complex structures
\begin{equation}\label{eq:IJK}
\begin{split}
\mathbf{I}(a,\dot \psi) & = (Ja,-J\dot \psi),\\
\mathbf{J}(a,\dot \psi) & = (-\dot \psi,a),\\
\mathbf{K}(a,\dot \psi) & = (-J \dot \psi,-Ja),
\end{split}\end{equation}
satisfy the quaternion relations, and define a flat hyperk\"ahler structure on $\DDD \cong T^*\AAA$ with metric tensor
\begin{equation}\label{eq:gHK}
\mathbf{g}((a,\dot \psi),(a,\dot \psi)) = \int_\Sigma B(a \wedge Ja) + \int_\Sigma B(\dot \psi \wedge J\dot \psi).
\end{equation}
The symplectic structures $\omega_{\mathbf{I}} = \mathbf{g}(\mathbf{I},)$ and $\omega_{\mathbf{K}} = \mathbf{g}(\mathbf{K},)$ combine to give the $\mathbf{J}$-holomorphic symplectic structure studied in Section \ref{ssec:FlatG}
\begin{equation}\label{eq:omegaIK}
\Omega_{\mathbf{J}} = \omega_{\mathbf{I}} + i \omega_{\mathbf{K}}.
\end{equation}
A Hamiltonian action of the gauge group $\KKK$ of $E_K$ for the symplectic structure $\omega_{\mathbf{J}} = \mathbf{g}(\mathbf{J},)$, given by
\begin{equation}\label{eq:omehaJ}
\omega_{\mathbf{J}}((a_1,\dot \psi_1),(a_2,\dot \psi_2)) =  \int_\Sigma B(a_1 \wedge J\dot \psi_2) - \int_\Sigma B(\dot \psi_1 \wedge Ja_2),
\end{equation}
was studied by Corlette \cite{corlette:1988}.

\begin{proposition}\label{prop:corlette}
The action of $\KKK$ on $(\DDD,\omega_{\mathbf{J}})$ is Hamiltonian, with equivariant moment map given by
$$
\langle \mu_{\KKK}^{\mathbf{J}}(D), \zeta \rangle = \int_\Sigma B(\zeta,d_A(J\psi)),
$$
where $\zeta \in \Lie \KKK=\Omega^0(\Sigma,E_K(\liek))$.
\end{proposition}

The moment maps $\mu_\GGG = \mu_{\KKK}^{\mathbf{I}} + i \mu_{\KKK}^{\mathbf{K}}$ and $\mu_{\KKK}^{\mathbf{J}}$ combine to give a hyperk\"ahler moment map for the $\KKK$-action on $\DDD$
$$
{\bmu}_\KKK = (\mu_{\KKK}^{\mathbf{I}},\mu_{\KKK}^{\mathbf{J}},\mu_{\KKK}^{\mathbf{K}})
$$
whose zero locus corresponds to solutions of the \emph{harmonicity equations}
\begin{equation}\label{harmonicity}
\begin{split}
F_A -\frac{1}{2}[\psi,\psi]&= 0,\\
d_A\psi&=0,\\
d_A^\ast\psi&=0.
\end{split}
\end{equation}
Note that the first two are equivalent to the flatness condition for $D$, via the formula
$$
F_D = F_A -\frac{1}{2}[\psi,\psi] + i d_A\psi,
$$
while the last equation is equivalent to the vanishing of the moment map $\mu_{\KKK}^{\mathbf{J}}$, due to the standard identity $\ast \psi = J\psi$.

The harmonicity condition for the reduction $h$ admits an interpretation in terms of the the holonomy representation
$\rho:\pi_1(\Sigma)\to G$ corresponding to $D$. The reduction $h$ is equivalent to a $\pi_1(\Sigma)$-equivariant smooth map
$$
\tilde{h}: \tilde{\Sigma} \lra G/K,
$$
where $\tilde{\Sigma}$ is the universal covering of $\Sigma$. Here $\pi_1(\Sigma)$ acts on $ \tilde{\Sigma}$ by Deck transformations, and via the representation $\rho$ on $G/K$. The symmetric space $G/K$ is equipped with a canonical $G$-invariant Riemannian metric determined by the Killing form $B$, and (\ref{harmonicity}) is
equivalent to the harmonicity of $\tilde{h}$ in the usual sense. 

The equations \eqref{harmonicity} are invariant under the action of $\KKK$ and hence we can consider the moduli space
$$
\Mh(G):=\{(A,\psi)\;\;\mbox{satisfying}\;\;
(\ref{harmonicity})\}/\KKK,
$$
given by the infinite-dimensional hyperk\"ahler reduction ${\bmu}_\KKK^{-1}(0)/\KKK$.

The following theorem, due to  Donaldson \cite{donaldson:1987} for $G=\SL(2,\C)$,  and Corlette \cite{corlette:1988} in general, is one of the main structural results of the theory.

\begin{theorem}\label{corlette}
The map \eqref{eq:TADDD} induces a homeomorphism
$$
\Mh(G)\cong\Mf(G).
$$
Furthermore, $\Mh(G)$ carries a natural hyperk\"ahler structure on its smooth locus, whose generic complex structure is biholomorphic to $\Mf(G)$.
\end{theorem}

In particular, the previous theorem states that a solution of the harmonicity equations \eqref{harmonicity} has an associated reductive flat connection $D = A + i \psi$. Conversely, any pair $(A,\psi)$ such that $D = A + i \psi$ is a reductive flat connection, admits an element on its $\GGG$-orbit, unique up to the $\KKK$-action, solving \eqref{harmonicity}. The non-generic complex structures on $\Mh(G)$ correspond to the moduli space of $G$-Higgs bundles over $X$, induced by complex structure $\mathbf{I}$, and its conjugate, as we discuss in the next sections. 

%%%%%%%%%%%%%%%%%%%%%%%%%%%%%%%%%%%%%%%%%%%%%%%%%%%%%%%%%%%%%%%%%%%%
\subsection{$\boldsymbol{G}$-Higgs bundles and Hitchin's equations}
\label{section-hitchin-equations}
%%%%%%%%%%%%%%%%%%%%%%%%%%%%%%%%%%%%%%%%%%%%%%%%%%%%%%%%%%%%%%%%%%%%

Through this section we consider a connected semisimple complex Lie group $G$, with Lie algebra $\glie$, and
fix an antiholomorphic involution $\tau$ of $G$ defining  a maximal compact subgroup $K:=G^\tau\subset G$, with Lie algebra
$\liek$. The Killing form will be denoted by $B$. 

We also consider a compact Riemann surface $X$ with canonical line bundle $K_X$.

A \emph{$G$-Higgs bundle} over $X$ is a pair   $(E,\varphi)$
consisting of a  holomorphic principal $G$-bundle
$E$ over $X$ and an element $\varphi\in H^0(X,E(\glie)\otimes K_X)$, that is a holomorphic section
of $E(\glie)\otimes K_X$, where $E(\glie)$ is the adjoint bundle associated to $E$ via
the adjoint representation of $G$ in $\glie$.

There are appropriate notions of stability, semistability and
polystability and we can consider $\cM^{Higgs}(G)$ to be the \emph{moduli space of isomorphism classes polystable $G$-Higgs bundles over $X$}.

Let $h\in \Omega^0E(X,G/K)$ be a smooth reduction of structure group of $E$ to $K$ and let $E_h$ be the corresponding smooth principal $K$-bundle. 
Let $F_h\in \Omega^2(X,E_h(\liek))$ be the curvature of the unique connection compatible with $h$ and the holomorphic structure on $E$, defined by the Chern--Singer correspondence~\cite{chern,singer}. 
Here $E_h(\liek)$ is the adjoint bundle of $E_h$.

Abusing notation, let
$$
\tau\colon\Omega^{1,0}(X,E(\lieg)) \lra\Omega^{0,1}(X,E(\lieg))
$$
be the conjugation  defined by $h$ combined with the conjugation
on complex $1$-forms on $X$.
The Higgs field $\varphi$ can be
viewed as a $(1,0)$-form $\varphi \in \Omega^{1,0}(X,E(\lieg))$, then
$\tau(\phi)\in \Omega^{0,1}(X,E(\lieg))$, and hence 
  $[\varphi,\tau(\varphi)]\in \Omega^2(X,E_h(\liek))$.

The proof of the following theorem is due to Hitchin \cite{hitchin:1987} for $G=\SL(2,\C)$ and Simpson \cite{simpson:1988} for arbitrary $G$.

\begin{theorem} \label{higgs-hk}
  A reduction $h$ of structure group of $E$ to $K$
  satisfies Hitchin's equation
  $$
  F_h -[\varphi,\tau(\varphi)]= 0 
  $$
  if and only if $(E,\varphi)$ is polystable.
\end{theorem}

From the point of view of moduli spaces it is convenient to fix a $C^\infty$ principal $K$-bundle $E_K$ and study the moduli space of solutions to \emph{Hitchin's equations} for a pair $(A,\varphi)$ consisting of  a connection $A$ on $E_K$ and $\varphi\in \Omega^{1,0}(X,E_K(\lieg))$:
\begin{equation}\label{hitchin}
\begin{split}
F_A -[\varphi,\tau(\varphi)]&=0,\\
\dbar_A\varphi&=0.
\end{split}
\end{equation}
Here  $\dbar_A$ is the $(0,1)$ part  of the covariant derivative $d_A$ defined by $A$.
This is of course the Dolbeault operator defined by the  holomorphic structure $J_A$ on $E_G$ corresponding  to  $A$ via de Chern--Singer correspondence.
The bundle $E_G$ is the smooth principal $G$-bundle obtained from $E_K$ by extension of structure group.
The gauge group $\KKK$  of automorphisms of  $E_K$ acts on the space of solutions, and the moduli space of solutions is
$$
\Mg(G):= \{ (A,\varphi)\;\;\mbox{satisfying}\;\;
(\ref{hitchin})\}/\KKK.
$$
Now, from Theorem \ref{higgs-hk} one has the following.

\begin{theorem}\label{hk}
There is a homeomorphism
\[
\cM^{Higgs}(G)\cong \Mg(G).
\]
\end{theorem}

To explain this correspondence we interpret the moduli
space of $G$-Higgs bundles in terms of pairs $(J_E, \varphi)$ consisting
of a holomorphic structure $J_E$ on the smooth $G$-bundle
$E_G$ obtained from $E_K$ by the extension of structure group, and
$\varphi\in \Omega^{1,0}(X,E_{G}(\lieg))$
satisfying $\dbar_E\varphi=0$.
Such pairs are in correspondence with  $G$-Higgs bundles $(E,\varphi)$,
where $E$ is the holomorphic $G$-bundle defined by $J_E$ on $E_{G}$, and $\dbar_E\varphi=0$,  that is,
$\varphi\in H^0(X,E(\lieg)\otimes K_X)$. 
The moduli space of polystable $G$-Higgs bundles $\cM^{Higgs}(G)$ can now
be identified with the orbit space
$$
\{(J_E,\varphi)\;\;\mbox{with}\;\; \dbar_E\varphi=0\;\;\mbox{such that $(E,\varphi)$ is  polystable}\}/\GGG,
$$
where $\GGG$ is the gauge group of automorphisms of $E_{G}$, which is in fact
the complexification of $\KKK$.
Since, by the Chern--Singer correspondence, there is a bijection between
connections on $E_K$ and holomorphic structures on $E_{G}$,
the correspondence given in Theorem \ref{hk} can be interpreted
by saying that in the $\GGG$-orbit of a polystable $G$-Higgs
bundle $(J_{E_0},\varphi_0)$ one can find another Higgs bundle
$(J_E,\varphi)$
whose corresponding pair $(A,\varphi)$ satisfies
$F_A -[\varphi,\tau(\varphi)]= 0$, and this is unique up to gauge
transformations in $\KKK$.

%%%%%%%%%%%%%%%%%%%%%%%%%%%%%%%%%%%%%%%%%%%%%%%%%%%%%%%%%%%
\subsection{Hitchin's equations and hyperk\"ahler structure}
\label{hitchin-hyperkahler}
%%%%%%%%%%%%%%%%%%%%%%%%%%%%%%%%%%%%%%%%%%%%%%%%%%%%%%%%%%%

Coming back to the setup of Section \ref{section-hitchin-equations},
let $E_K$ be a smooth principal $K$-bundle  
over $X$, and let $E_G$ the principal $G$-bundle obtained by extension
of structure group.

The space $\AAA$  of connections on $E_K$ is an  affine space modelled 
on  $\Omega^1(X,E_K(\liek))$, which is equipped 
with  a symplectic structure defined by
$$
\omega_\AAA(a,b)=\int_X B(a \wedge b),\;\;\;\mbox{for}\;\;
A\in \AAA\;\; \mbox{and}\;\; a, b \in T_A\AAA=\Omega^1(X,E_K(\liek)).
$$
This is obviously closed since it is independent of $A\in\AAA$.

Now, the set $\CCC$ of holomorphic structures on $E_{G}$
is an affine space modelled on $\Omega^{0,1}(X,E_{G}(\lieg))$, and it has
a complex structure $J_\CCC$, induced  by the complex structure of the 
Riemann surface, which is defined by
$$
J_\CCC(\alpha)=i\alpha,\;\;\;\mbox{for}\;\; J_E\in \CCC\;\; 
\mbox{and} \;\; \alpha\in T_{J_E}\CCC=\Omega^{0,1}(X,E_{G}(\lieg)).
$$

As mentioned above, the Chern--Singer correspondence establishes an isomorphism
\begin{equation}\label{connections-holomorphic}
\begin{array}{rcl} 
  \AAA  &\lra & \CCC  \\ 
  A   &\longmapsto &  J_A.
\end{array}\end{equation}
The corresponding  tangent spaces are in bijection under the map
\begin{equation}\label{connections-holomorphic-tangent}  
\begin{array}{rcl} 
  \Omega^{0,1}(X,E_G(\lieg)) &\lra & \Omega^1(X,E_K(\liek))  \\ 
  \alpha   &\longmapsto &  a:=\alpha -\tau(\alpha).
\end{array}\end{equation}
Under this identification $J_\CCC$ defines  a complex structure $J_\AAA$ on $\AAA$.

The symplectic structure $\omega_\AAA$ and the complex structure $J_\AAA$
define a K\"ahler structure on $\AAA$, which is preserved by the action of 
the gauge group $\KKK$ of $E_K$. 
We have  that $\Lie \KKK=\Omega^0(X,E_K(\liek))$ and hence 
its dual $(\Lie \KKK)^\ast$ can be identified with $\Omega^2(X,E_K(\liek))$.
By Atiyah--Bott  \cite{atiyah-bott:1982} one has 
that the moment map for the action of $\KKK$ on $\AAA$ is given by
\begin{equation}\label{curvature-moment-map}
\begin{array}{rcl}
 \AAA   & \lra & \Omega^2(X,E_K(\liek)) \\
A &\longmapsto &  F_A.
\end{array}\end{equation}

Now, let us denote 
$\Omega=\Omega^{1,0}(X,E_{G}(\lieg))$. The linear  space $\Omega$ has a natural
complex structure $J_\Omega$   defined by multiplication by $i$, and a 
symplectic structure given by
$$
\omega_\Omega (\eta,\nu)=i\int_X B(\eta\wedge \nu^\ast),
\;\;\;\mbox{for}\;\;\varphi\in \Omega\;\;\mbox{and}\;\; \eta,\nu\in T_\varphi\Omega=\Omega.
$$
We can now consider $\AAA\times \Omega$ with the symplectic structure
$\omega_\AAA + \omega_\Omega$ and complex structure
$J_\AAA + J_\Omega$. The action of $\KKK$ on $\AAA\times \Omega$ preserves 
these symplectic and complex structures and there is a moment map
given by (\cite{hitchin:1987})
\begin{equation}\label{moment-higgs}
\begin{array}{rcl}
 \AAA\times \Omega  & \lra & \Omega^2(X,E_K(\liek)) \\
(A,\varphi)& \longmapsto & F_A-[\varphi,\tau(\varphi)].
\end{array}
\end{equation}

Let us denote
$J_1:=J_\AAA + J_\Omega$. Via the identification $\AAA\cong\CCC$, we have for $\alpha\in \Omega^{0,1}(X,E_{G}(\lieg))$ and $\eta\in\Omega^{1,0}(X,E_G(\lieg))$ the following three complex structures on $\AAA\times \Omega$:
\begin{align*}
 J_1(\alpha,\eta)& = (i\alpha,i\eta),\\
 J_2(\alpha,\eta)& = (-i\tau(\eta),i\tau(\alpha)),\\
 J_3 (\alpha,\eta)& = (\tau(\eta),-\tau(\alpha)).
\end{align*}
These complex structures correspond to the complex structures $\mathbf{I},\mathbf{J}$ and $\mathbf{K}$ in~\eqref{eq:IJK} via the Chern--Singer correspondence.

Consequently, $J_i$, $i=1,2,3$, satisfy the quaternion relations, and define
a hyperk\"ahler structure on  $\AAA\times \Omega$, with symplectic structures
$\omega_i$, $i=1,2,3$, where $\omega_1=\omega_\AAA + \omega_\Omega$.
The symplectic structures $\omega_2$ and $\omega_3$ combine to define the $J_1$-holomorphic symplectic structure
$\omega_c:=\omega_2 + i\omega_3$,  given
by
$$
\omega_c((a,\eta), (b,\nu))=\int_X B(\eta\wedge \beta - \nu\wedge \alpha),
$$
for
$(a,\eta), (b,\nu)\in T_{A,\varphi)}(\AAA\times \Omega) \cong \Omega^1(X,E_K(\liek))\times \Omega^{1,0}(X,E_G(\lieg))$,
where $\alpha$ and $\beta$ are the elements in
$\Omega^{0,1}(X,E(\lieg)$ corresponding to $a$ and $b$ respectively under the
identification (\ref{connections-holomorphic-tangent}).

The action of the gauge group $\KKK$ on $\AAA\times \Omega$ preserves the 
hyperk\"ahler structure and there are  moment maps
given by
$$
\mu_1(A,\varphi)=F_A - [\varphi,\tau(\varphi)],\;\;\;
\mu_2(A,\varphi)=\Real(\dbar_E\varphi),\;\;\;
\mu_3(A,\varphi)=\Imaginary(\dbar_E\varphi).
$$
We thus have that ${\bmu}^{-1}({0})/\KKK$
is the moduli space $\Mg(G)$
of solutions to Hitchin's equations (\ref{hitchin}).
In particular, if we consider the set ${\bmu}_*^{-1}({0})$ of  irreducible solutions
(equivalently, smooth) one has that
$$
{\bmu}_*^{-1}({0})/\KKK$$
is a hyperk\"ahler manifold which, by 
Theorem \ref{hk},  is homeomorphic to the subvariety of smooth
points of the moduli space
$\cM^{Higgs}(G)$, consisting  of stable and simple $G$-Higgs bundles on $E_G$.

%%%%%%%%%%%%%%%%%%%%%%%%%%%%%%%%%%%%%%%%%%%%%%
\subsection{Non-abelian Hodge correspondence}
%%%%%%%%%%%%%%%%%%%%%%%%%%%%%%%%%%%%%%%%%%%%%

Let us now denote by $\cM_c(G)$ the moduli space of isomorphism classes of polystable $G$-Higgs bundles where the $G$-bundle has tological class $c\in \pi_1(G)$. We have the following.

\begin{theorem}\label{na-Hodge}
There is a homeomorphism
$\calR_c(G) \cong \cM_c(G)$. 
\end{theorem}

\begin{remark}
On the open subvarieties defined by the smooth 
points of $\mathcal{R}_c(G)$ and $\mathcal{M}_c(G)$, this correspondence is in 
fact an isomorphism of real analytic varieties.
\end{remark}

Theorem \ref{na-Hodge} is proved by combining Proposition \ref{flat-rep} with Theorems \ref{corlette} and \ref{hk}, together with the following proposition.
\begin{proposition}\label{prop:circle}
The correspondence $(A,\varphi)\mapsto (A,\psi:=-i(\varphi-\tau(\varphi)))$
defines  a homeomorphism
$$
\Mg(G)\cong \Mh(G).
$$
\end{proposition}

One can easily see that under the 
affine map 
$$
\begin{array}{rcl}
\AAA\times \Omega & \lra &\DDD,\\
(A,\varphi)&\longmapsto & A - i \varphi + i\tau(\varphi),
\end{array}
$$
$\AAA\times\Omega$ with complex structure $J_2$ corresponds to $\DDD$ with complex
structure $\mathbf{J}$ (see \cite{hitchin:1987}).

Now, Theorems  \ref{hk}  and \ref{corlette} 
can be regarded as existence theorems, establishing the non-emptiness
of the hyperk\"ahler quotient, obtained by focusing on different
complex structures. For Theorem \ref{hk} one gives a special
status to the complex 
structure $J_1$. Combining the symplectic forms determined by  $J_2$
and $J_3$ one has  the $J_1$-holomorphic symplectic form
$\omega_c=\omega_2 +i\omega_3$ on $\AAA\times\Omega$. The
gauge group $\GGG=\KKK^\C$ acts on $\AAA\times \Omega$ preserving $\omega_c$.
The symplectic quotient construction can also be extended to the
holomorphic situation (see e.g. \cite{kobayashi:1987}) to obtain
the holomorphic symplectic quotient 
$\{(J_E,\varphi)\;:\;\dbar_E\varphi=0\}/\GGG $.
What Theorem \ref{hk} says is that for a class 
$[(J_E,\varphi)]$ in this quotient
to have a representative (unique up to $K$-gauge) satisfying
$\mu_1=0$ it is necessary  and sufficient that the pair 
$(J_E, \varphi)$ be polystable. 
This identifies the hyperk\"ahler quotient to 
the set of equivalence classes of polystable $G$-Higgs bundles on $E_G$.
If one now takes $J_2$ on $\AAA\times \Omega$ or equivalently 
$\DDD$ with $\mathbf{J}$ and argues in a similar way, one gets Theorem 
\ref{corlette} identifying the  hyperk\"ahler quotient to 
the set of equivalence classes of reductive  flat
connections on $E_G$.

%%%%%%%%%%%%%%%%%%%%%%%%%%%%%%%%%%%%%%%%%%%%%%%%%%%%%%%%%%%%%%%%%%%%%%%
\section{K\"ahler fibrations and coupled harmonic equations}
\label{extended}
%%%%%%%%%%%%%%%%%%%%%%%%%%%%%%%%%%%%%%%%%%%%%%%%%%%%%%%%%%%%%%%%%%%%%%%

%%%%%%%%%%%%%%%%%%%%%%%%%%%%%%%%%%%%%%%%%%%%%%%%%
\subsection{K\"ahler fibrations}\label{ssec:Kfibration}
%%%%%%%%%%%%%%%%%%%%%%%%%%%%%%%%%%%%%%%%%%%%%%%%%

In this section we recall some basic aspects of the theory of K\"ahler fibrations, which we will use later. We follow closely \cite{GLS,mundet:2000}.

A K\"ahler fibration is a holomorphic fibre bundle $\XXX \to \JJJ$, with typical fibre a K\"ahler manifold $(\DDD,\omega_\DDD)$, and such that transition functions between local holomorphic trivializations are contained in the group of K\"ahler isometries of $\DDD$. Equivalently, $\XXX$ admits a smoothly varying K\"ahler structure on the fibres, that we shall denote $\widehat{\boldsymbol{\omega}}$. An Ehresmann connection $\Gamma$ on $\XXX$, given by a distribution of horizontal subspaces
$$
H^\Gamma \subset T \XXX,
$$
is said to be K\"ahler if the associated parallel transport is by K\"ahler isometries of the fibres. By the fibrewise non-degeneracy of $\widehat{\boldsymbol{\omega}}$, Ehresmann connections on $\XXX$ correspond to real $2$-forms $\boldsymbol{\sigma} \in \Omega^{2}(\XXX,\RR)$, which restrict to $\widehat{\boldsymbol{\omega}}$ on the fibres. Given such a $\boldsymbol{\sigma}$, the horizontal subspace of the associated connection $\Gamma^{\boldsymbol{\sigma}}$ is
\begin{equation}\label{eq:Hsigma}
H^{\boldsymbol{\sigma}} = \{v \in T\XXX \; | \; i_v\boldsymbol{\sigma}_{|V\XXX} = 0\},
\end{equation}
where $V\XXX \subset T\XXX$ is the vertical bundle of the fibration $\XXX \to \JJJ$. Conversely, given a connection $\Gamma$, we can define
$$
\boldsymbol{\sigma}_\Gamma = \widehat{\boldsymbol{\omega}}(\Gamma,\Gamma),
$$
where $\Gamma \colon T\XXX \to V\XXX$ is the projection induced by $\Gamma$. 

\begin{definition}
Given a K\"ahler fibration $(\XXX \to \JJJ,\widehat{\boldsymbol{\omega}})$ with K\"ahler Ehresmann connection $\Gamma$, a \emph{coupling form} for the connection $\Gamma$ is a closed 2-form $\boldsymbol{\sigma} \in \Omega^{1,1}(\XXX,\RR)$ on $\XXX$ restricting on the fibres to $\widehat{\boldsymbol{\omega}}$ and such that $H^{\boldsymbol{\sigma}} = H^\Gamma$.
\end{definition}

The following basic result shows that the existence of a closed $(1,1)$-form restricting on the fibres to $\widehat{\boldsymbol{\omega}}$ is indeed a sufficient condition for the existence of a K\"ahler connection. The proof follows similarly to~\cite[Theorem 1.2.4]{GLS} and is omited.

\begin{theorem}\label{th:Kfib}
A K\"ahler fibration $(\XXX \to \JJJ,\widehat{\boldsymbol{\omega}})$ admits a K\"ahler connection provided that there exists $\boldsymbol{\sigma} \in \Omega^{1,1}(\XXX,\RR)$ restricting on the fibres to $\widehat{\boldsymbol{\omega}}$ and such that
$$
d \boldsymbol{\sigma} = 0.
$$
If this is the case, the horizontal subspace of the induced connection is given by $H^{\boldsymbol{\sigma}}$ in \eqref{eq:Hsigma}, and $\boldsymbol{\sigma}$ is a coupling form for $\Gamma^{\boldsymbol{\sigma}}$. 
\end{theorem}

The existence of a coupling form is a non-trivial question, related to the topology of the fibration. In the next result we recall a remarkable identity which relates the curvature of $\Gamma$ with the horizontal part of any coupling $\boldsymbol{\sigma}$ (see \cite[Eq. (1.12)]{GLS}). Recall that the curvature of an Ehresmann connection $\Gamma$ on $\XXX \to \JJJ$ is the basic 2-form $F_\Gamma \in \Omega^2(\XXX,V\XXX)$ defined by
$$
F_\Gamma(v_1,v_2) = - \Gamma[v_1^\Gamma,v_2^\Gamma],
$$
where $v_j \in T\XXX$, and $v_j^\Gamma \in H^\Gamma$ denotes the horizontal projection with respect to $\Gamma$. 
%Under the hypothesis of the previous theorem, there is a remarkable identity which relates the curvature of $\Gamma = \Gamma^{\boldsymbol{\sigma}}$ with the horizontal part of $\boldsymbol{\sigma}$ (see \cite[Eq. (1.12)]{GLS}).

\begin{proposition}\label{p:Kfibide}
Assume the hypothesis of Theorem \ref{th:Kfib}. Then
\begin{equation}\label{eq:magicid}
\widehat{\boldsymbol{\omega}}(F_\Gamma(v_1,v_2),\Gamma \cdot) = - d(\boldsymbol{\sigma}(v_1^\Gamma,v_2^\Gamma))(\Gamma \cdot).
\end{equation}
In particular, if the horizontal part of the coupling form $\boldsymbol{\sigma}$ is non-constant along the fibres, then the connection $\Gamma$ cannot be flat.
\end{proposition}

When the fibre $\DDD$ is compact and simply connected, the previous result implies the existence of a natural choice of coupling form $\boldsymbol{\sigma}$, given by imposing a `Gysin type condition' (see \cite[Th. 1.4.1]{GLS}). The idea is that, under this hypothesis, the vertical vector field $F_\Gamma(v_1,v_2)$ is Hamiltonian, and hence it determines up to a constant the vertical variation of the function $\boldsymbol{\sigma}(v_1^\Gamma,v_2^\Gamma)$. More explicitly, provided that there exists a coherent choice of a fibrewise Hamiltonian function $\mu(v_1,v_2)$ for $F_\Gamma(v_1,v_2)$, for any pair of horizontal vector fields $v_1,v_2$, that is,
$$
d\mu(v_1,v_2) = \widehat{\boldsymbol{\omega}}(F_\Gamma(v_1,v_2),)
$$
the natural choice of coupling form is
\begin{equation}\label{eq:couplingsigma}
\boldsymbol{\sigma} = \widehat{\boldsymbol{\omega}}(\Gamma,\Gamma) - \mu,
\end{equation}
where $\mu$ is regarded as a basic 2-form. This is the case, for instance, if the connection $\Gamma$ has holonomy contained in a subgroup $K$ of Hamiltonian isometries of the fibre $\DDD$, and there is a $K$-equivariant moment map $\mu \colon \DDD \to \mathfrak{k}^*$. 

A fundamental question in the theory of K\"ahler fibrations is whether there exists a K\"ahler metric $\omega_\XXX$ on $\XXX$ which restricts to $\widehat{\boldsymbol{\omega}}$ on the fibres. This is the case, for instance, if $\JJJ$ is K\"ahler, with K\"ahler form $\boldsymbol{\omega}_\JJJ$, and both $\XXX$ and $\JJJ$ are compact. In this situation, and under the hypothesis of Theorem \ref{th:Kfib}, a natural choice of K\"ahler form on $\XXX$ is the \emph{minimal coupling}
$$
\boldsymbol{\omega}_\alpha = \boldsymbol{\omega}_\JJJ + \alpha \boldsymbol{\sigma}
$$
for a small `coupling constant' $0<\alpha\ll 1$.

In the present paper, we are interested in the study of K\"ahler fibrations satisfying the hypothesis of Theorem \ref{th:Kfib}. We will find situations in which the presymplectic manifold $(\XXX,\boldsymbol{\sigma})$ admits a Hamiltonian action of a real Lie group $\widetilde{\GGG}$, preserving the holomorphic fibration structure. This will have an impact on the structure of the group of symmetries, which will typically appear as a non-trivial extension
\begin{equation}
\label{eq:extcxgaugeabs}
  1\to \GGG \lra{} \widetilde{\GGG} \lra \HHH \to 1,
\end{equation}
of a (real) subgroup $\HHH$ of the holomorphic automorphisms of the base $\JJJ$ by a subgroup of K\"ahler isometries $\GGG$ of the typical fibre $\DDD$. Previous work on this type of reductions, in the context of gauge theory, can be found in \cite{AGG1,AGG3,AGGP,AGGPY}.

%%%%%%%%%%%%%%%%%%%%%%%%%%%%%%%%%%%%%%%%%%%%%%%%%%%%%%%%%%%%%%%%%%%%%%%
\subsection{The K\"ahler connection and the squared norm of the Higgs field}\label{ssec:potential}
%%%%%%%%%%%%%%%%%%%%%%%%%%%%%%%%%%%%%%%%%%%%%%%%%%%%%%%%%%%%%%%%%%%%%%%

We fix a smooth oriented compact surface $\Sigma$. We also consider a connected semisimple complex Lie group $G$, with Lie algebra $\glie$, and fix an antiholomorphic involution $\tau$ of $G$ defining a maximal compact subgroup $K:=G^\tau\subset G$, with Lie algebra $\liek$.

Let $E_G$ be a smooth principal $G$-bundle over $\Sigma$.
Let $\DDD$ be the space of connections on $E_G$, equipped with the constant complex structure \eqref{eq:bfJ}, and $\JJJ$ the space of complex structures on $\Sigma$ that are compatible with the given orientation. 
Consider the space
\begin{equation}\label{eq:configuration.1}
\XXX = \JJJ\times \DDD,
\end{equation}
endowed with the product complex structure
$$
\mathbb{J}(\dot J, \dot D) = (J\dot J,\mathbf{J} \dot D).
$$
Here we identify $T_J\JJJ$ with the space of endomorphisms $\dot J \colon T\Sigma \to T\Sigma$ such that $\dot J J = - J \dot J$.
In particular, the map $\pi_1 \colon \XXX \to \JJJ$ is holomorphic. By construction, a natural structure of K\"ahler fibration $\widehat{\boldsymbol{\omega}}_{\mathbf{J}}$ on $\XXX$ over $\JJJ$ is given by the  $\JJJ$-dependent symplectic structure $\omega_{\mathbf{J}}(J)$ on $\DDD$, defined in \eqref{eq:omehaJ}. We should emphasize that, even though the holomorphic fibration structure on $\XXX$ is trivial, $\widehat{\boldsymbol{\omega}}_{\mathbf{J}}$ has a non-trivial dependence on the base $\JJJ$.

In order to give a more explicit description of $\widehat{\boldsymbol{\omega}}_{\mathbf{J}}$, we consider a reduction $h\in \Omega^0(E_G(G/K))$ of the structure group of $E_G$ to $K$, and let $E_K$ be the corresponding principal $K$-bundle. From the decomposition
$\lieg=\liek\oplus i\liek$ one has the identification $\DDD \cong \AAA \times \Omega^1(\Sigma,E_K(\liek))$ in \eqref{eq:TADDD}, which induces a biholomorphism
\begin{equation}\label{eq:configuration.2}
\XXX \cong \JJJ\times \AAA\times \Omega^1(\Sigma,E_K(\liek)),
\end{equation}
where the right-hand side is endowed with the product complex structure
\begin{equation}\label{eq:uJ}
\mathbb{J}(\dot J, a , \psi) = (J\dot J,-\dot \psi, a).
\end{equation}
Consider the structure of K\"ahler fibration $\widehat{\boldsymbol{\omega}}_{\mathbf{J}}$ on $\XXX \to \JJJ$ defined above, which in the present setup can be described explicitly by
\begin{equation}\label{eq:hatomehaJ}
\widehat{\boldsymbol{\omega}}_{\mathbf{J}|(J,A,\psi)}((0,a_1,\dot \psi_1),(0,a_2,\dot \psi_2)) =  \int_\Sigma B(a_1 \wedge J\dot \psi_2) - \int_\Sigma B(\dot \psi_1 \wedge Ja_2).
\end{equation}
It is clear that $\widehat{\boldsymbol{\omega}}_{\mathbf{J}}$ has a non-trivial dependence on $J \in \JJJ$, and hence defines a structure of a non-trivial K\"ahler fibration on $\XXX \to \JJJ$.

The aim of this section is to prove that $(\XXX \to \JJJ,\widehat{\boldsymbol{\omega}}_{\mathbf{J}})$ admits a K\"ahler Ehresmann connection, following Theorem \ref{th:Kfib}. For this, we follow a suggestion by N. Hitchin (cf. \cite[Section 9]{hitchin:1987}) and define a real $(1,1)$-form $\boldsymbol{\sigma}^{\mathbb{J}} \in \Omega^{1,1}(\XXX,\RR)$ by
\begin{equation}\label{eq:sigmaabs}
\boldsymbol{\sigma}^{\mathbb{J}} = i\partial_{\mathbb{J}}\dbar_{\mathbb{J}}\|\psi\|^2_{L^2},
\end{equation}
where $\|\psi\|^2_{L^2}$ denotes the $J$-dependent $L^2$-norm of the `unitary Higgs field' $\psi$, namely
$$
\|\psi\|^2_{L^2} = \int_\Sigma B(\psi \wedge J \psi).
$$
In the next result, we derive an explicit formula for $\boldsymbol{\sigma}^{\mathbb{J}}$.

\begin{lemma}\label{lem:sigmaex}
%For $x = (J,A,\psi) \in \XXX$, one has
\begin{equation}\label{eq:sigmaex}
\begin{split}
\boldsymbol{\sigma}^{\mathbb{J}}_{|(J,A,\psi)}((\dot J_1,a_1,\dot \psi_1),(\dot J_2,a_2,\dot \psi_2)) & =  \int_\Sigma B(a_1 \wedge (J\dot \psi_2 - \psi(\dot J_2)))\\
& - \int_\Sigma B(a_2 \wedge (J\dot \psi_1 - \psi(\dot J_1)))\\
& - \int_\Sigma B((J\dot \psi_1 - \tfrac{1}{2} \psi(\dot J_1))\wedge \psi(\dot J_2))\\
& + \int_\Sigma B((J\dot \psi_2 - \tfrac{1}{2} \psi(\dot J_2))\wedge \psi(\dot J_1))
\end{split}
\end{equation}
\end{lemma}

\begin{proof}
We set $\nu(J,A,\psi):=\|\psi\|^2_{L^2}$ and calculate
$$
d\nu_{|(J,A,\psi)}(\dot J,a,\dot \psi) = 2\int_\Sigma B(\dot \psi \wedge J \psi) - \int_\Sigma B(\psi \wedge \psi(\dot J)).
$$
By definition of $\mathbb{J}$, we also have
\begin{align*}
d^c_{\mathbb{J}}\nu_{|(J,A,\psi)}(\dot J,a,\dot \psi) & = - 2\int_\Sigma B(a \wedge J \psi) + \int_\Sigma B(\psi \wedge \psi(J\dot J))\\
& = - 2\int_\Sigma B(a \wedge J \psi) - \int_\Sigma B(J \psi \wedge \psi(\dot J)).
\end{align*}
The statement follows now from $dd^c_{\mathbb{J}}\nu = 2i\partial_{\mathbb{J}}\dbar_{\mathbb{J}}\nu$.
\end{proof}

We will need the following technical lemma.

\begin{lemma}\label{l:magicid}
For $C \in \End T\Sigma$ and $\psi_1,\psi_2 \in \Omega^1(\Sigma,E_K(\liek))$, the following pointwise identity holds:
$$
B(\psi_1(C) \wedge \psi_2) + B(\psi_1 \wedge \psi_2(C)) = 0.
$$
\end{lemma}

\begin{proof}
By direct calculation at a point $x\in\Sigma$, if $C=v\otimes\gamma\in T_x\Sigma\otimes T^*_x\Sigma$, we see by dimensional reasons that 
\begin{align*}
0 & = \gamma \wedge i_v B(\psi_1 \wedge \psi_2)\\
& = \gamma \wedge (B(\psi_1(v),\psi_2) - B(\psi_1,\psi_2(v))) = B(\psi_1(v)\gamma \wedge \psi_2) + B(\psi_1 \wedge \psi_2(v) \gamma).
\qedhere
\end{align*}
\end{proof}

Our next result proves the existence of a natural K\"ahler Ehresmann connection on the K\"ahler fibration $\XXX \to \JJJ$, whose coupling form is $\boldsymbol{\sigma}^{\mathbb{J}}$ and whose curvature is essentially given by the `unitary Higgs field' $\psi$.

\begin{proposition}\label{p:existenceGamma}
The K\"ahler fibration $(\XXX \to \JJJ,\widehat{\boldsymbol{\omega}}_{\mathbf{J}})$ admits a K\"ahler Ehresmann connection $\Gamma^{\mathbb{J}} \colon T\XXX \to V\XXX$, with horizontal bundle given by
\begin{equation*}
H^{\mathbb{J}} = \{v \in T\XXX \; | \; i_v\boldsymbol{\sigma}^{\mathbb{J}}_{|V\XXX} = 0\},
\end{equation*}
where $\boldsymbol{\sigma}^{\mathbb{J}}$ is defined by \eqref{eq:sigmaabs}. More explicitly,
\begin{equation}\label{eq:Hexp}
H^{\mathbb{J}}_{|(J,A,\psi)} = \{(\dot J, \psi(\dot J), (J\psi)(\dot J)) \; | \; \dot J \in T_J\JJJ\}.
\end{equation}
Furthermore, $H^{\mathbb{J}}$ is preserved by $\mathbb{J}$ and the curvature $F_{\mathbb{J}} := F_{\Gamma^{\mathbb{J}}} \in \Omega^{2}(\XXX,V\XXX)$ of $\Gamma^{\mathbb{J}}$ is of type $(1,1)$, and given explicitly by
\begin{equation}\label{eq:FGamma}
(F_{\mathbb{J}})_{|(J,A,\psi)}(v_1,v_2) =  (0,\psi([\dot J_2,\dot J_1]),0),
\end{equation}
for any pair of horizontal vector fields $v_1,v_2 \in H^{\mathbb{J}}$ covering $\dot J_1,\dot J_2$, respectively.
\end{proposition}

\begin{proof}
The existence of $\Gamma$ follows from Theorem \ref{th:Kfib} combined with Lemma \ref{lem:sigmaex}, which implies that
\begin{equation}\label{eq:sigmaex.1}
\begin{split}
\boldsymbol{\sigma}^{\mathbb{J}}((0,a_1,\dot \psi_1),(0,a_2,\dot \psi_2)) & =  \int_\Sigma B(a_1 \wedge J\dot \psi_2) + \int_\Sigma B(J^2\dot \psi_1 \wedge Ja_2)\\
& = \widehat{\boldsymbol{\omega}}_{\mathbf{J}}((0,a_1,\dot \psi_1),(0,a_2,\dot \psi_2)),
\end{split}
\end{equation}
where we omit the evaluation at the point $(J,A,\psi)$ for simplicity in the notation. Contracting now $\boldsymbol{\sigma}^{\mathbb{J}}$ with a vertical vector field
\begin{equation*}\label{eq:FGamma.1}
\begin{split}
\boldsymbol{\sigma}^{\mathbb{J}}((0,a_1,\dot \psi_1),(\dot J_2,a_2,\dot \psi_2)) & =  \int_\Sigma B(a_1 \wedge (J\dot \psi_2 - \psi(\dot J_2)))
 - \int_\Sigma B(J\dot \psi_1 \wedge (\psi(\dot J_2) - a_2)),
\end{split}
\end{equation*}
and therefore $(\dot J_2,a_2,\dot \psi_2) \in H^{\mathbb{J}}$ if and only if
$$
a_2 = \psi(\dot J_2), \qquad \dot \psi_2 = (J\psi)(\dot J_2).
$$
Since $\boldsymbol{\sigma}^{\mathbb{J}}$ is of type $(1,1)$ and $\mathbb{J}$ preserves $V\XXX$, it follows that $\mathbb{J}$ preserves $H^{\mathbb{J}}$. Finally, evaluating $\boldsymbol{\sigma}^{\mathbb{J}}$ in a pair of horizontal vector fields $v_j = (\dot J_j,\psi(\dot J_j),(J\psi)(\dot J_j))$, we have
\begin{equation*}
\begin{split}
\boldsymbol{\sigma}^{\mathbb{J}}(v_1,v_2) % & =  \int_\Sigma B(\psi(\dot J_1)\wedge (J((J\psi)(\dot J_2)) - \psi(\dot J_2)))\\ & - \int_\Sigma B(\psi(\dot J_2) \wedge (J((J\psi)(\dot J_1)) - \psi(\dot J_1)))\\
% & - \int_\Sigma B\left(\left(J((J\psi)(\dot J_1)) - \tfrac{1}{2} \psi(\dot J_1)\right)\wedge \psi(\dot J_2) \right)\\
% & + \int_\Sigma B\left(\left(J((J\psi)(\dot J_2)) - \tfrac{1}{2} \psi(\dot J_2)\right)\wedge \psi(\dot J_1) \right)\\
& = - \int_\Sigma B((\psi(\dot J_1) - \tfrac{1}{2} \psi(\dot J_1))\wedge \psi(\dot J_2)) + \int_\Sigma B((\psi(\dot J_2) - \tfrac{1}{2} \psi(\dot J_2))\wedge \psi(\dot J_1))\\
& = - \int_\Sigma B(\psi(\dot J_1)\wedge \psi(\dot J_2)),
\end{split}
\end{equation*}
where we have used that
$$
J((J\psi)(\dot J)) = - J(\psi(J\dot J)) = \psi(J\dot J J) = \psi(\dot J).
$$
Consequently,
$$
d(\boldsymbol{\sigma}^{\mathbb{J}}(v_1,v_2))((0,a,\dot \psi)) = - \int_\Sigma B(\dot \psi(\dot J_1)\wedge \psi(\dot J_2)) - \int_\Sigma B(\psi(\dot J_1)\wedge \dot \psi(\dot J_2)).
$$
On the other hand, applying Lemma~\ref{l:magicid}, we obtain 
\begin{align*}
\widehat{\boldsymbol{\omega}}_{\mathbf{J}}((0,(J\psi)([\dot J_1,\dot J_2]),0),(0,a,\dot \psi)) & =  \int_\Sigma B((J\psi)([\dot J_1,\dot J_2]) \wedge J\dot \psi)\\
& =  \int_\Sigma B(J(\psi([\dot J_1,\dot J_2])) \wedge J\dot \psi)\\
& =  \int_\Sigma B(\psi([\dot J_1,\dot J_2]) \wedge \dot \psi)\\
& =  - \int_\Sigma B(\psi(\dot J_1) \wedge \dot \psi(\dot J_2)) + \int_\Sigma B(\psi(\dot J_2) \wedge \dot \psi(\dot J_1))\\
& = d(\boldsymbol{\sigma}^{\mathbb{J}}(v_1,v_2))(0,a,\dot \psi),
\end{align*}
which proves the last part of the statement, by Proposition \ref{p:Kfibide}.
\end{proof}

To finish this section, we provide a formula for the coupling form $\boldsymbol{\sigma}^{\mathbb{J}}$ adapted to its associated connection $\Gamma^{\mathbb{J}}$. Note that the vertical projection with respect to $\Gamma^{\mathbb{J}}$ is explicitly given by
$$
\Gamma^{\mathbb{J}}(\dot J,a,\dot \psi) = (0,a - \psi(\dot J),\dot \psi - (J\psi)(\dot J)).
$$
We will also consider the symmetric tensor on $\XXX$ defined by
$$
\mathbf{g}_{\mathbb{J}} := \boldsymbol{\sigma}^{\mathbb{J}}(,\mathbb{J}).
$$
By construction, this coincides with the flat hyperk\"ahler metric \eqref{eq:gHK} along the fibres of $\XXX \to \JJJ$. As we will see shortly, $\mathbf{g}_{\mathbb{J}}$ is negative semi-definite along the horizontal directions of the connection $\Gamma^{\mathbb{J}}$. This reveals difficulties in the fundamental question of constructing a positive-definite K\"ahler metric on the K\"ahler fibration $(\XXX \to \JJJ,\widehat{\boldsymbol{\omega}}_{\mathbf{J}})$. %Further insights will be provided in Section \ref{section-ghiggs}, via the \emph{universal Higgs bundle moduli space}.

\begin{corollary}\label{cor:negative}
For any tangent vectors $v_j = (\dot J_j,a_j,\dot \psi_j), v = (\dot J,a,\dot \psi) \in T_{(J,A,\psi)} \XXX$ one has
\begin{equation}\label{eq:sigmaexadapted}
\begin{split}
\boldsymbol{\sigma}^{\mathbb{J}}(v_1,v_2) & = \int_\Sigma B((a_1 - \psi(\dot J_1)) \wedge J(\dot \psi_2 - (J\psi)(\dot J_2))) \\
& - \int_\Sigma B((\dot \psi_1 - (J\psi)(\dot J_1)) \wedge J(a_2 - \psi(\dot J_2)))\\
& - \int_\Sigma B(\psi(\dot J_1)\wedge \psi(\dot J_2)),\\
\mathbf{g}_{\mathbb{J}}(v,v) & = \int_\Sigma B((a - \psi(\dot J)) \wedge J(a - \psi(\dot J))) \\
& + \int_\Sigma B((\dot \psi - (J\psi)(\dot J)) \wedge J(\dot \psi - (J\psi)(\dot J)))\\
& - \int_\Sigma B(\psi(\dot J)\wedge J (\psi(\dot J))).
\end{split}
\end{equation}
In particular, given a horizontal vector field $v \in H^{\mathbb{J}}$ at $(J,A,\psi)$, covering $\dot J \in T_J \JJJ$, one has
\begin{equation}\label{eq:gJhor}
\begin{split}
\mathbf{g}_{\mathbb{J}}(v,v) %& = \int_\Sigma B(\psi(\dot J))\wedge \psi(\dot J J))\\
%& 
= - \int_\Sigma B(\psi(\dot J)\wedge J (\psi(\dot J))).
\end{split}
\end{equation}
Consequently, $\mathbf{g}_{\mathbb{J}}$ is negative semi-definite along the horizontal directions of $\Gamma^{\mathbb{J}}$.
\end{corollary}

%%%%%%%%%%%%%%%%%%%%%%%%%%%%%%%%%%%%%%%%%%%%%%%%%%%%%%%%%%%%%%%%%%%%
\subsection{Minimal coupling and the Hamiltonian action}
\label{coupling}
%%%%%%%%%%%%%%%%%%%%%%%%%%%%%%%%%%%%%%%%%%%%%%%%%%%%%%%%%%%%%%%%%%%%

Let $(\XXX \to \JJJ,\widehat{\boldsymbol{\omega}}_{\mathbf{J}})$ be the K\"ahler fibration defined in Section \ref{ssec:potential}. 
In this section we apply the minimal-coupling construction to the connection $\Gamma^{\mathbb{J}}$ introduced in Proposition~\ref{p:existenceGamma}.
We will see that the corresponding (pre)symplectic structures on $\XXX$ admit a Hamiltonian action by a suitable \emph{extended gauge group} $\widetilde{\KKK}$,
%subgroup $\widetilde{\KKK} \subset \widetilde{\GGG}$, 
naturally associated to a choice of symplectic structure $\omega$ on $\Sigma$. This will allow us, in Section \ref{Umoduli}, to construct pseudo-K\"ahler structures on the universal moduli space of solutions of the harmonicity equations \eqref{harmonicity} over the Teichm\"uller space. %The existence of a compatible K\"ahler fibration structure is a delicate question, related to Corollary \ref{cor:negative}, which we address in Section ...

We start by introducing the K\"ahler structure that we will consider, up to sign, on the base of the fibration $\XXX \to \JJJ$. Following Donaldson \cite{Do2} and Fujiki~\cite{Fj}, we fix a symplectic (volume) form $\omega$ on $\Sigma$ compatible with the given orientation. The space $\JJJ$ of complex structures $J$ on $\Sigma$ compatible with the orientation (and hence with
$\omega$) is an infinite-dimensional K\"ahler manifold, with complex
structure $\mathbb{J}_\JJJ \colon T_J\JJJ \to T_J\JJJ$ and K\"ahler form
$\boldsymbol{\omega}_\JJJ$ given, respectively, by
\begin{equation}
\label{eq:SympJ}
\mathbb{J}_\JJJ \dot J := J\dot J \;\; \text{ and }\;\;
\boldsymbol{\omega}_\JJJ (\dot J_1,\dot J_2) := \frac{1}{2}\int_{\Sigma}\tr(J\dot J_1\dot J_2) \omega,
\end{equation}
for $\dot J_1,\dot J_2 \in T_J\JJJ$. Note that, for dimensional reasons, any $\dot J \in T_J\JJJ$, regarded as an endomorphism $\dot J \colon T\Sigma \to T\Sigma$, is symmetric with respect to the induced metric $g = \omega(\cdot,J\cdot)$.

%Here we identify $T_J\JJJ$ with the space of endomorphisms $\dot J \colon T\Sigma \to T\Sigma$ such that $\dot J$ is symmetric with respect to the induced metric $g = \omega(\cdot,J\cdot)$ and satisfies $\dot J J = - J \dot J$.

Let $(\XXX \to \JJJ,\widehat{\boldsymbol{\omega}}_{\mathbf{J}})$ be the K\"ahler fibration defined in Section \ref{ssec:potential}. Given a real `coupling constant' $\alpha > 0$ and $\varepsilon \in \{-1,1\}$, the family of minimal coupling symplectic structures of our interest is defined by
\begin{equation}\label{eq:uomega}
\boldsymbol{\omega}^{\mathbb{J}}_{\alpha,\varepsilon} = \varepsilon\boldsymbol{\omega}_\JJJ + \alpha \boldsymbol{\sigma}^{\mathbb{J}},
\end{equation}
where $\boldsymbol{\sigma}^{\mathbb{J}}$ is the exact $(1,1)$-form in Lemma \ref{lem:sigmaex}. By construction, $\boldsymbol{\omega}^{\mathbb{J}}_{\alpha,\varepsilon}$ is closed and of type $(1,1)$ with respect to the complex structure $\mathbb{J}$. Furthermore, along the fibres of $\XXX \to \JJJ$ the $2$-form $\boldsymbol{\omega}^{\mathbb{J}}_{\alpha,\varepsilon}$ restricts to the K\"ahler structure $\alpha \widehat{\boldsymbol{\omega}}$. Consider the associated symmetric tensor 
\begin{equation}\label{eq:ug}
\mathbf{g}^{\mathbb{J}}_{\alpha,\varepsilon} := \boldsymbol{\omega}^{\mathbb{J}}_{\alpha,\varepsilon} (,\mathbb{J}) = \varepsilon \mathbf{g}_{\JJJ} + \alpha \mathbf{g}_{\mathbb{J}}.
\end{equation}
Applying Corollary \ref{cor:negative}, $\boldsymbol{\omega}^{\mathbb{J}}_{\alpha,-1}$ is negative-definite along the horizontal subspace $H^{\mathbb{J}} \subset T \XXX$. The analysis for the case $\varepsilon = 1$ is much more subtle, as we can see from the next result.

\begin{lemma}\label{lemma:indef}
Let $v \in H^{\mathbb{J}}$ be a horizontal vector field at $(J,A,\psi)$, covering $\dot J \in T_J \JJJ$. 
Then 
\begin{equation}\label{eq:gJhor.1}
\begin{split}
\mathbf{g}^{\mathbb{J}}_{\alpha,\varepsilon}(v,v) = \frac{\varepsilon}{2}\int_{\Sigma}\tr(\dot J\dot J) \omega - \alpha \int_\Sigma B\left(\psi(\dot J)\wedge J (\psi(\dot J))\right).
\end{split}
\end{equation}
Consequently, for any $\alpha>0$:
\begin{enumerate}

\item If $\varepsilon = -1$, $\mathbf{g}^{\mathbb{J}}_{\alpha,\varepsilon}$ is negative definite along $H^{\mathbb{J}}$. In particular, $\boldsymbol{\omega}_\alpha^{-1}$ is a non-degenerate symplectic structure.

\item If $\varepsilon = 1$ and $\psi\neq 0$, $\mathbf{g}^{\mathbb{J}}_{\alpha,\varepsilon}$ changes signature along the line $(J,A,\lambda\psi) \in \XXX$, for $\lambda \in \RR$.
\end{enumerate}
\end{lemma}

\begin{proof}
The case $\varepsilon = -1$ is a direct consequence of Corollary \ref{cor:negative}. For the case $\varepsilon = 1$, let $(J,A,\psi) \in \XXX$. Then, for any $\dot J \in T_J\JJJ$, we have pointwise identities
\begin{align*}
\tr(\dot J\dot J) = C |\dot J|^2, \qquad \Lambda_\omega B\left(\psi(\dot J)\wedge J (\psi(\dot J))\right) = C' |\psi|^2|\dot J|^2, 
\end{align*}
for suitable positive constants $C,C' \in \RR$, where the $|\cdot|$ denotes tensorial norms with respect to $B$ and $g = \omega(,J)$. Hence, for $\psi \neq 0$ and
$$
\lambda^2_0 = \frac{C}{\alpha C'|\psi|^2},
$$
the tensor $\mathbf{g}^{\mathbb{J}}_{\alpha,\varepsilon}$ is:

\begin{itemize}

\item positive definite at points $(J,A,\lambda\psi) \in \XXX$, for $|\lambda| < \lambda_0$,

\item positive definite along vertical directions and negative definite along the horizontal bundle $H^{\mathbb{J}}$, at points $(J,A,\lambda\psi) \in \XXX$, for $|\lambda| > \lambda_0$,

\item null at the horizontal bundle $H^{\mathbb{J}}$ at the points $(J,A,\lambda_0\psi) \in \XXX$.

\end{itemize}
The result follows now from the previous cases.
\end{proof}

Our next goal is to prove that $(\XXX,\boldsymbol{\omega}^{\mathbb{J}}_{\alpha,\varepsilon})$ admits a Hamiltonian action by a suitable \emph{extended gauge group} $\widetilde{\KKK}$, determined by $\omega$ and the reduction $E_K \subset E_G$. Consider the group  $\HHH$ of Hamiltonian symplectomorphisms of $(\Sigma,\omega)$. The natural group of symmetries of our theory is the (Hamiltonian) \emph{extended gauge group} $\widetilde{\KKK}$.  By definition, $\widetilde{\KKK}$ is the group of automorphisms of $E_K$ which cover elements of the group $\HHH$. There is a canonical short exact sequence of Lie groups
\begin{equation}
\label{eq:coupling-term-moment-map-1}
  1\to \KKK \lra{} \widetilde{\KKK} \stackrel{p}\lra \HHH \to 1,
\end{equation}
where $p$ maps each $g\in\widetilde{\KKK}$ to the Hamiltonian symplectomorphism
$p(g)\in\HHH$ that it covers, and so its kernel $\KKK$ is the 
gauge group of $E_K$, that is, the normal subgroup of $\widetilde{\KKK}$
consisting of automorphisms of $E_K$ covering the identity map on $\Sigma$.

As in \cite[Section 2.2]{AGG1}, the group $\widetilde{\KKK}$ acts on $(\XXX,\mathbb{J},\boldsymbol{\omega}^{\mathbb{J}}_{\alpha,\varepsilon})$, preserving both $\mathbb{J}$ and $\boldsymbol{\omega}^{\mathbb{J}}_{\alpha,\varepsilon}$, and covering the $\HHH$-action on $\JJJ$ given by push-forward.
This last action preserves the K\"ahler structure on $\JJJ$ and is moreover Hamiltonian, with equivariant moment map $\mu\colon \JJJ\to(\Lie\HHH)^*$ given by
\begin{equation}
\label{eq:scmom}
\langle \mu(J), \eta_f\rangle = -\int_\Sigma f S_J \omega,
\end{equation}
where $f \in C^{\infty}_0(\Sigma)$ is identified with an element $\eta_f$ in $\Lie\HHH$, and $S_J$ is the scalar curvature of the metric $g = \omega(,J)$ determined by $\omega$ and $J$. 
Here, $C^{\infty}_0(\Sigma)$ denotes the space of smooth functions on $\Sigma$ such that $\int_\Sigma f \omega = 0$. 
This fact was established by Donaldson \cite{Do2} and Fujiki~\cite{Fj} independently for symplectic manifolds of arbitrary dimension, and it appears to have been known to Quillen in the case of surfaces.

\begin{proposition}\label{p:mmap}
The action of $\widetilde{\KKK}$ on $(\XXX,\boldsymbol{\omega}^{\mathbb{J}}_{\alpha,\varepsilon})$ is Hamiltonian, with equivariant moment map given by
\begin{equation}\label{eq:mmap}
\begin{split}
\langle \mu_{\widetilde{\KKK}}^{\mathbb{J}}(J,A,\psi), \zeta \rangle & = \alpha \int_\Sigma B(A \zeta + 2\psi(Jp\zeta),d_A(J\psi))\\
& - \int_\Sigma f \left(\varepsilon S_J - \alpha \Lambda_\omega d\left(B\left(\Lambda_\omega F_A,J\psi\right)\right)\right) \omega,
\end{split}
\end{equation}
where $A\zeta\in\Omega^0(\Sigma,E_K(\liek))$ denotes the vertical part of the $K$-equivariant vector field $\zeta$ on the total space of $E_K$, with respect to the connection $A$.
\end{proposition}

\begin{proof}
Since $\sigma = \tfrac{1}{2} dd^c_{\mathbb{J}}\nu$, with $\nu(J,A,\psi) : = \|\psi\|^2_{L^2}$ (see Lemma \ref{lem:sigmaex}), and because $d^c_{\mathbb{J}}\nu$ is preserved by the $\widetilde{\KKK}$-action, there exists an equivariant moment map given by 
$$
\langle \mu_{\widetilde{\KKK}}^{\mathbb{J}}(J,A,\psi), \zeta \rangle = - \varepsilon \int_\Sigma f S_J \omega - \frac{\alpha}{2} d^c_{\mathbb{J}}\nu(\zeta \cdot (J,A,\psi)),
$$
where $\zeta \cdot (J,A,\psi)$ denotes the infinitesimal action of $\zeta$ on $(J,A,\psi)$. 
Using the complex connection $D = A + i \psi$, it follows that
\begin{align*}
\zeta \cdot D & = -d_D(D\zeta) - i_{y}F_D\\
& = - \left(d_A(A\zeta) + i[\psi,A\zeta] + i d_A(\psi(y)) - [\psi,\psi(y)]\right) - i_{y}\left(F_A - \tfrac{1}{2}[\psi \wedge \psi] + i d_A\psi\right)\\
& = - d_A(A\zeta) - i_{y}F_A + [\psi,\psi(y)] + \tfrac{1}{2}i_y[\psi \wedge \psi] - i\left( d_A(\psi(y)) + i_{y}(d_A\psi) + [\psi,A\zeta]\right)\\
& = - d_A(A\zeta) - i_{y}F_A - i\left( d_A(\psi(y)) + i_{y}(d_A\psi) + [\psi,A\zeta]\right),
\end{align*}
where $y := p\zeta$, and therefore, by the proof of Lemma \ref{lem:sigmaex} and integration by parts,
\begin{align*}
d^c_{\mathbb{J}}\nu(\zeta \cdot (J,A,\psi)) & = 2\int_\Sigma B\left(\left(d_A(A\zeta) + i_{y}F_A\right) \wedge J \psi\right) - \int_\Sigma B(J \psi \wedge \psi(L_{y} J))\\
& = - 2 \int_\Sigma \left(B (A \zeta, d_A(J\psi)) + B\left(F_A,J\psi(y)\right) + 2B (\psi(Jy),d_A(J\psi))\right)\\
& = - 2 \int_\Sigma B (A \zeta + 2\psi(Jy), d_A(J\psi)) - 2 \int_\Sigma f d\left(B\left(\Lambda_\omega F_A,J\psi\right)\right).
\end{align*}
For the second equality, we use the identity $L_yJ = 2 i\dbar y^{1,0} - 2 i\partial y^{0,1}$ for any vector field $y$ on $\Sigma$, and therefore
\begin{align*}
B(J \psi \wedge \psi(L_y J)) & = 2B(\psi^{1,0} \wedge \psi(\dbar y^{1,0})) + \textnormal{c.c.} \\
& = 2B(\psi^{1,0} \wedge (i_{y^{1,0}}(\dbar_A\psi^{1,0}) + \dbar_A(i_{y^{1,0}}\psi))) + \textnormal{c.c.} \\
& = - 2\dbar (B(\psi^{1,0}, \psi(y^{1,0}))) + 4 B(\dbar_A\psi^{1,0},\psi(y^{1,0})) + \textnormal{c.c.},
\end{align*}
where as usual ``c.c.'' stands for complex conjugate.
\end{proof}

%Unlike the complex moment map in Proposition \ref{Flat-exreduction}, 
The $\widetilde{\KKK}$-action produces a `coupling term' on the base $\Sigma$ (cf. Remark \ref{rem:couplingterm}), which interacts with the scalar curvature of the metric $g = \omega(,J)$. 
In particular, as in \cite{AGG1}, the zeros of the moment map $\mu_{\widetilde{\KKK}}^{\mathbb{J}}$ correspond to the solutions of the coupled system of equations
\begin{equation}
\begin{split}
d_A^*\psi & = 0,\\
\varepsilon S_g - \alpha * d\left(B\left(\Lambda_\omega F_A,*\psi\right)\right)& = \varepsilon \frac{2\pi \chi(\Sigma)}{V},
\end{split}
\end{equation}
where $\chi(\Sigma)$ is the Euler characteristic of $\Sigma$ and $V = \int_\Sigma \omega$ is the total volume. In the next section we combine these equations with the flatness condition $F_D = 0$ for the complex connection to introduce the \emph{universal Hitchin moduli space}.

%%%%%%%%%%%%%%%%%%%%%%%%%%%%%%%%%%%%%%%%%%%%%%%%%%%%%%%%%%%%%%%%%%%%%%%
\section{Universal moduli spaces}
\label{Umoduli}
%%%%%%%%%%%%%%%%%%%%%%%%%%%%%%%%%%%%%%%%%%%%%%%%%%%%%%%%%%%%%%%%%%%%%%%

%%%%%%%%%%%%%%%%%%%%%%%%%%%%%%%%%%%%%%%%%%%%%%%%%
\subsection{Universal moduli space of flat $\boldsymbol{G}$-connections}\label{ssec:UFlat}
%%%%%%%%%%%%%%%%%%%%%%%%%%%%%%%%%%%%%%%%%%%%%%%%%

We fix a smooth oriented compact surface $\Sigma$. We also consider a connected semisimple complex Lie group $G$, with Lie algebra $\glie$, and fix an antiholomorphic involution $\tau$ of $G$ defining  a maximal compact subgroup $K:=G^\tau\subset G$, with Lie algebra $\liek$.

Let $E_G$ be a smooth principal $G$-bundle over $\Sigma$. Let $\DDD$ be the space of connections on $E_G$, equipped with the constant complex structure $\mathbf{J}$ given by~\eqref{eq:bfJ}. 
Let $\JJJ$ be the space of complex structures on $\Sigma$ compatible with the given orientation. 
Consider the K\"ahler fibration $(\XXX \to \JJJ,\widehat{\boldsymbol{\omega}}_{\mathbf{J}})$ defined in Section \ref{ssec:potential}. 
Our aim in this section is to investigate the existence of a K\"ahler structure on $\XXX$ and on its quotients, by combining the minimal coupling associated with the connection $\Gamma^{\mathbb{J}}$ constructed in Proposition~\ref{p:existenceGamma}, as discussed in Section~\ref{coupling}, with symplectic reduction. 
Difficulties arise from Corollary~\ref{cor:negative}, which shows that the associated symmetric tensor $\mathbf{g}_{\mathbb{J}} := \sigma(,\mathbb{J})$ is negative semi-definite along the horizontal directions.

In order to have some control in our construction, we need to impose the integrability condition $F_D = 0$. For this, we first construct a universal moduli space of flat $G$-connections, with a holomorphic fibration structure over Teichm\"uller space. The first step is to show that the flatness condition for $D \in \DDD$ is compatible with the natural group of symmetries in the present context: there is an extended complex gauge group of symmetries, acting on $\XXX$ and preserving the complex structure, defined by an extension
\begin{equation}
\label{eq:extcxgauge}
  1\to \GGG \lra{} \widetilde{\GGG} \stackrel{p}\lra \operatorname{Diff}_0(\Sigma) \to 1.
\end{equation}
Here, $\widetilde{\GGG}$ is the subgroup of $G$-equivariant diffeomorphisms of $E_G$ which project to $\operatorname{Diff}_0(\Sigma)$, the component of the identity in the group of diffeomorphisms of $\Sigma$, and $p$ maps each $g\in\widetilde{\GGG}$ into the diffeomorphism
$p(g)\in\operatorname{Diff}_0(\Sigma)$ that it covers. Note that the kernel of $p$ is the 
gauge group of $E_G$, that is, the normal subgroup of $\widetilde{\GGG}$
consisting of  automorphisms of $E_G$ covering the identity map on $\Sigma$. Note also that, for any choice of symplectic structure $\omega$ on $\Sigma$ and reduction $h\in \Omega^0(E_G(G/K))$ of structure group of $E_G$ to $K$, there is a natural group homomorphism $\widetilde{\KKK} \subset \widetilde{\GGG}$ induced by the inclusion map.

More explicitly, the $\widetilde{\GGG}$-action on $\XXX$ is given by
$$
g(J,D) = (p(g)_*J,gD),
$$
which is compatible with the holomorphic fibration structure, in the sense that
$$
\pi_1(g(J,D)) = p(g)_*J.
$$ 
In other words, the extension \eqref{eq:extcxgauge} is given by restriction of the natural map from holomorphic automorphisms of $\XXX$ which preserve the fibration structure, to  holomorphic automorphisms of the base $\JJJ$. It is important to observe that even though $\operatorname{Diff}_0(\Sigma)$ acts by holomorphic automorphisms of the base $\JJJ$, this is a \emph{real} infinite-dimensional Lie group, which does not admit a complexification (see e.g.  \cite{Do3}).

The holomorphic symplectic structure $\Omega_{\mathbf{J}}$ on $\DDD$ is constant on $\JJJ$ (see \eqref{eq:bfJ}), and hence it makes sense to consider the holomorphic presymplectic structure on $\XXX$, defined by pull-back
$$
\Omega_{\mathbb{J}} = \pi_2^*\Omega_{\mathbf{J}}.
$$
As in the previous section, $\Omega_{\mathbb{J}}$ induces a structure of holomorphic symplectic fibration on $\XXX \to \JJJ$ and a holomorphic Ehresmann connection, which in this case is simply the trivial connection. We will return to this structure when studying the \emph{universal Higgs bundle moduli space} in Section \ref{section-ghiggs}. The following result is an immediate consequence of the proof of \cite[Proposition 1.6]{AGG1}, but we sketch the proof in the present setup for the convenience of the reader.

\begin{proposition}\label{Flat-exreduction}
The action of $\widetilde{\GGG}$ on $(\XXX,\Omega_{\mathbb{J}})$ is Hamiltonian, with equivariant moment map given by
$$
\langle \mu_{\widetilde{\GGG}}(J,D), \zeta \rangle = \int_\Sigma B(F_D, D \zeta),
$$
where $D \zeta \in \Omega^0(\Sigma,E_G(\lieg))$ denotes the vertical part of the $G$-equivariant vector field $\zeta$ on the total space of $E_G$, with respect to the connection $D$.
\end{proposition}

\begin{proof}
Taking variations on the formula for $\mu_{\widetilde{\GGG}}$ and integrating by parts, we have
\begin{align*}
\langle d\mu_{\widetilde{\GGG}}(\dot J,\dot D), \zeta \rangle & = \int_\Sigma B(d_D(\dot D) , D \zeta) + \int_\Sigma B(F_D, i_{p\zeta}\dot D)\\
& = \int_\Sigma B(\dot D\wedge  d_D(D \zeta)) - \int_\Sigma B(i_{p\zeta} F_D \wedge  \dot D)\\
& = - \int_\Sigma B((d_D(D \zeta)+ i_{p\zeta} F_D) \wedge \dot D).
\end{align*}
The result now follows because the infinitesimal action of $\zeta$ on $D$ is given by 
\[
\zeta \cdot D = -d_D(D\zeta) - i_{p\zeta}F_D.
\qedhere
\]
\end{proof}

It is remarkable that, unlike in the moment-map calculation in Proposition \ref{p:mmap}, in the present setup the extension of the complex gauge group by diffeomorphisms does not produce any additional `coupling term' in the base $\Sigma$ (cf. Remark \ref{rem:couplingterm}). One can interpret this fact as a functorial property of flat connections under the action of diffeomorphisms on a surface. This allows us to define the universal moduli space of flat $G$-connections, as follows, and to impose the flatness condition for $D$ as an \emph{integrability condition} in the next sections.

Let us denote $\DDD^* \subset \DDD$ the complex subspace of reductive connections, and set 
$$
\XXX^* = \JJJ \times \DDD^*. 
$$
It is not difficult to see that $\DDD^*$, and hence $\XXX^*$, is preserved by the $\widetilde{\GGG}$-action. Therefore, considering the pull-back of $\Omega_{\mathbb{J}}$ to $\XXX^*$, there is an induced moment map away from the singularities of $\XXX^*$, which, by abuse of notation, we denote simply by
$$
\mu_{\widetilde{\GGG}|\XXX^*} \colon \XXX^* \lra \operatorname{Lie} \widetilde{\GGG}^*.
$$
From the previous discussion, it is natural to define the \emph{universal moduli space of flat $G$-connections} on $E_G$ as the complex symplectic quotient
$$
\Uf(G):=\{(J,D) \; \in \; \XXX^* \;\;  \mbox{with}\;\;
F_D=0\}/\widetilde{\GGG} = \mu_{\widetilde{\GGG}|\XXX^*}^{-1}(0)/\widetilde{\GGG}.
$$
By construction, $\Uf(G)$ fibres over the Teichm\"uller space $\mathcal{T}$ of complex structures modulo diffeomorphisms isotopic to the identity
\begin{equation}\label{eq:FibTeichmuller}
\Uf(G) \lra \mathcal{T} := \JJJ/\operatorname{Diff}_0(\Sigma).
\end{equation}
This fibration is naturally holomorphic, and $\Omega_{\mathbb{J}}$ induces a structure of holomorphic symplectic fibration with flat Ehresmann connection.

\begin{remark}\label{rem:couplingterm}
For a $G$-bundle on a higher dimensional symplectic manifold $(M^{2n},\omega)$, consider the subgroup $\widetilde{\KKK} \subset \widetilde{\GGG}$ given by the preimage by $p$ of the group of Hamiltonian symplectomorphisms on $M$. Then the proof of \cite[Proposition 1.6]{AGG1} shows that the induced $\widetilde{\KKK}$-action on $\DDD$ is Hamiltonian, with equivariant moment map
$$
\langle \boldsymbol{\mu}_{\widetilde{\KKK}}(J,D), \zeta \rangle = \frac{1}{(n-1)!}\int_M B(F_D, D \zeta)\wedge \omega^{n-1} - \frac{1}{2(n-2)!} \int_M f B(F_D \wedge F_D) \wedge \omega^{n-2},
$$
where $f \in C^\infty(M,\RR)$ is the unique smooth function satisfying $i_{y}\omega = df$ and $\int_M f \omega^n = 0$, with $p(\zeta) = y$. In this case, there is a `coupling term' on the base corresponding to the function 
$$
-\frac{(n-1)}{2}\frac{B(F_D \wedge F_D) \wedge \omega^{n-2}}{\omega^n}.
$$
\end{remark}

\subsection{Universal moduli space of the coupled harmonic equations}
\label{section-Uhitchin}

In this section we combine the results of the previous two sections in order to define universal moduli spaces of solutions of the harmonicity equations \eqref{harmonicity}, varying over Teichm\"uller space. As in the previous sections, we consider a fixed principal $K$-bundle $E_K$ over a smooth compact oriented surface $\Sigma$ with fixed symplectic form $\omega$.

\begin{definition}
A \emph{solution of the coupled harmonic equations} with coupling
constant $\alpha >0$ and parameter $\varepsilon \in \{-1,1\}$ is a
triple $(J,A, \psi) \in \XXX \cong \JJJ\times \AAA\times
\Omega^1(\Sigma,E_K(\liek))$ that satisfies the equations
\begin{equation}\label{eq:charmonicity}
\begin{split}
F_A -\frac{1}{2}[\psi,\psi] & = 0,\\
d_A\psi & = 0,\\
d_A^\ast\psi & = 0,\\
S_g - \alpha * d\left(B\left(\Lambda_\omega F_A,*\psi\right)\right) & = \frac{2\pi \chi(\Sigma)}{V},
\end{split}
\end{equation}
where $g = \omega(,J)$ and $*$ is the corresponding Hodge star operator.
\end{definition}

As in Section \ref{ssec:UFlat}, we consider the complex subspace $\DDD^* \subset \DDD$ of reductive flat connections. Then $\XXX^* := \JJJ \times \DDD^*$ is formally a complex submanifold of $\XXX$ preserved by the $\widetilde{\KKK}$-action (see Section \ref{coupling}), and inherits a minimal coupling structure $\boldsymbol{\omega}^{\mathbb{J}}_{\alpha,\varepsilon}$ and moment map $\mu_{\widetilde{\KKK}|\XXX^*}$. We define the \emph{universal Hitchin moduli space} as the symplectic quotient
$$
\Uh(G)_\alpha^\varepsilon := \mu_{\widetilde{\KKK}|\XXX^*}^{-1}(0)/\widetilde{\KKK}.
$$
By construction, $\Uh(G)_\alpha^\varepsilon$ parametrizes solutions of the coupled harmonic equations \eqref{eq:charmonicity} with coupling constant $\alpha >0$ and parameter $\varepsilon \in \{-1,1\}$, modulo the $\widetilde{\KKK}$-action, and inherits a natural presymplectic structure on its smooth locus. This structure is furthermore symplectic in the case $\varepsilon = -1$, by Lemma \ref{lemma:indef}. There are natural maps (cf. \eqref{eq:FibTeichmuller})
$$
\Uh(G)_\alpha^\varepsilon \lra\Uf(G)\lra\mathcal{T}.
$$
Hence, in particular, $\Uh(G)_\alpha^\varepsilon$ can be regarded as a fibration over the Teichm\"uller space $\mathcal{T}$. Observe that, since the symmetric $\mathbf{g}^{\mathbb{J}}_{\alpha,\varepsilon}$ is not positive definite (see Lemma \ref{lemma:indef}), it is not obvious a priori that $\Uh(G)_\alpha^\varepsilon$ inherits a complex structure compatible with the (pre)symplectic structure induced by $\boldsymbol{\omega}^{\mathbb{J}}_{\alpha,\varepsilon}$. The main goal of this section is to study sufficient conditions under which this natural condition for the moduli space holds, furthermore proving that the map $\Uh(G)_\alpha^\varepsilon \to \Uf(G)$ is holomorphic.

The first step is to undertake a \emph{gauge fixing} for solutions of the coupled harmonic equations \eqref{eq:charmonicity}, whereby the complex structure \eqref{eq:uJ} and the symmetric tensor $\mathbf{g}^{\mathbb{J}}_{\alpha,\varepsilon}$ descend to the moduli space. Difficulties will arise, due to the fact that $\mathbf{g}^{\mathbb{J}}_{\alpha,\varepsilon}$ is neither a definite pairing nor non-degenerate.

Consider a solution $(J,A,\psi)$ of the coupled harmonic equations~\eqref{eq:charmonicity} with coupling constant $\alpha >0$ and parameter $\varepsilon \in \{-1,1\}$. We start by characterizing the tangent space to $\Uh(G)_\alpha^\varepsilon$ at $[(J,A,\psi)]$. An infinitesimal variation of the triple $(J,A,\psi)$ is given by 
\begin{equation}\label{eq:cS1}
(\dot J,a, \dot \psi) \in \cS^1 := T_J\JJJ \oplus \Omega^1(\Sigma,E_K(\liek)) \oplus \Omega^1(\Sigma,E_K(\liek)).
\end{equation}
As above, we identify $T_J\JJJ$ with the space of endomorphisms $\dot
J \colon T\Sigma \to T\Sigma$ such that $\dot J J = - J \dot J$. The
proof of the following lemma follows from a straightforward
calculation.

\begin{lemma}\label{lem:linear}
The linearization of the coupled harmonic equations \eqref{eq:charmonicity} at $(J,A,\psi)$ is given by
\begin{equation}\label{eq:linear}
\begin{split}
d_Aa - [\dot \psi,\psi] & = 0,\\
d_A \dot \psi + [a,\psi] & = 0, \\ 
d_A (J \dot \psi) + [a, J \psi] - d_A (\psi( \dot J)) & = 0, \\ 
\varepsilon \delta S(\dot J) - \alpha * d(B(\Lambda_\omega d_A a,*\psi)) - \alpha * d(B(\Lambda_\omega F_A,* \dot \psi)) + \alpha * d(B(\Lambda_\omega F_A,\psi(\dot J))) & = 0,
\end{split}
\end{equation}
where $\delta S \colon T_J\JJJ \to C_0^\infty(\Sigma,\RR)$ is the linearization of the scalar curvature.
\end{lemma}

We denote by ${\bf L}_\alpha^\varepsilon (\dot J,a, \dot \psi)$ the differential operator defined by the left-hand side of equations \eqref{eq:linear}.
We turn next to the study of the infinitesimal action, in order to define a complex of differential operators.
From the proof of Proposition \ref{p:mmap}, using the connection $A$ we can identify elements $\zeta\in \operatorname{Lie} \widetilde{\KKK}$ with pairs 
$$
\zeta \cong (f,u) \in \mathcal{S}^0 := C_0^\infty(\Sigma,\RR) \oplus \Omega^0(\Sigma,E_K(\liek)),
$$ 
and the infinitesimal action $\mathbf{P}(f,u):= (f,u)\cdot (J,A,\psi)$ at $(J,A,\psi)$ is
\begin{align}\label{eq:infinitesimalactionmod}
\mathbf{P}(f,u) = -(L_{\eta_f}J,d_A u + i_{\eta_f}F_A, d_A(i_{\eta_f}\psi) + [\psi,u]).
\end{align}
Define the vector space
$$
\mathcal{S}^2 := C_0^\infty(\Sigma,\RR) \oplus (\Omega^2(\Sigma,E_K(\liek)))^{\oplus 3}, 
$$
so that ${\bf L}_\alpha^\varepsilon (\dot J,a, \dot \psi)\in \mathcal{S}^2 $ (where we read equations \eqref{eq:linear} from bottom to top), and consider the complex of linear differential operators
\begin{equation}\label{eq:TWMcomplexring}
\begin{tikzcd}
(\mathcal{S}^*)\colon & 0 \ar[r] & \mathcal{S}^0 \ar[r,"{\mathbf{P}}"] & \mathcal{S}^1 \ar[r,"{\bf L}_\alpha^\varepsilon"] & \mathcal{S}^2 \ar[r] & 0.
\end{tikzcd}
\end{equation}
The cohomology $H^1(\mathcal{S}^*) := \frac{\ker{\bf L}_\alpha^\varepsilon}{\operatorname{Im}\, \mathbf{P}}$ can be formally identified with the tangent space $T_{[(J,A,\psi)]} \Uh(G)_\alpha^\varepsilon$. Our next result shows that the moduli space $\Uh(G)_\alpha^\varepsilon$ is finite dimensional. Notice that both $\mathbf{P}$ and ${\bf L}_\alpha^\varepsilon$ are multi-degree differential operators, and hence we shall use the generalized notion of ellipticity provided by Douglis and Nirenberg \cite{DN}. For the general theory of linear multi-degree elliptic differential operators we refer to \cite{LM2}.

\begin{lemma}\label{lem:sesTWM}
The sequence \eqref{eq:TWMcomplexring} is an elliptic complex of multi-degree linear differential operators. Consequently, the cohomology groups $H^j(\cS^*)$, with $j = 0,1,2$, are finite-dimensional.
\end{lemma}

\begin{proof}
We consider the decompositions 
$$
\mathbf{P} = 
\left(
\begin{array}{cccc}
\mathbf{P}_{11} & \mathbf{P}_{12} \\
\mathbf{P}_{21} & \mathbf{P}_{22} \\
\mathbf{P}_{31} & \mathbf{P}_{32} \\
  \end{array}
\right)
\qquad 
{\bf L}_\alpha^\varepsilon = 
\left(
\begin{array}{cccc}
\mathbf{L}_{11} & \mathbf{L}_{12} & \mathbf{L}_{13} \\
\mathbf{L}_{21} & \mathbf{L}_{22} & \mathbf{L}_{23} \\
\mathbf{L}_{31} & \mathbf{L}_{32} & \mathbf{L}_{33} \\
\mathbf{L}_{41} & \mathbf{L}_{42} & \mathbf{L}_{43}
  \end{array}
\right),
$$
where to clarify the notation, we notice e.g. that
\begin{align*}
\mathbf{P}_{11}(f) &= - L_{\eta_f}J, &  \mathbf{P}_{12} &= 0, & \mathbf{P}_{21}(f) & = - i_{\eta_f}F_A,\\
\mathbf{P}_{22}(u) & = - d_A u, & \mathbf{P}_{31}(f) &= - d_A(i_{\eta_f}\psi), & \mathbf{P}_{32}(u) &= - [\psi,u].
\end{align*}
Then, the tuples $\mathbf{t} = (2,2)$ and $\mathbf{s} = (0,1,0)$ form a system of orders for $\mathbf{P}$, since
\begin{align*}
o(\mathbf{P}_{11}) &= 2 \leq 2-0, &  o(\mathbf{P}_{12}) &= 0 \leq 2-0, & o(\mathbf{P}_{21}) & = 1 \leq 2-1,\\
o(\mathbf{P}_{22}) & = 1 \leq 2-1, & o(\mathbf{P}_{31}) & = 2 \leq 2-0, & o(\mathbf{P}_{32}) & = 0 \leq 2-0.
\end{align*}
and the associated leading symbol is
\begin{equation*}
\sigma_{\mathbf{P}}(v)(f,u)= (\sigma_{\mathbf{P}_1}(v),\sigma_{\mathbf{P}_2}(v),\sigma_{\mathbf{P}_3}(v)),
\end{equation*}
with
\begin{align*}
      \sigma_{\mathbf{P}_1}(v) & = - f Jv \otimes \omega^{-1}v,\\ 
      \sigma_{\mathbf{P}_2}(v) & = -v \otimes u - f i_{\omega^{-1}v}F_A, \\
      \sigma_{\mathbf{P}_3}(v) & =   f v \otimes \psi(Jg^{-1}v).
\end{align*}
Similarly, $\mathbf{t} = (2,2,2)$ and $\mathbf{s} = (0,1,1,1)$ form a system of orders for ${\bf L}_\alpha^\varepsilon$, since
\begin{align*}
o(\mathbf{L}_{11}) &= 2 \leq 2-0, &  o(\mathbf{L}_{12}) &= 2 \leq 2-0, & o(\mathbf{L}_{13}) & = 1 \leq 2-0,\\
o(\mathbf{L}_{21}) & = 1 \leq 2-1, & o(\mathbf{L}_{22}) & = 0 \leq 2-1, & o(\mathbf{L}_{23}) & = 1 \leq 2 - 1,\\
o(\mathbf{L}_{31}) & = 0 \leq 2-1, & o(\mathbf{L}_{32}) & = 0 \leq 2-1, & o(\mathbf{L}_{33}) & = 1 \leq 2 - 1,\\
o(\mathbf{L}_{41}) & = 0 \leq 2-1, & o(\mathbf{L}_{42}) & = 1 \leq 2-1, & o(\mathbf{L}_{43}) & = 0 \leq 2 - 1,
\end{align*}
and the associated leading symbol is
\begin{equation*}
\sigma_{{\bf L}_\alpha^\varepsilon}(v)(\dot J ,a,\dot \psi)= (\sigma_{\mathbf{L}_1}(v),\sigma_{\mathbf{L}_2}(v),\sigma_{\mathbf{L}_3}(v),\sigma_{\mathbf{L}_4}(v)),
\end{equation*}
with
\begin{align*}
      \sigma_{\mathbf{L}_1}(v) & =  \varepsilon * v \wedge v(\dot J) - \alpha * v \wedge B(\Lambda_\omega (v \wedge a),*\psi),\\ 
      \sigma_{\mathbf{L}_2}(v) & = -v \wedge \psi(\dot J) + v \wedge J\dot \psi, \\
      \sigma_{\mathbf{L}_3}(v) & =  v \wedge \dot \psi,\\
      \sigma_{\mathbf{L}_4}(v) & =  v \wedge a.
\end{align*}
We prove next that the associated sequence of symbols is exact. Assuming $\sigma_{\mathbf{P}}(v)(f,u)=0$, it follows from $\sigma_{\mathbf{P}_1}(v) = 0$ that $f = 0$. Hence, $\sigma_{\mathbf{P}_2}(v) = 0$ implies $u = 0$, and $\sigma_{\mathbf{P}}(v)$ is injective. 

To prove that $\sigma_{{\bf L}_\alpha^\varepsilon}(v)$ is surjective, we take $w \in \mathcal{S}^2$, which we can assume to be of the form
$$
w = (\nu,u_2 \omega, u_3\omega,u_4\omega).
$$
By dimensional reasons $v \wedge \omega = 0$, and hence
$$
u_4 \omega = v \wedge a, \qquad  \textrm{ for } a \in \Omega^1(\Sigma,E_K(\liek).
$$
We can choose $\dot J$ in the line spanned by $g^{-1}v \otimes Jv$ such that
$$
\varepsilon * v \wedge v(\dot J) = \nu + \alpha * v \wedge B(\Lambda_\omega (v \wedge a),*\psi).
$$
Finally, with these choices, $\sigma_{{\bf L}_\alpha^\varepsilon}(v)(\dot J ,a,\dot \psi) = w$ is equivalent to the equation
$$
v^{0,1} \wedge \dot \psi^{1,0}  = (u_2 + iu_3)\omega + i v \wedge \psi(\dot J).
$$
Again, by dimensional reasons $v^{0,1} \wedge ((u_2 + iu_3)\omega + i v \wedge \psi(\dot J)) = 0$, and therefore $\sigma_{{\bf L}_\alpha^\varepsilon}(v)$ is surjective.

To finish, assume that $\sigma_{{\bf L}_\alpha^\varepsilon}(v)(\dot J ,a,\dot \psi) = 0$. Then, $\sigma_{\mathbf{L}_4}(v) = 0$ implies that there exists $x \in \Omega^0(\Sigma,E_K(\liek)$ such that
$$
a = v \otimes x. 
$$
From this, we obtain 
$$
B(\Lambda_\omega (v \wedge a),*\psi) = 0 %B(x,(*\psi))* v \wedge v = 0,
$$
and therefore $\sigma_{\mathbf{L}_1}(v) = 0$ implies $v \wedge v(\dot J)  = 0$. By \cite{Fj}, there exists a smooth function $f \in C^\infty_0(\Sigma,\RR)$ such that
$$
\dot J =   -f Jv \otimes \omega^{-1}v.
$$
Combined with $\sigma_{\mathbf{L}_2}(v) = 0$, we obtain
$$
v \wedge J\dot \psi = v \wedge \psi(\dot J) = - f \psi(\omega^{-1}v) v \wedge Jv = f \psi(Jg^{-1}v) v \wedge Jv.
$$
Using now that $\dot \psi \wedge v = 0$, it follows that
$$
\dot \psi = f v \otimes \psi(Jg^{-1}v).
$$
Finally, writing $\omega =  \lambda v \wedge Jv$ and
$$
F_A = t v \wedge Jv \textrm{ for } t \in \Omega^0(\Sigma,E_K(\liek),
$$
it follows that
$$
a + f i_{\omega^{-1}v}F_A = v \otimes (x + \lambda^{-1}f t),
$$
and therefore we conclude 
\[
(\dot J,a,\dot \psi) = \sigma_{{\bf P}}(v)(f,- x - \lambda^{-1}f t).
\qedhere
\]
\end{proof}

Our strategy to build a complex structure induced by \eqref{eq:uJ} on the moduli space is to work orthogonally to the image of the infinitesimal action operator $\mathbf{P}$ in \eqref{eq:TWMcomplexring} with respect to the indefinite pairing $\mathbf{g}^{\mathbb{J}}_{\alpha,\varepsilon}$ in \eqref{eq:ug}. The existence of this complex structure will automatically yield a symmetric tensor of type $(1,1)$, since the 2-form $\boldsymbol{\omega}^{\mathbb{J}}_{\alpha,\varepsilon}$ in \eqref{eq:uomega} is well defined on the cohomology $H^1(\mathcal{S}^*)$ by Proposition \ref{p:mmap}. Consider the $L^2$-pairing on $\mathcal{S}^0 := C_0^\infty(\Sigma,\RR) \oplus \Omega^0(\Sigma,E_K(\liek)$, the domain of the operator $\mathbf{P}$, induced by $\omega$ and $B$:
\begin{equation}\label{eq:L2ell}
	\begin{split}
\langle(f,u),(f,u)\rangle  & {} = \int_\Sigma f^2 \omega + \int_\Sigma B(u,u) \omega.
	\end{split}
\end{equation}
Notice that, regarded as a pairing on $\Lie \widetilde{\KKK}$ (see Section \ref{coupling}), this is $A$-dependent, and hence it can be regarded as a family of pairings varying over the configuration space $\XXX$ in Proposition \ref{p:mmap}. Consider the map $\mu \colon \XXX \to \mathcal{S}^0$ defined by (cf. \eqref{eq:mmap})
$$
\mu(J,A,\psi) = (- \varepsilon S_J + \alpha * d(B(\Lambda_\omega F_A,*\psi)),\alpha d_A(J\psi)).
$$
Consider the operator $\widetilde{\mathbf{P}} \colon \mathcal{S}^0 \to \mathcal{S}^1$, defined by
\begin{align}\label{eq:modinfinitesimalaction}
\widetilde{\mathbf{P}}(f,u) = -(L_{\eta_f}J,d_A (u + 2\psi(J\eta_f)) + i_{\eta_f}F_A, d_A(i_{\eta_f}\psi) + [\psi,(u + 2\psi(J\eta_f)]),
\end{align}
which we regard as a \emph{modified infinitesimal action} (see \eqref{eq:modinfidea}). Observe that, by definition, $\operatorname{Im}\widetilde{\mathbf{P}} = \operatorname{Im}\mathbf{P}$.

\begin{lemma}\label{lem:P*}
The following operator provides a formal adjoint of $\widetilde{\mathbf{P}}$ for the pairings \eqref{eq:L2ell} and \eqref{eq:ug}
$$
\widetilde{\mathbf{P}}^* = \delta \mu \circ \mathbb{J} \colon \mathcal{S}^1 \lra \mathcal{S}^0.
$$
More explicitly, setting $\widetilde{\mathbf{P}}^* = \widetilde{\mathbf{P}}^*_1 \oplus \widetilde{\mathbf{P}}^*_2$, we have:
\begin{align*}
\widetilde{\mathbf{P}}^*_1(\dot J,a,\dot \psi) & =  - \varepsilon \delta S(J \dot J) + \alpha * d \left(B(\Lambda_\omega d_A \dot \psi,*\psi) - B(\Lambda_\omega F_A,* a) + B(\Lambda_\omega F_A,\psi(J \dot J))\right),\\
\widetilde{\mathbf{P}}^*_2(\dot J,a,\dot \psi) & = \alpha * \left(d_A (J a) - [\dot \psi, J \psi] - d_A (\psi(J\dot J))\right).
\end{align*}

\end{lemma}
 	
\begin{proof}
The proof follows from a straightforward calculation using the formal properties of the moment map in Proposition \ref{p:mmap}. Let $(f,u) \in \mathcal{S}^0$ and define
$$
\widetilde \zeta = u + 2\psi(J\eta_f) + A^\perp(\eta_f).
$$
Notice that
\begin{equation}\label{eq:modinfidea}
\widetilde \zeta \cdot (J,A,\psi) = \widetilde{\mathbf{P}}(f,u).
\end{equation}
Then, setting $v = (\dot J,a,\dot \psi)$, we have
\begin{align*}
\langle \widetilde{\mathbf{P}}^*v,(f,u)\rangle &= \langle d\mu_{\widetilde{\KKK}}^{\mathbb{J}}(\mathbb{J}v), \zeta \rangle = \boldsymbol{\omega}^{\mathbb{J}}_{\alpha,\varepsilon}(\zeta \cdot (J,A,\psi),\mathbb{J}v) 
\\&
= \mathbf{g}^{\mathbb{J}}_{\alpha,\varepsilon}(\widetilde{\mathbf{P}}(f,u),v) = \mathbf{g}^{\mathbb{J}}_{\alpha,\varepsilon}(v,\widetilde{\mathbf{P}}(f,u)).
\qedhere
\end{align*}
\end{proof}

Consider now the differential operator
\begin{equation}\label{eq:Loperator}
\begin{array}{cccl}
\mathcal{L}\colon & \mathcal{S}^0  & \lra & \mathcal{S}^0 \\
& (f,u) & \longmapsto & \widetilde{\mathbf{P}}^*\circ \widetilde{\mathbf{P}}(f,u).
\end{array}
\end{equation}
The key condition on the solution $(J,A,\psi)$ of \eqref{eq:charmonicity} which we need to assume in order to construct the complex structure on the moduli space is the vanishing of the kernel of $\mathcal{L}$. Notice that, unlike in the standard cases in gauge theory in which the parameter space metric is positive definite, $\ker \mathcal{L}$ does not relate in general to automorphisms of the triple $(J,\psi,A)$, but rather to null vectors with respect to $\mathbf{g}^{\mathbb{J}}_{\alpha,\varepsilon}$. We build on the following technical result, which shows in particular that the subspace of null vectors in $\operatorname{Im} \mathbf{P}$ is finite-dimensional. Using the $L^2$ norm \eqref{eq:L2ell}, we extend the domain of $\mathcal{L}$ to an appropriate Sobolev completion.

\begin{proposition}\label{prop:zeroindex}
The operator $\mathcal{L}$ is Fredholm with zero index. Furthermore, elements $(f,u) \in \ker \mathcal{L}$ are smooth.
\end{proposition}

\begin{proof}
We observe that $\widetilde{\mathbf{P}} = \mathbf{P} \circ \mathbf{T}$, where $\mathbf{T} \colon \mathcal{S}^0 \to \mathcal{S}^0$ is the invertible operator
\begin{align}\label{eq:modinfinitesimalaction.1}
\mathbf{T}(f,u) = (f,u + 2\psi(J\eta_f)).
\end{align}
Consider the operator $\widetilde{\mathcal{L}} := \widetilde{\mathbf{P}}^* \circ \mathbf{P} \colon \cS^0 \to \cS^0$, explicitly given by
\begin{align*}
\widetilde{\mathcal{L}}_1(f,u) & =  \varepsilon \delta S(J L_{\eta_f}J) - \alpha * d \left(B\left(\Lambda_\omega([F_A,i_{\eta_f}\psi] + d_A[\psi,u]),*\psi\right)\right)\\
& + \alpha * d\left(B\left(\Lambda_\omega F_A,* (d_A u + i_{\eta_f}F_A)\right) - B\left(\Lambda_\omega F_A,\psi(J L_{\eta_f}J)\right)\right),\\
\widetilde{\mathcal{L}}_2(f,u) & = -\alpha * \left(d_A (J(d_A u + i_{\eta_f}F_A)) - [( d_A(i_{\eta_f}\psi) + [\psi,u]), J \psi] - d_A (\psi(JL_{\eta_f}J))\right).
\end{align*}
Then, the tuples $\mathbf{t} = (4,4)$ and $\mathbf{s} = (0,2)$ form a system of orders for $\widetilde{\mathcal{L}}$, since
\begin{alignat*}{2}
o(\widetilde{\mathcal{L}}_{11}) &= 4 \leq 4-0, \quad&\quad o(\widetilde{\mathcal{L}}_{12}) &= 2 \leq 4-0, \\
o(\widetilde{\mathcal{L}}_{21}) & = 2 \leq 4-2, & o(\widetilde{\mathcal{L}}_{22}) & = 2 \leq 4-2,
\end{alignat*}
and the associated leading symbol is
\begin{align*}
\sigma_{\widetilde{\mathcal{L}}_1}(v)(f,u) & = - \varepsilon * v \wedge v(J \omega^{-1}v)  Jv f & \\
& = - \varepsilon |v|^4 f,\\
\sigma_{\widetilde{\mathcal{L}}_2}(v)(f,u) & = - \alpha * \left(v \wedge J(v \otimes u + f i_{\omega^{-1}}F_A) - f v \wedge i_{\omega^{-1}} \psi + f v \wedge \psi(Jv \otimes \omega^{-1}v) \right)\\
& = - \alpha |v|^2 u - \alpha * \left(v \wedge J i_{\omega^{-1}}F_A + v \wedge i_{\omega^{-1}} \psi - v \wedge \psi(Jv \otimes \omega^{-1}v)\right) f,
\end{align*}
which is clearly invertible. Consequently, by the general theory of linear multi-degree elliptic differential operators (see \cite{LM2}) it follows that $\widetilde{\mathcal{L}}$ is Fredholm and elements in $\ker \widetilde{\mathcal{L}}$ are smooth. The result follows now from $\ker \mathcal{L} = \mathbf{T}^{-1} \ker \widetilde{\mathcal{L}}$ and the fact that $\mathcal{L}$ is self-adjoint.
\end{proof}

Assuming that $\ker \mathcal{L}$ is trivial, in the next result we obtain a natural gauge fixing via a $\mathbf{g}^{\mathbb{J}}_{\alpha,\varepsilon}$-orthogonal decomposition
\begin{equation}\label{eq:perp}
\mathcal{S}^1 = \operatorname{Im} \; \widetilde{\mathbf{P}} \oplus (\operatorname{Im} \; \widetilde{\mathbf{P}})^{\perp_{\mathbf{g}^{\mathbb{J}}_{\alpha,\varepsilon}}}.
\end{equation}

\begin{lemma}\label{lem:gaugefixing}
Assume that $\ker \mathcal{L} = \{0\}$. Then, there exists an orthogonal decomposition \eqref{eq:perp} for the pairing $\mathbf{g}^{\mathbb{J}}_{\alpha,\varepsilon}$. Consequently, for any element $v \in \mathcal{S}^1$ there exists a unique $\Pi v \in \operatorname{Im} \; \widetilde{\mathbf{P}}$ such that $(\dot J,a, \dot \psi) = v - \Pi v$ solves the linear equations
\begin{equation*}\label{eq:gaugefixing}
\begin{split}
\alpha d_A (J a) - \alpha [\dot \psi, J \psi] - \alpha d_A (\psi(J\dot J)) & = 0,\\
\varepsilon \delta S(J \dot J) + \alpha * d \left(B(\Lambda_\omega d_A \dot \psi,*\psi) - B(\Lambda_\omega F_A,* a) + B(\Lambda_\omega F_A,\psi(J \dot J))\right) & = 0.
\end{split}
\end{equation*}
\end{lemma}

\begin{proof}
Notice first that from the non-degeneracy of $B$, the pairing given in \eqref{eq:L2ell} is non-degenerate. Thus
$$
\ker \, {\bf{\widetilde P}}^*= (\operatorname{Im} \,\widetilde{\mathbf{P}})^{\perp_{\mathbf{g}^{\mathbb{J}}_{\alpha,\varepsilon}}}.
$$
If $v\in\operatorname{Im}  \; \widetilde{\mathbf{P}} \cap (\operatorname{Im}  \; \widetilde{\mathbf{P}})^{\perp_{\mathbf{g}^{\mathbb{J}}_{\alpha,\varepsilon}}}$, then $v = \widetilde{\mathbf{P}}(y)$ for $y \in \cS^0$. But then $\widetilde{\mathbf{P}}^*\circ \widetilde{\mathbf{P}}(y)=0 $ and, by $\ker \mathcal{L} = \{0\}$, $v=0$. Thus
\begin{equation}\label{eq:glImpnondeg}
\operatorname{Im}  \; \widetilde{\mathbf{P}} \cap (\operatorname{Im}  \; \widetilde{\mathbf{P}})^{\perp_{\mathbf{g}^{\mathbb{J}}_{\alpha,\varepsilon}}}=\lbrace 0 \rbrace.
\end{equation}
Let $v\in \cS^1$. The condition
$$
v-{\bf \widetilde P}(y)\in (\operatorname{Im}  \;\widetilde{\mathbf{P}})^{\perp_{\mathbf{g}^{\mathbb{J}}_{\alpha,\varepsilon}}}
$$
for some $y\in \cS^0$
is equivalent to 
\begin{equation}
 \label{eq:solvingdirectsum}
{\bf \widetilde P}^*(v)={\bf{\widetilde P}}^*\circ\widetilde{\mathbf{P}} (y).
\end{equation}
But by Proposition \ref{prop:zeroindex} and condition $\ker \mathcal{L} = \{0\}$, $\widetilde{ \mathbf{P}}^*\circ {\bf \widetilde P}$ is surjective. Then, by elliptic regularity, one can solve \eqref{eq:solvingdirectsum} for $y \in \cS^0$. The orthogonal decomposition follows. The last statement of the lemma comes from the expression of $\bf{\widetilde P}^*$ in Lemma \ref{lem:P*}.
\end{proof}

The above lemma suggests to define the space of harmonic representatives of the complex \eqref{eq:TWMcomplexring}, as follows:
$$
\mathcal{H}^1(\cS^*)=\ker \bf{L}^\varepsilon_\alpha \cap \ker {\bf\widetilde P}^*.
$$
Our next result provides our gauge fixing mechanism for the linearization of coupled harmonic equations \eqref{eq:charmonicity}. 

\begin{proposition}\label{prop:lineargaugefixed}
Assume $\ker \mathcal{L} = \{0\}$ and $\alpha \neq 0$. Then, the inclusion $\mathcal{H}^1(\cS^*)\subset \ker \bf{L}^\varepsilon_\alpha$ induces an isomorphism
$$
\mathcal{H}^1(\cS^*)\cong H^1(\cS^*).
$$
More precisely, any class in the cohomology $H^1(\cS^*)$ of the complex \eqref{eq:TWMcomplexring} admits a unique representative $(\dot J,a,\dot \psi)$ solving the linear equations
\begin{equation}\label{eq:lineargaugefixed}
\begin{split}
d_Aa - [\dot \psi,\psi] & = 0,\\
d_A \dot \psi + [a,\psi] & = 0, \\ 
d_A (J \dot \psi) + [a, J \psi] - d_A (\psi( \dot J)) & = 0, \\ 
d_A (J a) - [\dot \psi, J \psi] - d_A (\psi(J\dot J)) & = 0 \\
\varepsilon \delta S(\dot J) - \alpha * d\left(B(\Lambda_\omega d_A a,*\psi) + B(\Lambda_\omega F_A,* \dot \psi) - B(\Lambda_\omega F_A,\psi(\dot J))\right) & = 0,\\
\varepsilon \delta S(J \dot J) - \alpha * d \left(- B(\Lambda_\omega d_A \dot \psi,*\psi) + B(\Lambda_\omega F_A,* a) - B(\Lambda_\omega F_A,\psi(J \dot J))\right) & = 0.
\end{split}
\end{equation}
\end{proposition}

\begin{proof}
The correspondence between $H^1(\cS^*)$ and the space of solutions of \eqref{eq:lineargaugefixed} follows from Lemma \ref{lem:linear} and Lemma \ref{lem:gaugefixing}.
\end{proof}

We are ready to prove our main result, which shows that the gauge fixing in Proposition \ref{prop:lineargaugefixed} enables us to descend the complex structure $\mathbb{J}$ in $\XXX$ and the symmetric tensor $\mathbf{g}^{\mathbb{J}}_{\alpha,\varepsilon}$, to an open subset of the moduli space $\Uh(G)_\alpha^\varepsilon$, via the symplectic reduction in Proposition \ref{p:mmap}. Define
$$
\mathcal{U}^*_{\alpha,\varepsilon} = \{[(J,A,\psi)] \; | \; \ker \mathcal{L}= \{0\} \} \subset \Uh(G)_\alpha^\varepsilon.
$$

\begin{theorem}\label{thm:metric}
The set $\; \mathcal{U}^*_{\alpha,\varepsilon}$ is open in $\; \Uh(G)_\alpha^\varepsilon$, and for any smooth point $[(J,A,\psi)] \in \mathcal{U}^*_{\alpha,\varepsilon}$ the tangent space to $\Uh(G)_\alpha^\varepsilon$ at $[(J,A,\psi)]$, identified with the space of solutions of the gauge fixed linear equations \eqref{eq:lineargaugefixed}, inherits a complex structure $\mathbb{J}$ and a symmetric tensor $\mathbf{g}^{\mathbb{J}}_{\alpha,\varepsilon}$ such that $\boldsymbol{\omega}^{\mathbb{J}}_{\alpha,\varepsilon} = \mathbf{g}^{\mathbb{J}}_{\alpha,\varepsilon}(\mathbb{J},)$, given respectively by \eqref{eq:uJ} and \eqref{eq:ug}, and where $\boldsymbol{\omega}^{\mathbb{J}}_{\alpha,\varepsilon}$ stands for the restriction of \eqref{eq:uomega}. Furthermore,

\begin{enumerate}

\item if $\varepsilon = 1$ the tensor $\mathbf{g}^{\mathbb{J}}_{\alpha,\varepsilon}$ is possibly degenerate,

\item if $\varepsilon = -1$ the tensor $\mathbf{g}^{\mathbb{J}}_{\alpha,\varepsilon}$ is non-degenerate, and defines a pseudo-K\"ahler structure on the moduli space.

\end{enumerate}
In either case, $\boldsymbol{\omega}^{\mathbb{J}}_{\alpha,\varepsilon}$ admits a global K\"ahler potential, that is, $\boldsymbol{\omega}^{\mathbb{J}}_{\alpha,\varepsilon} = dd^c_\mathbb{J} \Phi$, where
$$
\Phi = \varepsilon\nu_\JJJ + \frac{\alpha}{2} \|\psi\|^2_{L^2}, % \int_\Sigma B(\psi \wedge J \psi),
$$
and $\nu_\JJJ$ is induced by the global K\"ahler potential in the space of complex structures compatible with the orientation $\JJJ$ (see \cite[Section 4]{Fj}).

\end{theorem}

\begin{proof}
The fact that $\mathcal{U}^*_{\alpha,\varepsilon}$ is open follows from upper semicontinuity of the dimension of the kernel for elliptic operators. Given now a smooth point $[(J,A,\psi)] \in \mathcal{U}^*_{\alpha,\varepsilon}$, the tangent space is identified with $H^1(\cS^*)$ and inherits a complex structure by Proposition \ref{prop:lineargaugefixed}, using that $\mathbb{J}$ in \eqref{eq:uJ} preserves \eqref{eq:lineargaugefixed}. The existence of the symmetric tensor $\mathbf{g}^{\mathbb{J}}_{\alpha,\varepsilon}$ is a direct consequence of Lemma \ref{lem:sigmaex} and Proposition \ref{p:mmap}, while the listed signature properties follow from Lemma \ref{lemma:indef}. The formula for the K\"ahler potential is a direct consequence of \eqref{eq:sigmaabs}, while an explicit formula for $\boldsymbol{\omega}^{\mathbb{J}}_{\alpha,\varepsilon}$ follows from Lemma \ref{lem:sigmaex}.
\end{proof}

To finish this section, we provide an explicit formula for the (pre)symplectic structure $\boldsymbol{\omega}^{\mathbb{J}}_{\alpha,\varepsilon}$. The proof is straightforward from the previous discussion and Lemma \ref{lem:sigmaex}. 

\begin{corollary}\label{cor:omegaexp}
Let $[(J,A,\psi)] \in \mathcal{U}^*_{\alpha,\varepsilon}$ be a smooth point and take $v_1,v_2$ tangent vectors of $\Uh(G)_\alpha^\varepsilon$ at $[(J,A,\psi)]$, identified with solutions $(\dot J_j,a_j,\dot\psi_j)$ of the gauge fixed linear equations \eqref{eq:lineargaugefixed}. Then, one has
\begin{equation}\label{eq:omegagmoduliex}
\begin{split}
\boldsymbol{\omega}^{\mathbb{J}}_{\alpha,\varepsilon}(v_1,v_2) & = \frac{\varepsilon}{2}\int_{\Sigma}\tr(J\dot J_1\dot J_2) \omega \\
& + \alpha  \int_\Sigma B((a_1 - \psi(\dot J_1)) \wedge J(\dot \psi_2 - (J\psi)(\dot J_2))) \\
& - \alpha \int_\Sigma B((\dot \psi_1 - (J\psi)(\dot J_1)) \wedge J(a_2 - \psi(\dot J_2)))\\
& - \alpha \int_\Sigma B(\psi(\dot J_1)\wedge \psi(\dot J_2)),\\
\mathbf{g}^{\mathbb{J}}_{\alpha,\varepsilon}(v_1,v_2) & = \frac{\varepsilon}{2}\int_{\Sigma}\tr(\dot J_1\dot J_2) \omega \\
& + \alpha  \int_\Sigma B((a_1 - \psi(\dot J_1)) \wedge J(a_2 - \psi(\dot J_2))) \\
& + \alpha \int_\Sigma B((\dot \psi_1 - (J\psi)(\dot J_1)) \wedge J(\dot \psi_2 - (J\psi)(\dot J_2)))\\
& - \alpha \int_\Sigma B(\psi(\dot J_1)\wedge J(\psi(\dot J_2))).
\end{split}
\end{equation}
\end{corollary}

%%%%%%%%%%%%%%%%%%%%%%%%%%%%%%%%%%%%%%%%%%%%%%%%%%%%%%%%%%%%%%%%%%%%
\subsection{Comparison with $\boldsymbol{\Uf(G)}$ and existence}
\label{section-existence}
%%%%%%%%%%%%%%%%%%%%%%%%%%%%%%%%%%%%%%%%%%%%%%%%%%%%%%%%%%%%%%%%%%%%

The aim of this section is twofold. Firstly, we establish a comparison between the moduli spaces $\Uh(G)_\alpha^\varepsilon$ and the universal moduli space of flat $G$-connections $\Uf(G)$ constructed in Section \ref{ssec:UFlat}. As we will see, for any $\alpha > 0$ and $\varepsilon \neq 0$, Theorem \ref{thm:metric} induces a natural holomorphic map 
$$
\Uh(G)_\alpha^\varepsilon \supset \mathcal{U}^*_{\alpha,\varepsilon} \lra \Uf(G),
$$
and hence a holomorphic map into Teichm\"uller space $\cT$. 
%TO PROVE THE FOLLOWING I NEED THAT THE METRIC IS NON-DEGENERATE ALONG THE FIBRES, BUT THIS IS DIFFICULT: In particular, by general principles, Corollary \ref{cor:omegaexp} implies the existence of a K\"ahler Ehresmann connection on
%$$
%\Uh(G)_\alpha^\varepsilon \supset \mathcal{U}^*_{\alpha,\varepsilon} \to \cT
%$$
%(at least in the case $\varepsilon = -1$). 
Secondly, we will prove that for genus of the surface $\Sigma$ bigger than zero, the open $\mathcal{U}^*_{\alpha,\varepsilon} \subset\Uh(G)_\alpha^\varepsilon$ is non-empty for sufficiently small values of $\alpha$.

Consider a point $[(J,D)] \in \Uf(G)$, regarded as the $\widetilde{\GGG}$-orbit of $(J,A,\psi) \in \XXX^*$ solving the equations
\begin{equation}\label{eq:GFlatAPsi}
\begin{split}
F_A -\frac{1}{2}[\psi,\psi] & = 0,\\
d_A\psi & = 0.
\end{split}
\end{equation}
The tangent space $T_{\left[(J,A,\psi)\right]} \Uf(G)$ can be formally identified with the cohomology of the complex of linear differential operators
\begin{equation}\label{eq:TJDU}
\begin{tikzcd}
(\mathcal{C}^*)\colon & 0 \ar[r] & \mathcal{C}^0 \ar[r,"{\mathbf{P}^c}"] & \mathcal{C}^1 \ar[r,"{\bf L^c}"] & \mathcal{C}^2 \ar[r] & 0,
\end{tikzcd}
\end{equation}
where we have
$$
\mathcal{C}^0 = \Lie \widetilde{\GGG} \cong \Omega^0(T\Sigma) \oplus \Omega^0(\Sigma,E_G(\lieg)), \qquad \mathcal{C}^1 = \mathcal{S}^1, \qquad \mathcal{C}^2 = \Omega^2(\Sigma,E_G(\lieg)),
$$
and
\begin{align*}
\mathbf{P}^c(y,u_0+ i u_1) & = - \left(L_yJ,d_Au_0 + i_{y}F_A - [\psi,u_1], d_A u_1 + d_A(\psi(y)) + [\psi,u_0]\right),\\
\mathbf{L}^c(\dot J,a,\dot \psi) & = d_Aa - [\dot \psi,\psi] + i
\left(d_A \dot \psi + [a,\psi]\right).
\end{align*}

\begin{lemma}\label{lem:sesTWM.1}
The sequence \eqref{eq:TJDU} is an elliptic complex of degree-one linear differential operators. Consequently, the cohomology groups $H^j(\cC^*)$, with $j = 0,1,2$, are finite-dimensional.
\end{lemma}

\begin{proof}
The leading symbol of $\sigma_{\mathbf{P}^c}$ is
$$
\sigma_{\mathbf{P}^c}(v)(y,u)= (- Jv \otimes y, -v \otimes u_0, - v \otimes u_1 -v \otimes \psi(y)).
$$
Similarly, the associated leading symbol of ${\bf L}^c$ is 
$$
\sigma_{{\bf L}^c}(v)(\dot J ,a,\dot \psi) = v \wedge (a + i \dot \psi).
$$
The symbol $\sigma_{\mathbf{P}^c}(v)$ is obviously injective, while $\sigma_{{\bf L}^c}(v)$ is surjective for dimensional reasons. Assume now that $\sigma_{{\bf L}^c}(v)(\dot J ,a,\dot \psi) = 0$. Then, $a +i\dot \psi = v \otimes (u_0 + i u_1')$. Again, by dimensional reasons $\dot J = - Jv \otimes y$ for some vector $y$, and hence
\[
(\dot J ,a,\dot \psi) = \sigma_{\mathbf{P}^c}(v)(y,u_0 + i(u_1' - \psi(y))).
\qedhere
\]
\end{proof}

Applying Theorem \ref{corlette}, a solution $(J,A,\psi)$ of the coupled harmonic equations \eqref{eq:charmonicity} induces a reductive flat $G$-connection $D = A + i \psi \in \DDD^*$. This fact, jointly with the natural inclusion $\widetilde{\KKK} \subset \widetilde{\GGG}$, leads to a continuous map
\begin{equation}\label{eq:mapmoduli}
\begin{array}{rcl}
\Uh(G)_\alpha^\varepsilon & \lra & \Uf(G) 
\\
\left[(J,A,\psi)\right] & \longmapsto & \left[(J,A + i \psi)\right].
\end{array}
\end{equation}
Building on Theorem \ref{thm:metric}, our next goal is to prove that this induces a holomorphic map $\Uh(G)_\alpha^\varepsilon \supset \mathcal{U}^*_{\alpha,\varepsilon} \to \Uf(G)$.

\begin{lemma}\label{lem:tangentmap}
Let $[(J,A,\psi)] \in \mathcal{U}^*_{\alpha,\varepsilon}$. Then, \eqref{eq:mapmoduli} induces a complex linear map
\begin{equation}\label{eq:TUseq}
\begin{tikzcd}
H^1(\mathcal{S}^*) \ar[r] & H^1(\mathcal{C}^*),
\end{tikzcd}
\end{equation}
where the complex structure on $H^1(\mathcal{S}^*)$ is the one induced by Proposition \ref{prop:lineargaugefixed}.
\end{lemma}

\begin{proof}
Using that $[(J,A,\psi)] \in \mathcal{U}^*_{\alpha,\varepsilon}$ we can identify 
$$
H^1(\cS^*) \cong \mathcal{H}^1(\cS^*),
$$
where the right-hand side is given by solutions of the gauge-fixed
system of linear equations \eqref{eq:lineargaugefixed}. Then, the map
\[
\begin{array}{ccl}
\mathcal{H}^1(\cS^*)&\lra&H^1(\mathcal{C}^*)
\\
\left[(\dot J, a,\dot \psi)\right] &\longmapsto& \left[(\dot J, a + i \dot \psi)\right]
\end{array}
\]
is complex $\mathbb{C}$-linear, since both complex structures are induced by $\mathbb{I}$ in \eqref{eq:uJ}. 
%Assume that
%$$
%(\dot J, a + i \dot \psi) = \mathbf{P}^c(y,u_0+ i u_1) = - \left(L_yJ,d_Au_0 + i_{y}F_A - [\psi,u_1], d_A u_1 + d_A(\psi(y)) + [\psi,u_0]\right)
%$$
%for $(y,u_0+iu_1) \in \mathcal{C}^0$. Now, we can write
%$$
%y = \eta_{f_0} + J \eta_{f_1} + \beta
%$$
%where $\eta_{f_j}$ are the $\omega$-Hamiltonian vector fields of $f_j \in C_0^\infty(\Sigma,\RR)$ and $\beta$ is the dual of a $1$-form which is harmonic with respect to the K\"ahler metric $g = \omega(,J)$. From this, we can write
%$$
%(\dot J, a + i \dot \psi) = \mathbf{P}(f_0,u_0) - \left(L_{J \eta_{f_1} + \beta}J, i_{J \eta_{f_1} + \beta}F_A - [\psi,u_1], d_A \left(u_1 + \psi(J \eta_{f_1} + \beta)\right)\right)
%$$
%and hence we obtain a constraint
%\begin{equation}\label{eq:lineargaugefixedagain}
%\begin{split}
%d_A (J \dot \psi) + [a, J \psi] - d_A (\psi( \dot J)) & = 0, \\ 
%\varepsilon \delta S(\dot J) - \alpha * d\left(B(\Lambda_\omega d_A a,*\psi) + B(\Lambda_\omega F_A,* \dot \psi) - B(\Lambda_\omega F_A,\psi(\dot J))\right) & = 0.
%\end{split}
%\end{equation}
%I cannot anything from these equations. 
\end{proof}

To finish this section, we address the question of non-emptyness of the moduli space $\Uh(G)_\alpha^\varepsilon$, in genus $g(\Sigma) \geqslant 2$. We change our perspective on the equations \eqref{eq:charmonicity}: we fix a compact Riemann surface $X = (\Sigma,J)$ and consider a flat $G$-bundle $(E_G,D)$. In this setup, consider the coupled equations 
\begin{equation}\label{eq:charmonicg}
\begin{split}
d_{A_h}^*\psi_h & = 0,\\
\varepsilon S_g - \alpha * d\left(B\left(\Lambda_\omega F_A,*\psi_h\right)\right)& = \varepsilon \frac{2\pi \chi(\Sigma)}{V},
\end{split}
\end{equation}
for pairs $(g,h)$, where $g$ is a K\"ahler metric on $X$ with total volume $V$ and  $h\in \Omega^0(E_G(G/K))$ is a reduction
of structure group of $E_G$ to $K$. Here, we denote
$$
D = A_h + i \psi_h
$$
the natural decomposition of $D$ with respect to $h$. A solution $(g,h)$ of the equations \eqref{eq:charmonicg} determines then a solution of \eqref{eq:charmonicity}, for the same value of parameter $\varepsilon \in \{-1,1\}$ and coupling constant $\alpha$, given by the triple $(J,A_h,\psi_h)$.

\begin{theorem}\label{thm:existence}
Let $(E_G,D)$ be an irreducible flat $G$-bundle over a compact Riemann surface $X=(\Sigma,J)$ with genus $g(\Sigma) \geqslant 2$. Then, for any fixed total volume $V > 0$ and parameter $\varepsilon \in \{-1,1\}$, there exists $\alpha_0 > 0$ such that for any $0 < \alpha < \alpha_0$ there exists a solution $(g_\alpha,h_\alpha)$ of the equations \eqref{eq:charmonicg} with
\[
[(J,A_{h_\alpha},\psi_{h_\alpha})] \in \mathcal{U}^*_{\alpha,\varepsilon} \subset \Uh(G)_\alpha^\varepsilon.
\]
Furthermore, the induced map $\mathcal{U}^*_{\alpha,\varepsilon} \to \Uf(G)$ is holomorphic.
\end{theorem}

\begin{proof}
The first part of the proof follows by a perturbation argument. Let $g$ be the unique constant scalar curvature K\"ahler metric on $X$ with total volume $V$, and $h$ the unique harmonic reduction on $(E_G,D)$ (which exists by Theorem \ref{corlette}). Consider the non-linear operator 
$$
\mathbf{Q}_\alpha \colon C_0^\infty(\Sigma,\RR) \oplus \Omega^0(\Sigma,E_K(\liek)) \lra C_0^\infty(\Sigma,\RR) \oplus \Omega^2(\Sigma,E_K(\liek))
$$
defined by
$$
\mathbf{Q}_\alpha(f,u) = (S_{g_f} - \varepsilon^{-1}\alpha * d\left(B\left(\Lambda_{\omega_f} F_{A_{h_u}},*\psi_{h_u}\right)\right) - 2\pi \chi(\Sigma)V^{-1}, \Ad(e^{-iu})d_{A_{h_u}}J\psi_{h_u}),
$$
where
$$
\omega_f = \omega + 2i\partial \dbar f, \qquad h_u = e^{iu}h.
$$
Notice that $\mathbf{Q}_\alpha$ is a multidegree elliptic differential operator, $\mathbf{Q}_0(0,0) = 0$, and that the linearization of $\mathbf{Q}_0$ at $(0,0)$ is 
$$
\delta_0\mathbf{Q}_0(f,u) = - (\delta S(J L_{\eta_f}J),d_A (J d_{A_h}u) - [[\psi_h,u], J \psi]).
$$
The operator $f \mapsto \delta S(J L_{\eta_f}J)$ is the (real) Lichnerowicz operator acting on functions: this is an elliptic self-adjoint semipositive differential operator of order 4, whose kernel is given by the Hamiltonian functions of Killing Hamiltonian vector fields. In particular, by our assumption $g(\Sigma) \geq 2$, it is invertible. On the other hand, the operator 
$$
u \mapsto *\left(d_A (J d_{A_h}u) - [[\psi_h,u], J \psi]\right)
$$
is self-adjoint and has kernel given by the infinitesimal unitary gauge transformations preserving $D$. Hence, since $D$ is irreducible by hypothesis, this operator is also invertible. By an standard implicit function theorem argument,  taking Sobolev completions in the domain and target of $\mathbf{Q}_\alpha$, it follows that there exists $\alpha_0 > 0$ such that for any $0 < \alpha < \alpha_0$ there exists a smooth solution $(g_\alpha,h_\alpha)$ of the equations \eqref{eq:charmonicg}, that is, with
$$
[(J,A_{h_\alpha},\psi_{h_\alpha})] \in \Uh(G)_\alpha^\varepsilon.
$$
To finish the first part of the proof, we need to show that $[(J,A_{h_\alpha},\psi_{h_\alpha})] \in \mathcal{U}^*_{\alpha,\varepsilon}$ for sufficiently small $\alpha$. Following the notation in Lemma \ref{lem:P*}, for $(g_\alpha,A_{h_\alpha},\psi_{h_\alpha})$ define the one-parameter family of multidegree differential operators
$$
\mathcal{N}^\alpha = (\varepsilon^{-1}\widetilde{\mathbf{P}}_1^*,\alpha^{-1}\widetilde{\mathbf{P}}_2^*)\circ \widetilde{\mathbf{P}} \colon \mathcal{S}^0 \lra \mathcal{S}^0.
$$
Notice that $\mathcal{N}^0$ is well-defined, and furthermore
$$
\ker \mathcal{N}^\alpha = \ker \mathcal{L}_{(g_\alpha,A_{h_\alpha},\psi_{h_\alpha})}
$$
for any $\alpha \neq 0$, where $\mathcal{L}_{(g_\alpha,A_{h_\alpha},\psi_{h_\alpha})} =  \widetilde{\mathbf{P}}^*\circ \widetilde{\mathbf{P}}$ is the operator associated to the solution $(J,A_{h_\alpha},\psi_{h_\alpha})$ of \eqref{eq:charmonicity} with respect to the symplectic structure $\omega_\alpha = g_\alpha(J,)$. Decomposing
$$
\mathcal{N}^0 = \mathcal{N}^0_0 \oplus \mathcal{N}^0_1,
$$
it follows that
$$
\mathcal{N}^0_0(f,u) = \delta S(J L_{\eta_f}J).
$$
Hence, arguing as before, $(f,u) \in \ker \mathcal{N}^0$ implies that $f = 0$, and hence it follows that 
$$
\mathcal{N}^0(f,u) = - (0,*d_{A_h}(J d_{A_h}u) - * [[\psi_h,u],J\psi_h]) = 0.
$$
Again, by irreducibility of the $G$-connection $D$, this implies $u = 0$, and hence we conclude that $\ker \mathcal{N}^0 =\{0\}$. Arguing as in the proof of Proposition \ref{prop:zeroindex} and Theorem \ref{thm:metric}, $\ker \mathcal{N}^\alpha$ is upper semicontinuous, and hence for sufficiently small $\alpha$ we have
$$
\ker \mathcal{N}^\alpha =\{0\}.
$$
The last part of the statement follows from that fact that \eqref{eq:mapmoduli} is induced by a $C^\infty$-Frechet map, combined with Lemma \ref{lem:tangentmap}.
\end{proof}

\begin{remark}
We observe that, even though the proof of Theorem \ref{thm:metric} works in the case $g(\Sigma) = 1$, the hypothesis of existence of an irreducible flat $G$-connection is never satisfied in this case \cite{Franco}. We thank Emilio Franco for this observation.
\end{remark}

%%%%%%%%%%%%%%%%%%%%%%%%%%%%%%%%%%%%%%%%%%%%%%%%%%%%%%%%%%%%%%%%%%%%
\section{Universal $G$-Higgs bundle moduli}
\label{section-ghiggs}
%%%%%%%%%%%%%%%%%%%%%%%%%%%%%%%%%%%%%%%%%%%%%%%%%%%%%%%%%%%%%%%%%%%%

%%%%%%%%%%%%%%%%%%%%%%%%%%%%%%%%%%%%%%%%%%%%%%%%%%%%%%%%%%%
\subsection{The universal Higgs field}
\label{ssec:UHiggs}
%%%%%%%%%%%%%%%%%%%%%%%%%%%%%%%%%%%%%%%%%%%%%%%%%%%%%%%%%%%

We fix a smooth oriented compact surface $\Sigma$. We also consider a connected semisimple complex Lie group $G$, with Lie algebra $\glie$, and
fix an antiholomorphic involution $\tau$ of $G$ defining  a maximal compact subgroup $K:=G^\tau\subset G$, with Lie algebra
$\liek$.

Let $E_G$ be a smooth principal $G$-bundle over $\Sigma$, $h\in\Omega^0(E_G(G/K))$ a reduction of the structure group of $E_G$ to $K$, and $E_K$ the corresponding principal $K$-bundle. Let $\JJJ$ be the space of complex structures on $\Sigma$ compatible with the given orientation. Consider the space
\begin{equation*}\label{configurationagain}
\XXX = \JJJ\times \AAA\times \Omega^1(\Sigma,E_K(\liek)).
\end{equation*}
In this section we are interested in the geometry of the submanifold of $\XXX$ given by \emph{universal Higgs fields},
$$
\XXX^{Higgs} = \{(J,A,\psi) \in \XXX \; | \; \dbar_{J,A}\varphi = 0 \},
$$ 
where the \emph{Higgs field}
$$
\varphi(J,A,\psi):= \psi^{1,0_J}
$$
is regarded as a $\Omega^1(\Sigma,E_G(\lieg))$-valued function on the parameter space $\XXX$. Our first goal is to prove that $\XXX$ admits a complex structure which is compatible with the integrability condition $\dbar_{J,A}\varphi = 0$, making $\XXX^{Higgs} \subset \XXX$ a complex submanifold. For this, consider the complex structure
\begin{equation}\label{eq:uI}
\mathbb{I}(\dot J, a , \dot \psi) = (J\dot J,J a,- J \dot \psi + \psi(\dot J )),
\end{equation}
for $(\dot J, a , \dot \psi) \in T_{(J, A , \psi)}\XXX$. Notice that this complex structure corresponds to the natural fibrewise cotangent complex structure in the holomorphic fibration $\JJJ \times \AAA \to \JJJ$. In particular, the map $\pi_1 \colon \XXX \to \JJJ$ is holomorphic, and the complex structure along the fibres coincides with the $J$-dependent complex structure $\mathbf I$ in \eqref{eq:IJK}. Using the Chern--Singer correspondence, it is not difficult to give holomorphic coordinates on the space $\JJJ \times \AAA \cong \JJJ \times \CCC$, and hence on $(\XXX,\mathbb{I})$. In the next result we give an independent proof of the vanishing of the Nijenhuis tensor $N_{\mathbb{I}}$ of $\mathbb{I}$. 

\begin{lemma}\label{lem:Iint}
The complex structure $\mathbb{I}$ is formally integrable, that is, its Nijenhuis tensor vanishes:
$$
N_{\mathbb{I}} = 0.
$$
\end{lemma}

\begin{proof}
The complex structures on the fibres and base of the fibration $(\XXX,\mathbb{I}) \to (\JJJ,\mathbb{J}_\JJJ)$ are well-known to be integrable. Hence, it suffices to calculate the mixed component of the Nijenhuis tensor
\begin{align*}
N_{\mathbb{I}}(v_1,v_2) & = [\mathbb{I}v_1,\mathbb{I}v_2] - \mathbb{I}[\mathbb{I}v_1,v_2] - \mathbb{I}[v_1,\mathbb{I}v_2] - [v_1,v_2],
\end{align*}
that is, for $v_1 = (\dot J,0,0)$, and $v_2 =(0,a,\dot \psi)$. With this choice, the last term vanishes identically, and we have
\begin{align*}
N_{\mathbb{I}}(v_1,v_2) & = [(J\dot J,0,0),(0,Ja,- J \dot \psi)] - \mathbb{I}[(J\dot J,0,0),(0,a,\dot \psi)]\\
& - \mathbb{I}[(\dot J,0,0),(0,Ja,- J \dot \psi)],\\
& = (0,-a(J\dot J),\dot \psi(J\dot J)) - \mathbb{I}(0,-a(\dot J),\dot \psi (\dot J))\\
& = (0,-a(J\dot J + \dot J J),\dot \psi(J\dot J + \dot J J)) = 0.
\qedhere
\end{align*}
\end{proof}

In our next result we show that $\XXX^{Higgs} \subset (\XXX,\mathbb{I})$ is formally a complex submanifold.

\begin{proposition}\label{p:dbarAphi}
The tangent space of $\XXX^{Higgs}$ at $(J,A,\psi)$, is given by triples $(\dot J,a, \dot \psi)\in T_J\JJJ\oplus\Omega^1(\Sigma,E_K(\liek))\oplus\Omega^1(\Sigma,E_K(\liek))$ satisfying
\begin{equation}\label{eq:TUHiggs}
d_A \dot \psi^{1,0_J} + [a,\varphi] - \frac{i}{2} d_A(\psi (\dot J)) = 0.
\end{equation}
Consequently, $\XXX^{Higgs}$ is formally a complex submanifold of $(\XXX,\mathbb{I})$.
\end{proposition}

\begin{proof}
By dimensional reasons, the integrability condition $\dbar_{J,A} \psi^{1,0_J} = 0$ is equivalent to
$$
\frac{1}{2}d_A \left(\psi(\Id - i J)\right) = 0.
$$
Taking variations on the parameters, we obtain \eqref{eq:TUHiggs}. Now, given $(\dot J,a, \dot \psi)$ solving \eqref{eq:TUHiggs}, we have
\begin{align*}
d_A (-J\dot \psi + \psi(\dot J))^{1,0_J} + [Ja,\varphi] - \frac{i}{2} d_A\left(\psi (J\dot J)\right) & = i d_A \dot \psi^{1,0_J} + i[a,\varphi] + d_A\left(\psi^{0,1_J}(\dot J)\right)\\
& - \frac{1}{2}d_A\left(\psi^{0,1_J}(\dot J)\right) + \frac{1}{2}d_A\left(\psi^{1,0_J}(\dot J)\right)\\
& = i \left(d_A \dot \psi^{1,0_J} + [a,\varphi] - \frac{i}{2} d_A\left(\psi(\dot J)\right)\right) = 0,
\end{align*}
and hence the tangent space to $\XXX^{Higgs}$ is preserved by $\mathbb{I}$.
\end{proof}

Our next goal is to define a universal moduli space of Higgs bundles, with a holomorphic fibration structure over Teichm\"uller space. The first step is to show that the integrability condition for $\varphi$ is compatible with the group $\widetilde{\GGG}$ defined in \eqref{eq:extcxgauge}. We observe first that there is a holomorphic action of $\widetilde{\GGG}$ on $\JJJ \times \AAA$, given by
$$
g(J,A) = (p(g)_*J,A_g),
$$
where $A_g$ is the defined via the Chern--Singer correspondence: 
$$
A_g = A_{h,g_*\dbar_{J,A}}.
$$
This action extends to the fibrewise cotangent bundle, defining a holomorphic $\widetilde{\GGG}$-action on $(\XXX,\mathbb{I})$ and preserving $\XXX^{Higgs}$. Let us denote $\XXX^{ps} \subset \XXX^{Higgs}$ the subset of triples $(J,A,\psi)$ such that $(E,\varphi)$ is a polystable $G$-Higgs bundle over the compact Riemann surface $X = (\Sigma,J)$, where $E$ denotes the holomorphic principal $G$-bundle over $X$ corresponding to the connection $A$. The $\widetilde{\GGG}$-action on $\XXX$ preserves the subsets
$$
\XXX^{ps} \subset \XXX^{Higgs} \subset \XXX,
$$
and we define the \emph{universal moduli space of $G$-Higgs bundles} on $E_G$ as the quotient
$$
\mathcal{U}^{Higgs}(G) := \XXX^{ps}/\widetilde{\GGG}.
$$
By construction, $\mathcal{U}^{Higgs}(G)$ fibres over the Teichm\"uller space $\mathcal{T}$ of complex structures modulo diffeomorphisms isotopic to the identity
\begin{equation}\label{eq:FibTeichmuller}
\mathcal{U}^{Higgs}(G) \lra \mathcal{T} := \JJJ/\operatorname{Diff}_0(\Sigma).
\end{equation}
This fibration is naturally holomorphic with respect to the complex structure $\mathbb{I}$.

\begin{remark}
Despite our efforts, we have not been able to find a complex symplectic interpretation of the universal moduli space of $G$-Higgs bundles. Further insight on this is provided in ongoing work by N. Hitchin \cite{nigel}.
\end{remark}

%%%%%%%%%%%%%%%%%%%%%%%%%%%%%%%%%%%%%%%%%%%%%%%%%%%%%%%%%%%%%%%%%%%%
\subsection{K\"ahler fibration for the universal Higgs field}
\label{couplingHiggs}
%%%%%%%%%%%%%%%%%%%%%%%%%%%%%%%%%%%%%%%%%%%%%%%%%%%%%%%%%%%%%%%%%%%%

Consider the structure of K\"ahler fibration $\widehat{\boldsymbol{\omega}}_{\mathbf{I}}$ on $(\XXX,\mathbb{I}) \to \JJJ$ defined fibrewise by the symplectic structure $\omega_{\mathbf{I}} = \operatorname{Re}\Omega_{\mathbf{J}}$ on $\AAA\times \Omega^1(\Sigma,E_K(\liek))$. More explicitly,
\begin{equation}\label{eq:hatomehaI}
\widehat{\boldsymbol{\omega}}_{\mathbf{I}}((0,a_1,\dot \psi_1),(0,a_2,\dot \psi_2)) =  \int_\Sigma B(a_1 \wedge a_2) - \int_\Sigma B(\dot \psi_1 \wedge \dot \psi_2).
\end{equation}
The aim of this section is to show that such a structure admits a K\"ahler Ehresmann connection compatible with Hitchin's equations \eqref{hitchin}. The situation here is opposite to that in Section \ref{ssec:potential}: the holomorphic fibration defined by $\mathbb{I}$ is non-trivial, whereas $\widehat{\boldsymbol{\omega}}_{\mathbf{I}}$ defines a trivial symplectic fibration, constant along $\XXX$. Nonetheless, the combination of these two structures defines a non-trivial structure of K\"ahler fibration on $\XXX \to \JJJ$, different from the one considered in Section \ref{ssec:potential}.

In order to achieve our goal in this section, we will apply Theorem \ref{th:Kfib}. As we have seen in Proposition \ref{p:dbarAphi}, the integrability condition 
$$
\dbar_{J,A}\psi^{1,0_J} = 0
$$ 
cuts a holomorphic submanifold $\XXX^{Higgs}$ of $(\XXX,\mathbb{I})$. Given that $\widehat{\boldsymbol{\omega}}_{\mathbf{I}}$ is constant on $\XXX$, it provides a candidate for a coupling 2-form $\boldsymbol{\sigma}$ on $(\XXX,\widehat{\boldsymbol{\omega}}_{\mathbf{I}})$, inducing the trivial Ehresmann connection on $\XXX \to \JJJ$, regarded as the product $\XXX = \JJJ\times \AAA\times \Omega^1(\Sigma,E_K(\liek))$. However, one can readily check that the horizontal subspace of this connection is not tangent to the submanifold $\XXX^{Higgs}$, and hence is not appropriate for our goals. Motivated by this, we consider the Ehresmann connection $\Gamma^{\mathbb{I}}$ with horizontal subspace
$$
H^{\mathbb{I}} = \Big\lbrace\left(\dot J,0,-\tfrac{1}{2}\psi(J \dot J)\right) \; \Big\vert \; \dot J \in T_J \JJJ\Big\rbrace.
$$

\begin{lemma}\label{l:existenceGammaI}
The Ehresmann connection $\Gamma^{\mathbb{I}}$ is preserved by $\mathbb{I}$, and
$$
\Gamma^{\mathbb{I}}\circ \mathbb{I} = \mathbf{I}\circ \Gamma^{\mathbb{I}},
$$
where $\mathbf{I}$ denotes the complex structure along the fibres. Furthermore, its curvature $F_ {\mathbb{I}} \in \Omega^{2}(\XXX,V\XXX)$ is of type $(1,1)$, and given explicitly by
\begin{equation}\label{eq:FGammaI}
\begin{split}
(F_ {\mathbb{I}})_{|(J,A,\psi)}(v_1,v_2) & = \frac{1}{4}\left(0,0,\psi([\dot J_1,\dot J_2])\right),
\end{split}
\end{equation}
for any pair of horizontal vector fields $v_1,v_2 \in H^\sigma$ covering $\dot J_1,\dot J_2$, respectively.
\end{lemma}

\begin{proof}
We take $v = (\dot J, a , \dot \psi) \in T_{(J, A , \psi)}\XXX$ and notice that
$$
\Gamma^{\mathbb{I}}(\dot J, a , \dot \psi) = \left(0,a, \dot \psi + \tfrac{1}{2}\psi(J\dot J)\right),
$$
which implies that
\begin{align*}
\Gamma^{\mathbb{I}}(\mathbb{I}v) & = \left(0,Ja, -J\dot \psi + \psi(\dot J) - \tfrac{1}{2}\psi(\dot J)\right) = \mathbf{I}\left(0,a,\dot \psi - \tfrac{1}{2}\psi(\dot J J)\right) = \mathbf{I} \Gamma^{\mathbb{I}}(v).
\end{align*}
In particular, this implies that $\Gamma^{\mathbb{I}}$ is preserved by $\mathbb{I}$. To calculate the formula for the curvature, we choose horizontal vectors 
$$
v_j = \left(\dot J, 0, -\tfrac{1}{2}\psi(J\dot J_j)\right) \in H^{\mathbb{I}}_{(J,A,\psi)},
$$
with $j=1,2$, and vector fields $Y_j$ such that $Y_j(J,a,\psi) = v_j$. By taking coordinates in $\JJJ$, we can choose $Y_j$ such that the $T\JJJ$ component of the Lie bracket $[Y_1,Y_2]$ vanishes at $(J,A,\psi)$. Then, we calculate
\begin{align*}
F_{\mathbb{I}}(v_1,v_2) & = - [Y_1,Y_2]_{|(J,A,\psi)}\\
& = -\frac{1}{2}\frac{d}{dt}_{|t =0}\left(0,0,- \psi_{t,1}(J_{t,1}Y_2) + \psi_{t,2}(J_{t,2}Y_1))\right)\\
& = -\frac{1}{2}\left(0,0,\frac{1}{2}\psi(J\dot J_1J\dot J_2) - \psi(\dot J_1 \dot J_2) - \frac{1}{2}\psi(J\dot J_2J\dot J_1) + \psi(\dot J_2 \dot J_1) \right)\\
& = \frac{1}{4}\left(0,0,\psi([\dot J_1,\dot J_2])\right).
\end{align*}
The fact that $F_{\mathbb{I}}$ is of type $(1,1)$ follows from this formula and the definition of $\mathbb{I}$.
\end{proof}

As a direct consequence of the equality $\Gamma^{\mathbb{I}}\circ \mathbb{I} = \mathbf{I}\circ \Gamma^{\mathbb{I}}$, it follows that $\widehat{\boldsymbol{\omega}}_{\mathbf{I}}(\Gamma^{\mathbb{I}},\Gamma^{\mathbb{I}})$ defines a $(1,1)$-form on $(\XXX,\mathbb{I})$ such that the associated Ehresmann connection equals $\Gamma^{\mathbb{I}}$. In the next result we calculate an explicit formula for this $(1,1)$-form, the proof being straightforward and hence omitted.

\begin{lemma}\label{lem:hatomegaI}
Given $(J,A,\psi) \in \XXX$ and any pair of tangent vectors $v_j = (\dot J_j, a_j , \dot \psi_j) \in T_{(J, A , \psi)}\XXX$, one has
\begin{equation}\label{eq:hatomegaI}
\begin{split}
\widehat{\boldsymbol{\omega}}_{\mathbf{I}}(\Gamma^{\mathbb{I}}v_1,\Gamma^{\mathbb{I}}v_2) %& =  \int_\Sigma B(a_1 \wedge a_2)\\
% & - \int_\Sigma B\left(\left(\dot \psi_1 + \tfrac{1}{2}\psi(J \dot J_1)\right) \wedge \left(\dot \psi_2 + \tfrac{1}{2}\psi(J \dot J_2)\right)\right)\\
& =  \int_\Sigma B(a_1 \wedge a_2) -  \int_\Sigma B(\dot\psi_1 \wedge \dot \psi_2)\\
& - \frac{1}{2}\int_\Sigma B\left(\psi(J\dot J_1) \wedge \dot \psi_2 \right) - \frac{1}{2}\int_\Sigma B\left(\dot \psi_1 \wedge \psi(J\dot J_2)\right)\\
& - \frac{1}{4} \int_\Sigma B\left(\psi(\dot J_1) \wedge \psi(\dot J_2) \right).
\end{split}
\end{equation}
\end{lemma}

Our next result proves that the Ehresmann connection $\Gamma^{\mathbb{I}}$ is K\"ahler, by calculation of an explicit coupling form $\boldsymbol{\sigma}^{\mathbb{I}}$ via equation \eqref{eq:couplingsigma}. Define a basic 2-form $\mu\in\Omega^2(\XXX)$ by
$$
\mu\left((\dot J_1,a_1,\dot \psi_1),(\dot J_2,a_2,\dot \psi_2)\right) = \frac{1}{4}\int_\Sigma B\left(\psi(\dot J_1) \wedge \psi(\dot J_2)\right).
$$

\begin{proposition}\label{p:existencesigmaI}
The Ehresmann connection $\Gamma^{\mathbb{I}}$ is K\"ahler. Furthermore, the following formula gives a coupling form for $\Gamma^{\mathbb{I}}$:
$$
\boldsymbol{\sigma}^{\mathbb{I}} := \widehat{\boldsymbol{\omega}}_{\mathbf{I}}(\Gamma^{\mathbb{I}},\Gamma^{\mathbb{I}}) - \mu.
$$
More explicitly,
\begin{equation}\label{eq:sigmaI}
\begin{split}
\boldsymbol{\sigma}^{\mathbb{I}}_{|(J,A,\psi)}(v_1,v_2) & =  \int_\Sigma B(a_1 \wedge a_2) - \int_\Sigma B\left(\left(\dot \psi_1 + \tfrac{1}{2}\psi(J \dot J_1)\right) \wedge \left(\dot \psi_2 + \tfrac{1}{2}\psi(J \dot J_2)\right)\right)\\
& - \frac{1}{4}\int_\Sigma B\left(\psi(\dot J_1) \wedge \psi(\dot J_2)\right)\\
& =  \int_\Sigma B(a_1 \wedge a_2) -  \int_\Sigma B(\dot\psi_1 \wedge \dot \psi_2)\\
& - \frac{1}{2}\int_\Sigma B\left(\psi(J\dot J_1) \wedge \dot \psi_2 \right) - \frac{1}{2}\int_\Sigma B\left(\dot \psi_1 \wedge \psi(J\dot J_2)\right)\\
& - \frac{1}{2}\int_\Sigma B\left(\psi(\dot J_1) \wedge \psi(\dot J_2)\right),
\end{split}
\end{equation}
for $v_j = (\dot J_j, a_j, \dot \psi_j) \in T_{(J, A , \psi)}\XXX$, $j = 1,2$.
\end{proposition}

\begin{proof}
Notice that $\mu$ is of type $(1,1)$ for $\mathbb{I}$, and hence so is $\boldsymbol{\sigma}^{\mathbb{I}}$. By Theorem \ref{th:Kfib}, it suffices to prove that $\boldsymbol{\sigma}^{\mathbb{I}}$ is closed. Using the abstract formula for the coupling form in \eqref{eq:couplingsigma}, this last fact reduces to the identity
$$
d\mu(v_1,v_2) = \widehat{\boldsymbol{\omega}}_{\mathbf{I}}(F_{\mathbb{I}}(v_1,v_2),)
$$
evaluated on vertical vector fields. Now, by direct calculation
\begin{align*}
d\mu(v_1,v_2)(0,a,\dot \psi) & = \frac{1}{4}\int_\Sigma B\left(\dot \psi(\dot J_1) \wedge \psi(\dot J_2)\right) + \frac{1}{4}\int_\Sigma B\left(\psi(\dot J_1) \wedge \dot \psi(\dot J_2)\right)\\
& = \frac{1}{4}\int_\Sigma B\left(\dot \psi \wedge \psi([\dot J_1,\dot J_2])\right)\\
& = - \widehat{\boldsymbol{\omega}}_{\mathbf{I}}((0,a,\dot \psi),F_{\mathbb{I}}(v_1,v_2)).
\qedhere
\end{align*}
\end{proof}

To finish this section, we provide a formula for the symmetric tensor on $\XXX$ associated to the coupling form $\boldsymbol{\sigma}^{\mathbb{I}}$, explicitly given by
$$
\mathbf{g}_{\mathbb{I}} := \boldsymbol{\sigma}^{\mathbb{I}}(,\mathbb{I}).
$$
By construction, this coincides with the flat hyperk\"ahler metric \eqref{eq:gHK} along the fibres of $\XXX \to \JJJ$. As in Corollary \ref{cor:negative}, $\mathbf{g}_{\mathbb{I}}$ is negative semi-definite along the horizontal directions of the connection $\Gamma^{\mathbb{I}}$.

\begin{corollary}\label{cor:negativeI}
For any tangent vector $v = (\dot J, a,\dot \psi)\in T_{(J,A,\psi)}\XXX$ one has
\begin{equation}\label{eq:gI}
\begin{split}
\mathbf{g}_{\mathbb{I}}(v,v) & =  \int_\Sigma B(a \wedge J a) + \int_\Sigma B\left(\left(\dot \psi + \tfrac{1}{2}\psi(J \dot J)\right) \wedge J\left(\dot \psi + \tfrac{1}{2}\psi(J \dot J)\right)\right)\\
& - \frac{1}{4}\int_\Sigma B\left(\psi(\dot J) \wedge J \psi(\dot J)\right).
\end{split}
\end{equation}
In particular, given a horizontal vector field $v \in H^{\mathbb{I}}$ at $(J,A,\psi)$, covering $\dot J \in T_J \JJJ$, one has
\begin{equation}\label{eq:gIhor}
\begin{split}
\mathbf{g}_{\mathbb{I}}(v,v) %& = \int_\Sigma B(\psi(\dot J))\wedge \psi(\dot J J))\\
%& 
= - \frac{1}{4}\int_\Sigma B\left(\psi(\dot J)\wedge J (\psi(\dot J))\right).
\end{split}
\end{equation}
Consequently, $\mathbf{g}_{\mathbb{I}}$ is negative semi-definite along the horizontal directions of $\Gamma^{\mathbb{I}}$.
\end{corollary}

\begin{proof}
Formula \eqref{eq:gI} follows from \eqref{eq:sigmaI}, using that $\Gamma^{\mathbb{I}}\circ \mathbb{I} = \mathbf{I}\circ \Gamma^{\mathbb{I}}$ (see Lemma \ref{l:existenceGammaI}).
\end{proof}

%%%%%%%%%%%%%%%%%%%%%%%%%%%%%%%%%%%%%%%%%%%%%%%%%
\subsection{Hamiltonian action and the coupled Hitchin equations}\label{ssec:UHiggsHam}
%%%%%%%%%%%%%%%%%%%%%%%%%%%%%%%%%%%%%%%%%%%%%%%%%

We follow the notation of the previous section. 
Our aim now is to investigate a natural Hamiltonian action for the minimal couplings defined by the coupling form $\boldsymbol{\sigma}^{\mathbb{I}}$ associated to the connection  $\Gamma^{\mathbb{I}}$, on the K\"ahler fibration $(\XXX,\mathbb{I},\widehat{\boldsymbol{\omega}}_{\mathbf{I}})$ over the space of complex structures $\JJJ$.

We fix a symplectic form $\omega$ on $\Sigma$, compatible with the orientation. Given $\alpha > 0$ a real `coupling constant' and $\varepsilon \in \{-1,1\}$, the family of minimal coupling symplectic structures of our interest is defined by
\begin{equation}\label{eq:uomegaI}
\boldsymbol{\omega}^{\mathbb{I}}_{\alpha,\varepsilon} = \varepsilon\boldsymbol{\omega}_\JJJ + \alpha \boldsymbol{\sigma}^{\mathbb{I}},
\end{equation}
where $\boldsymbol{\sigma}^{\mathbb{I}}$ is the closed $(1,1)$-form in Proposition  \ref{p:existencesigmaI} and $\boldsymbol{\omega}_\JJJ$ is as in \eqref{eq:SympJ}. Notice that $\boldsymbol{\omega}_\JJJ$ depends on the choice of $\omega$. By construction, $\boldsymbol{\omega}^{\mathbb{I}}_{\alpha,\varepsilon}$ is closed and of type $(1,1)$ with respect to the complex structure $\mathbb{I}$. Furthermore, along the fibres of $\XXX \to \JJJ$ the $2$-form $\boldsymbol{\omega}^{\mathbb{I}}_{\alpha,\varepsilon}$ restricts to the K\"ahler structure $\alpha \widehat{\boldsymbol{\omega}}_{\mathbf{I}}$. Consider the associated symmetric tensor 
\begin{equation}\label{eq:ugI}
\mathbf{g}^{\mathbb{I}}_{\alpha,\varepsilon} := \boldsymbol{\omega}^{\mathbb{I}}_{\alpha,\varepsilon} (,\mathbb{I}) = \varepsilon \mathbf{g}_{\JJJ} + \alpha \mathbf{g}_{\mathbb{I}}.
\end{equation}
The following result is the analogue of Lemma~\ref{lemma:indef} in the present setup. Its proof is analogous, by application of Corollary~\ref{cor:negativeI} in this case, and is omitted.

\begin{lemma}\label{lemma:indefI}
Let $v \in H^{\mathbb{I}}$ be a horizontal vector field at $(J,A,\psi)$, covering $\dot J \in T_J \JJJ$. 
Then
\begin{equation}\label{eq:gJhor.2}
\begin{split}
\mathbf{g}^{\mathbb{I}}_{\alpha,\varepsilon}(v,v) = \frac{\varepsilon}{2}\int_{\Sigma}\tr(\dot J\dot J) \omega - \frac{\alpha}{4} \int_\Sigma B\left(\psi(\dot J)\wedge J (\psi(\dot J))\right).
\end{split}
\end{equation}
Consequently, for any $\alpha >0$:
\begin{enumerate}

\item If $\varepsilon = -1$, $\mathbf{g}^{\mathbb{I}}_{\alpha,\varepsilon}$ is negative definite along $H^{\mathbb{I}}$. In particular, $\boldsymbol{\omega}^{\mathbb{I}}_{\alpha,-1}$ is a non-degenerate symplectic structure.

\item If $\varepsilon = 1$ and $\psi \neq 0$, $\mathbf{g}^{\mathbb{I}}_{\alpha,\varepsilon}$ changes signature along the line $(J,A,\lambda\psi) \in \XXX$, for $\lambda \in \RR$.

\end{enumerate}

\end{lemma}

Our next goal is to prove that $(\XXX,\boldsymbol{\omega}^{\mathbb{I}}_{\alpha,\varepsilon})$ admits a Hamiltonian action by the extended gauge group $\widetilde{\KKK}$, determined by $\omega$ and the reduction $E_K \subset E_G$. For this, we will need some refined information about the coupling form $\boldsymbol{\sigma}^{\mathbb{I}}$. Consider the closed 2-form $\pi^*_2\omega_{\mathbf{I}}$ on $\XXX$, given by the pull-back of $\omega_{\mathbf{I}} = \operatorname{Re}\Omega_{\mathbf{J}}$ via the canonical projection $\pi_2 \colon \XXX \to \AAA\times \Omega^1(\Sigma,E_K(\liek))$. Explicitly, for $v_j = (\dot J_j, a_j, \dot \psi_j) \in T_{(J, A , \psi)}\XXX$, $j = 1,2$, one has
$$
\pi^*_2\omega_{\mathbf{I}}(v_1,v_2) = \int_\Sigma B(a_1 \wedge a_2) -  \int_\Sigma B(\dot\psi_1 \wedge \dot \psi_2).
$$
The 2-form $\pi^*_2\omega_{\mathbf{I}}$ restricts on the fibres to $\widehat{\boldsymbol{\omega}}_{\mathbf{I}}$, and hence, by the general theory of symplectic fibrations (see \cite[Section 1.6]{GLS}), one expects that there exists a 1-form $\lambda \in \Omega^1(\XXX)$ such that
\begin{equation}\label{eq:dlambda}
\boldsymbol{\sigma}^{\mathbb{I}} = \pi^*_2\omega_{\mathbf{I}} + d\lambda.
\end{equation}
In the next result we prove that this is the case in the present situation, providing a refined version of the previous formula.

\begin{lemma}\label{l:sigmaIddc}
Consider the 1-form $\lambda \in \Omega^1(\XXX)$ defined by
$$
\lambda(v) = \frac{1}{4}\int_\Sigma B\left(\psi(J\dot J) \wedge \psi \right)
$$
for $v = (\dot J,a, \dot \psi) \in T_{(J, A , \psi)}\XXX$. Then formula \eqref{eq:dlambda} holds. Furthermore, considering the $L^2$-norm of the `unitary Higgs field',
$$
\nu(J,A,\psi) = \|\psi\|^2_{L^2} := \int_\Sigma B(\psi \wedge J \psi),
$$
one has
\begin{equation}\label{eq:sigmaIabs}
\boldsymbol{\sigma}^{\mathbb{I}} = \pi^*_2\omega_{\AAA} + \tfrac{i}{2}\partial_{\mathbb{I}}\dbar_{\mathbb{I}}\nu,
\end{equation}
where 
$$
\omega_{\AAA}((a_1,\dot \psi_1),(a_2,\dot \psi_2)) = \int_\Sigma B(a_1\wedge a_2).
$$
\end{lemma}

\begin{proof}
By the proof of Lemma \ref{lem:sigmaex} and \eqref{eq:uI}, we have
\begin{align*}
d^c_{\mathbb{I}}\nu_{|(J,A,\psi)}(\dot J,a,\dot \psi) & = 2\int_\Sigma B((J \dot \psi - \psi(\dot J)) \wedge J \psi) + \int_\Sigma B(\psi \wedge \psi(J\dot J))\\
& = 2\int_\Sigma B(\dot \psi \wedge \psi) - 2\int_\Sigma B(\psi(\dot J J ) \wedge \psi) + \int_\Sigma B(\psi \wedge \psi(J\dot J))\\
& = 2\int_\Sigma B(\dot \psi \wedge \psi) + \int_\Sigma B(\psi(J\dot J) \wedge \psi)\\
& = 2\int_\Sigma B(\dot \psi \wedge \psi) + 4 \lambda.
\end{align*}
We now calculate
\begin{align*}
4 d\lambda((\dot J_1,a_1,\dot \psi_1),(\dot J_2,a_2,\dot \psi_2)) & = \int_\Sigma B(\dot \psi_1(J\dot J_2) \wedge \psi) - \int_\Sigma B(\dot \psi_2(J\dot J_1) \wedge \psi)\\
& + \int_\Sigma B(\psi(J\dot J_2) \wedge \dot \psi_1) - \int_\Sigma B(\psi(J\dot J_1) \wedge \dot \psi_2)\\
& + \int_\Sigma B(\psi(\dot J_1\dot J_2) \wedge \psi) - \int_\Sigma B(\psi(\dot J_2\dot J_1) \wedge \psi)\\
& = - 2\int_\Sigma B(\dot \psi_1 \wedge \psi(J\dot J_2)) - 2 \int_\Sigma B(\psi(J\dot J_1) \wedge \dot \psi_2)\\
& - 2\int_\Sigma B(\psi(\dot J_1) \wedge \psi(\dot J_2))\\
& = 4 \boldsymbol{\sigma}^{\mathbb{I}} - 4 \pi^*_2\omega_{\mathbf{I}},
\end{align*}
which proves the first of the statement. As for the final part, from the previous formulae we conclude
\begin{align*}
dd^c_{\mathbb{I}}\nu & = 2i\partial_{\mathbb{I}}\dbar_{\mathbb{I}}\nu \\
& = - 4 \int_\Sigma B(\dot \psi_1 \wedge \dot \psi_2) + 4 d\lambda \\
& = - 4 \int_\Sigma B(\dot \psi_1 \wedge \dot \psi_2) + 4 \boldsymbol{\sigma}^{\mathbb{I}} - 4 \pi^*_2\omega_{\mathbf{I}}\\
& = -4 \pi^*_2\omega_{\AAA} + 4 \boldsymbol{\sigma}^{\mathbb{I}}.
\qedhere
\end{align*}
\end{proof}

Consider the extended gauge group $\widetilde{\KKK}$ of $h$ and $\omega$. As in \cite[Section 2.2]{AGG1}, the group $\widetilde{\KKK}$ acts on $(\XXX,\mathbb{I},\boldsymbol{\omega}^{\mathbb{I}}_{\alpha,\varepsilon})$ preserving $\mathbb{I}$ and $\boldsymbol{\omega}^{\mathbb{I}}_{\alpha,\varepsilon}$, and covering the $\HHH$-action on $\JJJ$ by push-forward. We are ready to prove the main result of this section. We use the same notation as in Proposition \ref{p:mmap}.

\begin{proposition}\label{p:mmapI}
The action of $\widetilde{\KKK}$ on $(\XXX,\boldsymbol{\omega}^{\mathbb{I}}_{\alpha,\varepsilon})$ is Hamiltonian, with equivariant moment map given by
\begin{align*}
\langle \mu_{\widetilde{\KKK}}^{\mathbb{I}}(J,A,\psi), \zeta \rangle & = \alpha \int_\Sigma B\left(F_A - \tfrac{1}{2}[\psi\wedge \psi],A\zeta\right) \\
& - \int_\Sigma f \left(\varepsilon S_J \omega + \alpha d\left(B\left(\Lambda_\omega(d_A \psi)\right),\psi\right) - \alpha d^c\left(B \left(\psi,\Lambda_\omega(d_A(J\psi))\right)\right)\right).
\end{align*}
\end{proposition}

\begin{proof}
By Lemma \ref{l:sigmaIddc} we have that $\boldsymbol{\sigma}^{\mathbb{I}} = \pi^*_2\omega_{\mathbf{I}} + d \lambda$, and since $\lambda$ is preserved by the $\widetilde{\KKK}$-action, there exists an equivariant moment map for $\boldsymbol{\omega}^{\mathbb{I}}_{\alpha,\varepsilon}$ given by 
$$
\langle \mu_{\widetilde{\KKK}}^{\mathbb{I}}(J,A,\psi), \zeta \rangle = - \varepsilon \int_\Sigma f S_J \omega + \alpha \operatorname{Re} \int_\Sigma B(F_D, D \zeta) - \alpha\lambda(\zeta \cdot (J,A,\psi)),
$$
where $D = A + i \psi$ and we have used that $\omega_{\mathbf{I}} = \operatorname{Re}\Omega_{\mathbf{J}}$ and Proposition \ref{Flat-exreduction}. The proof follows from the formulae
\begin{align*}
\operatorname{Re} \int_\Sigma B(F_D, D \zeta) & = \int_\Sigma B\left(F_A - \tfrac{1}{2}[\psi\wedge \psi],A\zeta\right) - B\left(d_A \psi,\psi(y)\right)\\
& = \int_\Sigma B\left(F_A - \tfrac{1}{2}[\psi\wedge \psi],A\zeta\right) - B\left(\Lambda_\omega(d_A \psi),\psi\right)\wedge i_y\omega\\
& = \int_\Sigma B\left(F_A - \tfrac{1}{2}[\psi\wedge \psi],A\zeta\right) + \int_\Sigma df \wedge \left(B\left(\Lambda_\omega(d_A \psi)\right),\psi\right),\\
& = \int_\Sigma B\left(F_A - \tfrac{1}{2}[\psi\wedge \psi],A\zeta\right) - \int_\Sigma f d\left(B\left(\Lambda_\omega(d_A \psi)\right),\psi\right),
\end{align*}
where $\zeta$ covers the Hamiltonian vector field with Hamiltonian $f \in C^\infty_0(\Sigma)$, and (see the proof of Proposition \ref{p:mmap}):
\begin{align*}
\lambda(\zeta \cdot (J,A,\psi)) & = - \frac{1}{4}\int_\Sigma B(\psi(JL_yJ)\wedge \psi)\\
& = - \int_\Sigma B (\psi(Jy),d_A(J\psi))\\
& = - \int_\Sigma B (\psi,\Lambda_\omega(d_A(J\psi))) \wedge i_{Jy}\omega\\
& = \int_\Sigma i_{Jy}\omega \wedge B (\psi,\Lambda_\omega(d_A(J\psi))) \\
& = \int_\Sigma d^cf \wedge B \left(\psi,\Lambda_\omega(d_A(J\psi)))\right)\\
& = - \int_\Sigma f d^c\left(B \left(\psi,\Lambda_\omega(d_A(J\psi))\right)\right).
\qedhere
\end{align*}
%\begin{align*}
%\langle \mu_{\widetilde{\KKK}}(J,A,\psi), \zeta \rangle & =  - \varepsilon \int_\Sigma f S_J \omega + \alpha \operatorname{Re} \int_\Sigma B(F_D, D \zeta) - \alpha\lambda(\zeta \cdot (J,A,\psi))\\
%& = \alpha \int_\Sigma B\left(F_A - \tfrac{1}{2}[\psi\wedge \psi],A\zeta\right) \\
%& - \varepsilon \int_\Sigma f S_J \omega - \alpha \int_\Sigma f d\left(B\left(\Lambda_\omega(d_A \psi)\right),\psi\right) + \alpha  \int_\Sigma f d^c\left(B \left(\psi,\Lambda_\omega(d_A(J\psi))\right)\right)\\
%& = \alpha \int_\Sigma B\left(F_A - \tfrac{1}{2}[\psi\wedge \psi],A\zeta\right) \\
%& - \int_\Sigma f \left(\varepsilon S_J \omega + \alpha d\left(B\left(\Lambda_\omega(d_A \psi)\right),\psi\right) - \alpha d^c\left(B \left(\psi,\Lambda_\omega(d_A(J\psi))\right)\right)\right)
%\end{align*}

%\begin{align*}
%\int_\Sigma B(\psi(J L_{y} J) \wedge \psi) & = - \int_\Sigma B(\psi \wedge \psi(J L_{y} J))\\
%& = \int_\Sigma B(\psi \wedge \psi((L_{y} J)J))\\
%& = \int_\Sigma B(J \psi \wedge \psi(L_{y} J))\\
% & = 4 \int_\Sigma B (\psi(Jy),d_A(J\psi)).
%\end{align*}
\end{proof}

As in Proposition \ref{p:mmap}, the $\widetilde{\KKK}$-action produces a `coupling term' on the base $\Sigma$ (cf. Remark~\ref{rem:couplingterm}), which interacts with the scalar curvature of the metric $g = \omega(,J)$. 
In particular, as in \cite{AGG1}, zeros of the moment map $\mu_{\widetilde{\KKK}}^\mathbb{I}$ are given by solutions to the coupled system of equations
\begin{equation}\label{eq:cHitchin}
\begin{split}
F_A - \tfrac{1}{2}[\psi\wedge \psi] & = 0,\\
\varepsilon S_J + \alpha \Lambda_\omega \left(d\left(B\left(\Lambda_\omega(d_A \psi)\right),\psi\right) - d^c\left(B \left(\psi,\Lambda_\omega(d_A(J\psi))\right)\right)\right) & = \varepsilon \frac{2\pi \chi(\Sigma)}{V}.
\end{split}
\end{equation}
In the next section, we combine these equations with the integrability condition $\dbar_A \varphi  = 0$ for the Higgs field, in order to introduce the \emph{universal moduli space of Hitchin's equations}.

\subsection{Metric structure of the universal moduli of Hitchin's equations}
\label{section-Uhiggsmetric}

In this section we combine the results of the previous two sections to define a universal moduli space of solutions of Hitchin's equations~\eqref{hitchin}, varying over (a cover of) the Teichm\"uller space. 
As in the previous sections, we consider a fixed principal $K$-bundle $E_K$ over a smooth compact oriented surface $\Sigma$ with fixed symplectic form $\omega$.

\begin{definition}
A triple $(J,A, \psi) \in \XXX \cong \JJJ\times \AAA\times \Omega^1(\Sigma,E_K(\liek))$ is a solution of the \emph{coupled Hitchin equations}, if the following conditions are satisfied
\begin{equation}\label{eq:cHitchinint}
\begin{split}
F_A -[\varphi,\tau(\varphi)] & = 0,\\
\dbar_{J,A}\varphi & = 0,\\
S_g & = \frac{2\pi \chi(\Sigma)}{V},
\end{split}
\end{equation}
where $g = \omega(,J)$ and $\varphi = \psi^{1,0_J}$.
\end{definition}

As in Section \ref{ssec:UHiggs}, consider $\XXX^{Higgs} \subset \XXX$ the subset defined by the condition $\dbar_{J,A}\varphi = 0$. By Proposition \ref{p:dbarAphi}, $\XXX^{Higgs}$ is formally a complex submanifold of $(\XXX,\mathbb{I})$ preserved by the $\widetilde{\KKK}$-action. For any choice of coupling constant $\alpha >0$ and parameter $\varepsilon \in \{-1,1\}$, $\XXX^{Higgs}$ inherits a minimal coupling structure $\boldsymbol{\omega}^{\mathbb{I}}_{\alpha,\varepsilon}$ and moment map $\mu_{\widetilde{\KKK}^{\mathbb{I}}|\XXX^{Higgs}}$, as in Proposition \ref{p:mmapI}. We define the \emph{universal moduli space of solutions of Hitchin's equations} as the symplectic quotient
$$
\mathcal{U}^{Hit}(G)_\alpha^\varepsilon := (\mu_{\widetilde{\KKK}^{\mathbb{I}}|\XXX^{Higgs}})^{-1}(0)/\widetilde{\KKK}.
$$
By the explicit formula for the moment map in Proposition~\ref{p:mmapI}, the underlying space of $\mathcal{U}^{Hit}(G)_\alpha^\varepsilon$ is independent of the choice of parameters, since the `coupling term' in the scalar equation in~\eqref{eq:cHitchin} vanishes identically when the integrability condition $\dbar_{J,A}\varphi = 0$ is imposed. 
We shall denote this moduli space simply by $\mathcal{U}^{Hit}(G)$. 
Hence, for any $\alpha >0$ and $\varepsilon \in \{-1,1\}$, the moduli space $\mathcal{U}^{Hit}(G)$ parametrizes solutions of the equations \eqref{eq:cHitchinint} modulo the $\widetilde{\KKK}$-action, and the only difference is the induced presymplectic structure on the moduli space. 
This structure is furthermore symplectic in the case $\varepsilon = -1$, by Lemma \ref{lemma:indefI}. 
By construction, there are natural maps (cf. \eqref{eq:FibTeichmuller})
$$
\mathcal{U}^{Hit}(G) \lra \mathcal{U}^{Higgs}(G) \lra \mathcal{T}.
$$
Hence, in particular, $\mathcal{U}^{Hit}(G)$ can be regarded as a fibration over the Teichm\"uller space $\mathcal{T}$. 
As in Section~\ref{section-Uhitchin}, since the symmetric tensor $\mathbf{g}^{\mathbb{I}}_{\alpha,\varepsilon}$ is not positive definite (see Lemma~\ref{lemma:indefI}), it is not obvious \emph{a priori} whether $\mathcal{U}^{Hit}(G)$ inherits a complex structure compatible with the (pre)symplectic structure induced by $\boldsymbol{\omega}^{\mathbb{I}}_{\alpha,\varepsilon}$. 
The main goal of this section is to study sufficient conditions under which this natural condition holds for the moduli space, and to prove that the map $\mathcal{U}^{Hit}(G) \to \mathcal{U}^{Higgs}(G)$ is holomorphic.

The analysis is analogous to that in Sections~\ref{section-Uhitchin} and~\ref{coupling}, so we only outline the argument. 
The first step is to undertake a \emph{gauge fixing} for solutions of the coupled Hitchin equations \eqref{eq:cHitchinint}, whereby the complex structure \eqref{eq:uI} and the symmetric tensor $\mathbf{g}^{\mathbb{I}}_{\alpha,\varepsilon}$ descend to the moduli space. We observe that this will produce a priori different complex structures on $\mathcal{U}^{Hit}(G)$.

We fix a solution $(J,A,\psi)$ of the coupled Hitchin equations \eqref{eq:cHitchinint}.

\begin{lemma}\label{lem:linearHit}
The linearization of the coupled Hitchin equations \eqref{eq:cHitchinint} at $(J,A,\psi)$ is given by
\begin{equation}\label{eq:linearHit}
\begin{split}
d_Aa - [\dot \psi,\psi] & = 0,\\
d_A \dot \psi + [a,\psi] & = 0, \\ 
d_A (J \dot \psi) + [a, J \psi] - d_A (\psi( \dot J)) & = 0, \\ 
\varepsilon \delta S(\dot J) & = 0,
\end{split}
\end{equation}
where $\delta S \colon T_J\JJJ \to C_0^\infty(\Sigma,\RR)$ is the linearization of the scalar curvature.
\end{lemma}

We denote by ${\bf L} (\dot J,a, \dot \psi)$ the differential operator defined by the left-hand side of equations \eqref{eq:linearHit}. Using the same notation as in \eqref{eq:TWMcomplexring}, we define a complex of linear differential operators
\begin{equation}\label{eq:TWMHitchin}
\begin{tikzcd}
(\mathcal{S}^*_0)\colon & 0 \ar[r] & \mathcal{S}^0 \ar[r,"{\mathbf{P}}"] & \mathcal{S}^1 \ar[r,"{\bf L}"] & \mathcal{S}^2 \ar[r] & 0.
\end{tikzcd}
\end{equation}
The cohomology $H^1(\mathcal{S}^*_0) := \frac{\ker {\bf L}}{\operatorname{Im}\, \mathbf{P}}$ can be formally identified with the tangent space $T_{[(J,A,\psi)]} \mathcal{U}^{Hit}(G)$. Our next result shows that the moduli space $ \mathcal{U}^{Hit}(G)$ is finite dimensional.

\begin{lemma}\label{lem:sesTWMHit}
The sequence \eqref{eq:TWMHitchin} is an elliptic complex of multi-degree linear differential operators. Consequently, the cohomology groups $H^j(\cS^*_0)$, with $j = 0,1,2$, are finite-dimensional.
\end{lemma}

In order to construct a complex structure induced by \eqref{eq:uI} on the moduli space, we work orthogonal to the image of the infinitesimal action operator $\mathbf{P}$ in \eqref{eq:TWMHitchin} with respect to the indefinite pairing $\mathbf{g}^{\mathbb{I}}_{\alpha,\varepsilon}$. The existence of this complex structure will automatically yield a symmetric tensor of type $(1,1)$, since the 2-form $\boldsymbol{\omega}^{\mathbb{I}}_{\alpha,\varepsilon}$ is well defined on the cohomology $H^1(\mathcal{S}^*_0)$ by Proposition \ref{p:mmapI}. Consider the $L^2$-pairing on $\mathcal{S}^0$ define by \eqref{eq:L2ell}. Consider the map $\mu \colon \XXX \to \mathcal{S}^0$ defined by (cf. Proposition \ref{p:mmapI})
$$
\mu^\mathbb{I}(J,A,\psi) = \left(\mu^\mathbb{I}_0,\mu^\mathbb{I}_1\right),
$$
where 
\begin{align*}
\mu^\mathbb{I}_0 & = -\left(\varepsilon S_J \omega + \alpha d\left(B\left(\Lambda_\omega(d_A \psi)\right),\psi\right) - \alpha d^c\left(B \left(\psi,\Lambda_\omega(d_A(J\psi))\right)\right)\right),\\
\mu^\mathbb{I}_0 & = \alpha\left(F_A - \tfrac{1}{2}[\psi \wedge \psi]\right).
\end{align*}

\begin{lemma}\label{lem:P*}
The following operator provides a formal adjoint of the infinitesimal action $\mathbf{P}$ for the pairings \eqref{eq:L2ell} and \eqref{eq:ugI}:
$$
\mathbf{P}^*_{\alpha,\varepsilon} = \delta \mu^{\mathbb{I}} \circ \mathbb{I} \colon \mathcal{S}^1 \lra \mathcal{S}^0.
$$
\end{lemma}
 	
%\begin{proof}
%The proof follows from a straightforward calculation using the formal properties of the moment map in Proposition \ref{p:mmapI}. Let $(f,u) \in \mathcal{S}^0$ and define
%$$
%\zeta = u + A^\perp(\eta_f).
%$$
%Notice that
%\begin{equation}\label{eq:modinfidea}
%\zeta \cdot (J,A,\Psi) = \mathbf{P}(f,u).
%\end{equation}
%Then, setting $v = (\dot J,a,\dot \psi)$, we have
%\begin{equation*}
%\begin{split}
%\langle \mathbf{P}^*v,(f,u)\rangle = \langle d\mu_{\widetilde{\KKK}}^{\mathbb{I}}(\mathbb{I}v), \zeta \rangle = \boldsymbol{\omega}^{\mathbb{I}}_{\alpha,\varepsilon}(\zeta \cdot (J,A,\psi),\mathbb{I}v) = \mathbf{g}^{\mathbb{I}}_{\alpha,\varepsilon}(\mathbf{P}(f,u),v) = \mathbf{g}^{\mathbb{I}}_{\alpha,\varepsilon}(v,\mathbf{P}(f,u)).
%	\end{split}
%\end{equation*}
%\end{proof}

Consider now the differential operator 
\begin{equation}
   \label{eq:Loperator}
 \begin{array}{cccl}
  \mathcal{L}_0^{\alpha,\varepsilon}\colon & \mathcal{S}^0  & \longrightarrow & \mathcal{S}^0 \\
         & (f,u) & \longmapsto & \mathbf{P}^*_{\alpha,\varepsilon}\circ \mathbf{P}(f,u).\\
         \end{array}
 \end{equation}
The key condition on the solution $(J,A,\psi)$ of \eqref{eq:cHitchinint} which we need to assume in order to construct the complex structure on the moduli space is the vanishing of the kernel of $\mathcal{L}_0^{\alpha,\varepsilon}$. Notice that, unlike in the standard cases in gauge theory in which the parameter space metric is positive definite, $\ker \mathcal{L}_0^{\alpha,\varepsilon}$ does not relate in general to automorphisms of the triple $(J,\psi,A)$, but rather to null vectors with respect to $\mathbf{g}^{\mathbb{I}}_{\alpha,\varepsilon}$.

\begin{proposition}\label{prop:zeroindexI}
The operator $\mathcal{L}_0^{\alpha,\varepsilon}$ is Fredholm with zero index. Furthermore, elements $(f,u) \in \ker \mathcal{L}_0^{\alpha,\varepsilon}$ are smooth.
\end{proposition}

Assuming that $\ker \mathcal{L}_0^{\alpha,\varepsilon}$ is trivial, in the next result we obtain a natural gauge fixing via a $\mathbf{g}^{\mathbb{I}}_{\alpha,\varepsilon}$-orthogonal decomposition
\begin{equation}\label{eq:perp}
\mathcal{S}^1 = \operatorname{Im} \; \mathbf{P} \oplus (\operatorname{Im} \; \mathbf{P})^{\perp_{\mathbf{g}^{\mathbb{I}}_{\alpha,\varepsilon}}}.
\end{equation}

\begin{lemma}\label{lem:gaugefixingI}
Assume that $\ker \mathcal{L}_0^{\alpha,\varepsilon} = \{0\}$. Then, there exists an orthogonal decomposition \eqref{eq:perp} for the pairing $\mathbf{g}^{\mathbb{I}}_{\alpha,\varepsilon}$. Consequently, for any element $v \in \mathcal{S}^1$ there exists a unique $\Pi v \in \operatorname{Im} \; \mathbf{P}$ such that $(\dot J,a, \dot \psi) = v - \Pi v$ solves the linear equation
\begin{equation*}\label{eq:gaugefixingI}
\begin{split}
\mathbf{P}^*_{\alpha,\varepsilon}(\dot J,a, \dot \psi) = \delta \mu^{\mathbb{I}} \circ \mathbb{I}(\dot J,a, \dot \psi) = 0.
\end{split}
\end{equation*}
\end{lemma}

The above lemma suggests to define the space of harmonic representatives of the complex \eqref{eq:TWMcomplexring}, as follows:
$$
\mathcal{H}^1(\cS^*_0)=\ker \bf{L} \cap \ker {\bf P}^*_{\alpha,\varepsilon}.
$$
It is important to observe that, on the subspace $\ker \bf{L}$, one has an equality
$$
{\bf P}^*_{\alpha,\varepsilon}(v) = {\bf L} \circ \mathbb{I}(v), \qquad \textrm{ for any } v \in \ker \bf{L}.
$$
Consequently, we have
$$
\mathcal{H}^1(\cS^*_0)=\ker \bf{L} \cap \ker ({\bf L} \circ \mathbb{I}).
$$
Our next result provides our gauge fixing mechanism for the linearization of the coupled Hitchin equations \eqref{eq:cHitchinint}. 

\begin{proposition}\label{prop:lineargaugefixedI}
Assume $\ker \mathcal{L}_0^{\alpha,\varepsilon} = \{0\}$. Then, the inclusion $\mathcal{H}^1(\cS^*_0)\subset \ker \bf{L}$ induces an isomorphism
$$
\mathcal{H}^1(\cS^*_0)\cong H^1(\cS^*_0).
$$
More precisely, any class in the cohomology $H^1(\cS^*_0)$ of the complex \eqref{eq:TWMcomplexring} admits a unique representative $(\dot J,a,\dot \psi)$ solving the linear equations
\begin{equation}\label{eq:lineargaugefixedI}
{\bf L}(\dot J,a,\dot \psi) = 0, \qquad {\bf L} \circ \mathbb{I} (\dot J,a,\dot \psi) = 0.
\end{equation}
\end{proposition}

We are ready to prove our main result, which shows that the gauge fixing in Proposition \ref{prop:lineargaugefixedI} enables us to descend the complex structure $\mathbb{I}$ in $\XXX$ and the symmetric tensor $\mathbf{g}^{\mathbb{I}}_{\alpha,\varepsilon}$, to an open subset of the moduli space $\mathcal{U}^{Hit}(G)$, via the symplectic reduction in Proposition \ref{p:mmapI}. Define
$$
\mathcal{U}^*_{\alpha,\varepsilon} = \{[(J,A,\psi)] \; | \; \ker \mathcal{L}_0^{\alpha,\varepsilon}= \{0\} \} \subset \mathcal{U}^{Hit}(G).
$$

\begin{theorem}\label{thm:metricI}
For any coupling constant $\alpha > 0$ and parameter $\varepsilon \in \{-1,1\}$, the set $\; \mathcal{U}^*_{\alpha,\varepsilon}$ is open in $\; \mathcal{U}^{Hit}(G)$. For any smooth point $[(J,A,\psi)] \in \mathcal{U}^*_{\alpha,\varepsilon}$ the tangent space to $\mathcal{U}^{Hit}(G)$ at $[(J,A,\psi)]$, identified with the space of solutions of the gauge fixed linear equations \eqref{eq:lineargaugefixedI}, inherits a complex structure $\mathbb{I}$, independent of $\alpha$ and $\varepsilon$, and a symmetric tensor $\mathbf{g}^\mathbb{I}_{\alpha,\varepsilon}$ such that $\boldsymbol{\omega}^\mathbb{I}_{\alpha,\varepsilon} = \mathbf{g}^\mathbb{I}_{\alpha,\varepsilon}(\mathbb{I},)$, given respectively by \eqref{eq:uI} and \eqref{eq:gI}, and where $\boldsymbol{\omega}^\mathbb{I}_{\alpha,\varepsilon}$ stands for the restriction of \eqref{eq:uomegaI}. Furthermore,

\begin{enumerate}

\item if $\varepsilon = 1$ the tensor $\mathbf{g}^\mathbb{I}_{\alpha,\varepsilon}$ is possibly degenerate,

\item if $\varepsilon = -1$ the tensor $\mathbf{g}^\mathbb{I}_{\alpha,\varepsilon}$ is non-degenerate, and defines a pseudo-K\"ahler structure on the moduli space.

\end{enumerate}

\end{theorem}

To finish this section, we provide an explicit formula for the (pre)symplectic structure $\boldsymbol{\omega}^\mathbb{I}_{\alpha,\varepsilon}$. The proof is straightforward from the previous discussion and Proposition \ref{p:existencesigmaI}.

\begin{corollary}\label{cor:omegaexpI}
Let $[(J,A,\psi)] \in \mathcal{U}^*_{\alpha,\varepsilon}$ be a smooth point and take $v_1,v_2$ tangent vectors of $\mathcal{U}^{Hit}(G)$ at $[(J,A,\psi)]$, identified with solutions $(\dot J_j,a_j,\dot\psi_j)$ of the gauge fixed linear equations \eqref{eq:lineargaugefixedI}. Then, one has
\begin{equation}\label{eq:omegagmoduliexI}
\begin{split}
\boldsymbol{\omega}^{\mathbb{I}}_{\alpha,\varepsilon}(v_1,v_2) & = \frac{\varepsilon}{2}\int_{\Sigma}\tr(J\dot J_1\dot J_2) \omega \\
& + \alpha \int_\Sigma B(a_1 \wedge a_2) -  \alpha \int_\Sigma B(\dot\psi_1 \wedge \dot \psi_2)\\
& - \frac{\alpha}{2}\int_\Sigma B\left(\psi(J\dot J_1) \wedge \dot \psi_2 \right) - \frac{\alpha}{2}\int_\Sigma B\left(\dot \psi_1 \wedge \psi(J\dot J_2)\right)\\
& - \frac{\alpha}{2}\int_\Sigma B\left(\psi(\dot J_1) \wedge \psi(\dot J_2)\right),\\
\mathbf{g}^{\mathbb{I}}_{\alpha,\varepsilon}(v,v) & = \frac{\varepsilon}{2}\int_{\Sigma}\tr(\dot J_1\dot J_2) \omega \\
& + \alpha  \int_\Sigma B(a \wedge J a) + \alpha \int_\Sigma B\left(\left(\dot \psi + \tfrac{\alpha}{2}\psi(J \dot J)\right) \wedge J\left(\dot \psi_2 + \tfrac{\alpha}{2}\psi(J \dot J_2)\right)\right)\\
& - \frac{\alpha}{4}\int_\Sigma B\left(\psi(\dot J) \wedge J \psi(\dot J)\right).
\end{split}
\end{equation}
\end{corollary}

%%%%%%%%%%%%%%%%%%%%%%%%%%%%%%%%%%%%%%%%%%%%%%%%%%%%%%%%%%%%%%%%%%%%
\subsection{Comparison with $\boldsymbol{\mathcal{U}^{Higgs}(G)}$ and existence}
\label{section-existenceHiggs}
%%%%%%%%%%%%%%%%%%%%%%%%%%%%%%%%%%%%%%%%%%%%%%%%%%%%%%%%%%%%%%%%%%%%

In this section we establish a comparison between the moduli space $\mathcal{U}^{Hit}(G)$ and the universal moduli space of $G$-Higgs bundles $\mathcal{U}^{Higgs}(G)$ constructed in Section \ref{ssec:UHiggs}. As we will see, Theorem \ref{thm:metricI} induces a natural holomorphic map 
$$
\mathcal{U}^{Hit}(G) \supset \mathcal{U}^*_{\alpha,\varepsilon} \lra\mathcal{U}^{Higgs}(G),
$$
and hence a holomorphic map into Teichm\"uller space $\cT$. 
%TO PROVE THE FOLLOWING I NEED THAT THE METRIC IS NON-DEGENERATE ALONG THE FIBRES, BUT THIS IS DIFFICULT: In particular, by general principles, Corollary \ref{cor:omegaexp} implies the existence of a K\"ahler Ehresmann connection on
%$$
%\Uh(G)_\alpha^\varepsilon \supset \mathcal{U}^* \to \cT
%$$
%(at least in the case $\varepsilon = -1$). 
We will also prove that for genus of the surface $\Sigma$ bigger than zero, the open $\mathcal{U}^* \subset \mathcal{U}^{Hit}(G)$ is non-empty for sufficiently small values of $\alpha$. Our analysis is very similar to that in Section \ref{section-existence}, and hence we omit the details.

Consider a point $[(J,A,\psi)] \in \mathcal{U}^{Higgs}(G)$, regarded as the $\widetilde{\GGG}$-orbit of $(J,A,\psi) \in \XXX$ solving the equations
\begin{equation}\label{eq:GGHiggsApsi}
\begin{split}
\dbar_{J,A}\psi^{1,0_J} = 0.
\end{split}
\end{equation}
The tangent space $T_{[(J,A,\psi)]} \mathcal{U}^{Higgs}(G)$ can be formally identified with the cohomology of the complex of linear differential operators
\begin{equation}\label{eq:TJDU}
\begin{tikzcd}
(\mathcal{C}_0^*)\colon & 0 \ar[r] & \mathcal{C}^0 \ar[r,"{\mathbf{P}^c_0}"] & \mathcal{C}^1 \ar[r,"{\bf L^c_0}"] & \mathcal{C}^2 \ar[r] & 0,
\end{tikzcd}
\end{equation}
where we have
$$
\mathcal{C}^0 = \Lie \widetilde{\GGG} \cong \Omega^0(T\Sigma) \oplus \Omega^0(\Sigma,E_G(\lieg)), \qquad \mathcal{C}^1 = \mathcal{S}^1, \qquad \mathcal{C}^2 = \Omega^2(\Sigma,E_G(\lieg))
$$
and
\begin{align*}
\mathbf{P}^c_0(y,u_0+ i u_1) & = - \left(L_yJ,d_Au_0 + Jd_A u_1 + i_{y}F_A,[u_0,\psi] + [u_1,J\psi] + i_y d_A \psi\right),\\
\mathbf{L}^c_0(\dot J,a,\dot \psi) & = d_A \dot \psi^{1,0_J} + [a,\psi^{1.0_J}] - \frac{i}{2} d_A(\psi (\dot J)).
\end{align*}

\begin{lemma}\label{lem:sesTWM.2}
The sequence \eqref{eq:TJDU} is an elliptic complex of degree-one linear differential operators. Consequently, the cohomology groups $H^j(\cC^*)$, with $j = 0,1,2$, are finite-dimensional.
\end{lemma}

Applying Theorem \ref{hk}, a solution $(J,A,\psi)$ of the coupled Hitchin equations \eqref{eq:cHitchinint} induces a polystable $G$-Higgs bundle $(E,\varphi) \cong (\dbar_{J,A},\psi^{1,0_J})$. This fact, jointly with the natural inclusion $\widetilde{\KKK} \subset \widetilde{\GGG}$, leads to a continuous map
\begin{equation}\label{eq:mapmoduli.1}
\begin{array}{rcl}
\mathcal{U}^{Hit}(G)&\lra&\mathcal{U}^{Higgs}(G)
\\
\left[(J,A,\psi)\right]&\longmapsto&\left[(J,A, \psi)\right].
\end{array}\end{equation}
By definition of $\mathcal{U}^{Hit}(G)$, there is a continuous diagram of moduli spaces
\begin{equation}\label{eq:modulidiag}
\begin{tikzcd}
\mathcal{U}^{Hit}(G) \ar[d]\ar[r] & \mathcal{U}^{Higgs}(G) \ar[d]\\
\mu_{\HHH}^{-1}(0)/\HHH \ar[r] & \mathcal{T}
\end{tikzcd}
\end{equation}
where $\mu_{\HHH}^{-1}(0)/\HHH$ is the moduli space of constant scalar curvature K\"ahler metrics with total volume $V$, modulo $\omega$-Hamiltonian diffeomorphisms. The fibre of $\mu_{\HHH}^{-1}(0)/\HHH \to \mathcal{T}$ can be identified with $H^1(\Sigma,\RR)$ \cite{Fj}. Building on Theorem \ref{thm:metricI}, our next goal is to prove that this induces a holomorphic map $\mathcal{U}^{Hit}(G) \supset \mathcal{U}^* \to \mathcal{U}^{Higgs}(G)$.

\begin{lemma}\label{lem:tangentmap}
Let $[(J,A,\psi)] \in \mathcal{U}^*_0$. Then, \eqref{eq:mapmoduli.1} induces a complex linear map
\begin{equation}\label{eq:TUseq}
\begin{tikzcd}
H^1(\mathcal{S}^*_0) \ar[r] & H^1(\mathcal{C}^*_0),
\end{tikzcd}
\end{equation}
where the complex structure on $H^1(\mathcal{S}^*_0)$ is the one induced by Proposition \ref{prop:lineargaugefixedI}.
\end{lemma}

To finish this section, we address the question of non-emptyness of the moduli space $\mathcal{U}^{Hit}(G)$, in genus $g(\Sigma) \geqslant 2$. For this, we fix a compact Riemann surface $X = (\Sigma,J)$ and consider a $G$-Higgs bundle $(E,\varphi)$.
In this setup, we consider the coupled equations 
\begin{equation}\label{eq:charmonicgI}
\begin{split}
F_h - [\varphi \wedge \tau_h \varphi] & = 0,\\
\varepsilon S_g & = \varepsilon \frac{2\pi \chi(\Sigma)}{V},
\end{split}
\end{equation}
for pairs $(g,h)$ consisting of a K\"ahler metric $g$ on $X$ with total volume $V$ and a reduction $h\in \Omega^0(E_G(G/K))$ of the structure group of $E$ to $K$. A solution $(g,h)$ of the equations \eqref{eq:charmonicgI} determines then a solution of \eqref{eq:cHitchinint}, given by the triple $(J,A_h,-i(\varphi-\tau_h \varphi))$. We are ready to present our last main result.

\begin{theorem}\label{thm:existenceI}
Let $(E,\varphi)$ be stable $G$-Higgs bundle over a compact Riemann surface $X = (\Sigma,J)$ with genus $g(\Sigma) \geqslant 2$. Then, for any fixed total volume $V > 0$ and parameter $\varepsilon \in \{-1,1\}$, there exists $\alpha_0 > 0$ such that for any $0 < \alpha < \alpha_0$ there exists a solution $(g_\alpha,h_\alpha)$ of the equations \eqref{eq:cHitchinint} with
$$
[(J,A_{h_\alpha},\psi_{h_\alpha})] \in \mathcal{U}^*_{\alpha,\varepsilon} \subset \mathcal{U}^{Hit}(G).
$$
Furthermore, the induced maps $\mathcal{U}^*_{\alpha,\varepsilon} \to \mathcal{U}^{Higgs}(G)$ and
\begin{equation}\label{eq:KFibModHiggs}
\mathcal{U}^*_{\alpha,\varepsilon}\lra\mu_{\HHH}^{-1}(0)/\HHH 
\end{equation}
are holomorphic, where $\mu_{\HHH}^{-1}(0)/\HHH$ is the moduli space of constant scalar curvature K\"ahler metrics with total volume $V$. 

The restriction of the symmetric tensor $\mathbf{g}^{\mathbb{I}}_{\alpha,\varepsilon}$ in Corollary \ref{cor:omegaexpI} to the fibres of \eqref{eq:KFibModHiggs} is $\alpha \mathbf{g}$, where $\mathbf{g}$ denotes the hyperk\"ahler metric on the moduli space of solutions of Hitchin's equations. Consequently, \eqref{eq:KFibModHiggs} has a natural structure of K\"ahler fibration with coupling form $\boldsymbol{\omega}_{1,0}$ (see \eqref{eq:omegagmoduliexI}) and K\"ahler Ehresmann connection.
\end{theorem}

\begin{proof}
The proof is analogous to that of Theorem~\ref{thm:existence} and is therefore omitted. The only difference is that we can control the non-degeneracy of $\mathbf{g}^{\mathbb{I}}_{\alpha,\varepsilon}$ along the fibres over the moduli space of constant scalar curvature K\"ahler metrics $\mathcal{N} = \mu_{\HHH}^{-1}(0)/\HHH$.
To see this, we observe that if $(\dot J,a,\dot \psi)$ is in the tangent to the fibre over $\mathcal{N}$, then due to the gauge fixing in Proposition \ref{prop:lineargaugefixedI}, we must have $\dot J = 0$, and the statement follows from Corollary \ref{cor:omegaexpI}. 
Hence, assuming that $\dot J = L_{J \eta_f}J$, and since the kernel of the Lichnerowicz operator is given by the Hamiltonian functions of Killing holomorphic vector field, $f$ must be constant by our assumption $g(\Sigma) \geqslant 2$.
\end{proof}

\begin{remark}
Even though the proof of Theorem \ref{thm:metricI} works in the case $g(\Sigma) = 1$, the hypothesis of existence of a stable $G$-Higgs bundle is never satisfied in this case \cite{Franco}. We thank Emilio Franco for this observation.
\end{remark}

\providecommand{\bysame}{\leavevmode\hbox to3em{\hrulefill}\thinspace}

\end{document}